\documentclass[a4paper,10pt]{article}
\usepackage[utf8]{inputenc}

\usepackage{amsmath, amssymb}

\usepackage{amsthm}

\usepackage{color}
\usepackage{graphicx}
\usepackage{mdframed}
\usepackage{comment}
\usepackage[normalem]{ulem}
\usepackage[a4paper, total={6in, 8in}]{geometry}

\usepackage{xargs}                      
\usepackage[pdftex,dvipsnames]{xcolor}  

\usepackage[showonlyrefs]{mathtools}		

\newtheorem{remark}{Remark}
\newtheorem{lemma}{Lemma}
\newtheorem{corollary}{Corollary}
\newtheorem{proposition}{Proposition}
\newtheorem{theorem}{Theorem}
\newtheorem{assumption}{Assumption}
\newtheorem{algorithm}{Algorithm}

\mathtoolsset{showonlyrefs=true}

\graphicspath{{../Figures/}}

\newcommand{\ben}{\begin{equation}}
\newcommand{\een}{\end{equation}}
\newcommand{\benn}{\begin{equation*}}
\newcommand{\eenn}{\end{equation*}}

\usepackage[colorinlistoftodos,prependcaption,textsize=tiny]{todonotes}
\newcommandx{\unsure}[2][1=]{\todo[linecolor=red,backgroundcolor=red!25,bordercolor=red,#1]{#2}}
\newcommandx{\change}[2][1=]{\todo[linecolor=blue,backgroundcolor=blue!25,bordercolor=blue,#1]{#2}}
\newcommandx{\qanda}[2][1=]{\todo[linecolor=yellow,backgroundcolor=yellow!25,bordercolor=blue,#1]{#2}}
\newcommandx{\memoToMyself}[2][1=]{\todo[linecolor=green,backgroundcolor=green!25,bordercolor=blue,#1]{#2}}
\newcommandx{\question}[2][1=]{\todo[linecolor=pink,backgroundcolor=pink!25,bordercolor=blue,#1]{#2}}
\newcommandx{\info}[2][1=]{\todo[linecolor=OliveGreen,backgroundcolor=OliveGreen!25,bordercolor=OliveGreen,#1]{#2}}
\newcommandx{\improvement}[2][1=]{\todo[linecolor=Plum,backgroundcolor=Plum!25,bordercolor=Plum,#1]{#2}}
\newcommandx{\thiswillnotshow}[2][1=]{\todo[disable,#1]{#2}}
\newcommandx{\pcomment}[2][1=]{\todo[linecolor=red,backgroundcolor=red!25,bordercolor=red,#1]{#2}}
\newcommandx{\scomment}[2][1=]{\todo[linecolor=blue,backgroundcolor=blue!25,bordercolor=blue,#1]{#2}}
\renewcommand\unsure[1]{}
\renewcommand\change[1]{}
\renewcommand\qanda[1]{}
\renewcommand\memoToMyself[1]{}
\renewcommand\question[1]{}
\renewcommand\info[1]{}
\renewcommand\improvement[1]{}
\renewcommand\thiswillnotshow[1]{}

\newcommand{\ue}{{u_{\varepsilon} } }
\newcommand{\Te}{{T_{\varepsilon} } }
\newcommand{\Tt}{{\tilde{T} } }
\newcommand{\uet}{{\tilde{u}_{\varepsilon} } }
\newcommand{\he}{{h_{\varepsilon}}}
\newcommand{\He}{{H_{\varepsilon}}}
\newcommand{\Ht}{{\tilde{H}}}
\newcommand{\Ke}{{K_{\varepsilon}}}

\newcommand{\JJ}{{\mathcal J}}
\newcommand{\nuh}{\hat{\nu} }

\newcommand{\ome}{\omega_{\varepsilon}}
\newcommand{\ut}{\tilde u}

\newcommand{\ccos}{\mbox{cos}\,}
\newcommand{\ssin}{\mbox{sin}\,}
\newcommand{\ccosh}{\mbox{cosh}\,}
\newcommand{\ssinh}{\mbox{sinh}\,}

\newcommand{\adj}{p}
\newcommand{\Adj}{P}
\newcommand{\adjz}{{\adj}_0}
\newcommand{\adje}{{{\adj}_{\varepsilon} } }
\newcommand{\adjet}{{\tilde{\adj}_{\varepsilon} } }

\newcommand{\Adjz}{\Adj_0}

\newcommand{\test}{\eta}
\newcommand{\myP}{{R_1}}
\newcommand{\myRsub}{{R_2}}

\newcommand{\cDuOld}{\underline{c}_4}
\newcommand{\cDoOld}{\overline{c}_4}
\newcommand{\cEOld}{c_5}
\newcommand{\cFOld}{c_6}
\newcommand{\cGOld}{c_7}
\newcommand{\cGtOld}{\tilde c_7}
\newcommand{\cHOld}{c_8}
\newcommand{\cHtOld}{\tilde c_8}

\newcommand{\cAu}{\underline{c}_1}
\newcommand{\cAo}{\overline{c}_1}
\newcommand{\cB}{c_2}
\newcommand{\cC}{c_3}
\newcommand{\cD}{c_4}
\newcommand{\cDt}{\tilde c_4}
\newcommand{\cE}{c_5}
\newcommand{\cEt}{\tilde c_5}

\renewcommand{\cDuOld}{\cAu}
\renewcommand{\cDoOld}{\cAo}
\renewcommand{\cEOld}{\cB}
\renewcommand{\cFOld}{\cC}
\renewcommand{\cGOld}{\cD}
\renewcommand{\cGtOld}{\cDt}
\renewcommand{\cHOld}{\cE}
\renewcommand{\cHtOld}{\cEt}

\newcommand{\ddiv}{\mbox{div}}
\newcommand{\RN}{\mathbb R^2}
\newcommand{\epsN}{\varepsilon^2}

\newcommand{\Nhalf}{1}
\newcommand{\myN}{2}
\newcommand{\Nhalfmone}{0}
\newcommand{\onemN}{-1}
\newcommand{\HH}{{\mathcal H}}

\newcommand{\lleft}{{\left(\vphantom{\frac{du}{mmy}}\right.}}
\newcommand{\rright}{{\left.\vphantom{\frac{du}{mmy}}\right)}}

\newcommand{\holdAll}{D}
\newcommand{\DT}{\mathrm{D}T}
\newcommand{\DtT}{\mathrm{D}^2T}

\newcommand{\Omfr}{\Omega_f^{ref}} 
\newcommand{\numin}{\underline{\nu}}		
\newcommand{\numinp}{\underline{\nu}}		

\newcommand{\polMat}{\mathcal P}		
\newcommand{\myTDMat}{\mathcal M}		

\newcommand{\eigvOne}{\lambda_1}
\newcommand{\eigvTwo}{\lambda_2}

\newcommand{\Hii}{H^{(2)}}

\newcommand{\uii}{u_{0}^{(2)}}
\newcommand{\Kii}{K^{(2)}}

\newcommand{\Uoii}{{U_0^{(2)}}}
\newcommand{\Voii}{{\Adj_0^{(2)}}}

\newcommand{\vii}{\adj_{0}^{(2)}}
\newcommand{\Adjzii}{{\Adj_0^{(2)}}}

\newcommand{\adjii}{\adj_{0}^{(2)}}

\title{Toplogical derivative for nonlinear magnetostatic problem}
\author{Samuel Amstutz\footnotemark[2] \quad Peter Gangl\footnotemark[3]}

\begin{document}
\maketitle

\renewcommand{\thefootnote}{\fnsymbol{footnote}}

\footnotetext[2]{Universit\'{e} d'Avignon, UFR-ip Sciences Technologie Sant\'{e}, 301 rue Baruch de Spinoza, BP 21239, 84916 Avignon Cedex 9, France ({\tt samuel.amstutz@univ-avignon.fr}).}
\footnotetext[3]{Institute of Applied Mathematics, Graz University of Technology, Steyrergasse 30/III, 8010 Graz, Austria  ({\tt gangl@math.tugraz.at}).}

\pagestyle{myheadings}
\thispagestyle{plain}
\markboth{P. Gangl and S. Amstutz}{Toplogical derivative for nonlinear magnetostatics}

	\begin{abstract}
		The topological derivative represents the sensitivity of a domain-dependent functional with respect to a local perturbation of the domain and is a valuable tool in topology optimization. Motivated by an application from electrical engineering, we derive the topological derivative for an optimization problem which is constrained by the quasilinear equation of two-dimensional magnetostatics. Here, the main ingredient is to establish a sufficiently fast decay of the variation of the direct state at scale 1 as $|x|\rightarrow \infty$. In order to apply the method in a bi-directional topology optimization algorithm, we derive both the sensitivity for introducing air inside ferromagnetic material and the sensitivity for introducing material inside an air region. We explicitly compute the arising polarization matrices and introduce a way to efficiently evaluate the obtained formulas. Finally, we employ the derived formulas in a level-set based topology optimization algorithm and apply it to the design optimization of an electric motor.
	\end{abstract}

%
	
	\def\showfigures{1}

\section{Introduction} \label{sec:Introduction}
The goal of this paper is the rigorous derivation of the topological derivative for a shape optimization problem constrained by the quasi-linear partial differential equation (PDE) of two-dimensional magnetostatics. 
This study is motivated by a concrete application from electrical engineering, namely the problem of determining an optimal design for an elctric motor. More precisely, we are interested in finding a distribution of  ferromagnetic material in a design region of an electric motor such that the motor performs as well as possible with respect to a given cost functional $\mathcal J$.
In \cite{GanglLangerLaurainMeftahiSturm2015}, the same problem was addressed by means of a shape optimization method based on shape sensitivity analysis. In order to allow for a change of the topology in the course of the optimization process, it is beneficial to include topological senstivity information of the objective function into the optimization procedure.

The \textit{topological derivative} of a domain-dependent shape functional $\mathcal J = \mathcal J(\Omega)$ indicates whether a perturbation of the domain (i.e., an introduction of a hole) around a spatial point $x_0$ would lead to an increase or decrease of the objective functional. The idea of the topological derivative was first introduced for the compliance functional in linear elasticity in \cite{EschenauerKobelevSchumacher1994, Schumacher1995} in the framework of the bubble method, where classical shape optimization methods are combined with the repeated introduction of holes (so-called ``bubbles'') at optimal positions. The mathematical concept of the topological derivative was rigorously introduced in \cite{SokolowskiZochowski1999}, see also \cite{GarreauGuillaumeMasmoudi2001} for the case of linear elasticity.  Given an open set $\Omega\subset \mathbb R^d$ with $d$ the space dimension, and a fixed bounded, smooth domain $\omega$ containing the origin, the topological derivative of a shape functional $\mathcal J = \mathcal J(\Omega)$ at a spatial point $x_0$ is defined as the quantity $G(x_0)$ satisfying a topological asymptotic expansion of the form
\begin{align}
	\mathcal J(\Omega_{\varepsilon}) - \mathcal J(\Omega) = f(\varepsilon)\, G(x_0) + o(f(\varepsilon)), \label{topAsympExp_def1}
\end{align}
where $\Omega_{\varepsilon} = \Omega \setminus \overline{\omega}_{\varepsilon}$ with $\omega_{\varepsilon} = x_0 + \varepsilon \, \omega $ denotes the perturbed domain, and $f$ is a positive first order correction function that vanishes with $\varepsilon \rightarrow 0$. We remark that, in the case where $\mathcal J$ depends on the domain $\Omega_{\varepsilon}$ via the solution of a boundary value problem on $\Omega_{\varepsilon}$, boundary conditions also have to be specified on the boundary of the hole, i.e., on $\partial \omega_{\varepsilon}$. Then, the choice of the boundary conditions on this boundary has a great influence on the resulting formula for the topological derivative $G$. In \cite{SokolowskiZochowski1999}, the authors introduced the topological derivative concept with $f(\varepsilon)$ being the volume of the ball of radius $\varepsilon$ in $\mathbb R^d$, whereas in \cite{Masmoudi1998} the form \eqref{topAsympExp_def1} was used which also allowed to deal with Dirichlet boundary conditions on the boundary of the hole, see also \cite{NovotnySokolowski2013, CeaGarreauGuillaumeMasmoudi2000}.

Topological asymptotic expansions of the form \eqref{topAsympExp_def1} have been derived for many different problems constrained by linear PDEs. We refer the interested reader to \cite{CanelasNovotnyRoche2014, ChaabaneMasmoudiMeftahi2013, FeijooNovotnyPadraTaroco2002, MasmoudiPommierSamet2005, Amstutz2006, Amstutz2005, AmstutzNovotny2010, AmstutzNovotny2011, AmstutzNovotnyVangoethem2014} as well as the monograph \cite{NovotnySokolowski2013}. Besides the field of shape and topology optimization, topological derivatives are also used in applications from mathematical imaging, such as image segmentation \cite{Hintermueller2007} or electric impedance tomography \cite{HintermuellerLaurainNovotny2012, LarrabideFeijooNovotnyTaroco2008}, or other geometric inverse problems such as the detection of obstacles, of cracks or of impurities of a material, see e.g., \cite{ChaabaneMasmoudiMeftahi2013, Hackl2006} and the references therein.

In the context of magnetostatics, introducing a hole into a domain does not 
correspond to excluding this hole from the computational domain, but rather corresponds to the presence of an inclusion of a different material, namely air. Thus, in this scenario, both the perturbed and the unperturbed configurations live on the same domain $\Omega$, and only the material coefficient of the underlying PDE constraint is perturbed. Let $u_\varepsilon$ and $u_0$ denote the solutions to the perturbed and unperturbed state equation and $\mathcal J_{\varepsilon}$ and $\mathcal J_0$ the objective functionals defined on the perturbed and unperturbed configurations, respectively. Then, the asymptotic expansion corresponding to \eqref{topAsympExp_def1} reads
\begin{align}
	\mathcal J_{\varepsilon}(u_{\varepsilon}) - \mathcal J_0(u_0) = f(\varepsilon)\, G(x_0) + o(f(\varepsilon)), \label{def:topDerivative}
\end{align}
where, again, the function $f$ is positive and tends to zero with $\varepsilon$. The quantity $G(x_0)$ is then sometimes 
referred to as the configurational derivative of the shape functional $\mathcal J$ at point $x_0$, see \cite{NovotnySokolowski2013}. This sensitivity is analyzed for a class of linear PDE constraints in \cite{Amstutz2006}. We remark that, in the limit case where the material coefficient inside the inclusion tends to zero, the classical topological derivative defined by \eqref{topAsympExp_def1} with homogeneous Neumann boundary conditions on the boundary of the hole is recovered, see \cite[Remark 5.3]{NovotnySokolowski2013}. In our case, the function $f$ in \eqref{def:topDerivative} will be given by $f(\varepsilon) = \varepsilon^d$ with $d=2$ the space dimension.
Under a slight abuse of notation, we will refer to the configurational derivative defined by \eqref{def:topDerivative} as the topological derivative.

In this paper, we derive the topological derivative for a design optimization problem that is constrained by the quasilinear equation of two-dimensional magnetostatics. As opposed to the linear case, only a few problems constrained by nonlinear PDEs have been studied in the literature. We mention the paper \cite{NovotnyFeijooTarocoMasmoudiPadra2005} where the topological derivative is estimated for the $p$-Poisson problem and the papers \cite{Amstutz2006b} and \cite{MR2541192} for the topological asymptotic expansion in the case of a semilinear elliptic PDE constraint.
In the recent work \cite{AmstutzBonnafe2015} which is based on \cite{Bonnafe2013}, the authors considered a class of quasilinear PDEs and rigorously derived the topological derivative according to \eqref{def:topDerivative}, which consists of two terms: a first term that resembles the topological derivative in the linear case, and a second term which accounts for the nonlinearity of the problem.

The quasi-linear PDE we consider in this paper does not exactly fit the framework of \cite{Bonnafe2013}. However, it is very similar and we will follow the steps taken there in order to derive the topological derivative for the electromagnetic shape optimization problem described in Section \ref{sec:ProblemDescription}.

Large parts of this paper are following the lines of \cite{AmstutzBonnafe2015, Bonnafe2013}. Here, we want to give a brief overview over the main differences to the results obtained there.
The main technical difference of the considered problems can be seen from the definition of the perturbed operator $T_{\varepsilon}$ given in Section \ref{sec_pertStateEqn}. In this paper, we consider the perturbation of a nonlinear subdomain by an inclusion of linear material or the other way around,
\begin{align*}
	T_{\varepsilon}(x,W)= \begin{cases}
								\nu_0 W & \mbox{in } \omega_{\varepsilon},\\
								T(W) & \mbox{in } \holdAll \setminus \omega_{\varepsilon},
							\end{cases}  \quad \quad
	T_{\varepsilon}^{(2)}(x,W)= \begin{cases}
									T(W) & \mbox{in } \omega_{\varepsilon},\\
									\nu_0 W  & \mbox{in } \holdAll \setminus \omega_{\varepsilon},
								\end{cases}
\end{align*}
with a nonlinear operator $T:\mathbb R^2 \rightarrow \mathbb R^2$ 
whereas in \cite{AmstutzBonnafe2015, Bonnafe2013}, the authors consider the same nonlinear function multiplied by a different constant factor inside and outside the inclusion. In our notation, this would correspond to 
\begin{align*}
	T_{\varepsilon}(x,W)&= \begin{cases}
								\gamma_1 T_a(W) & \mbox{in } \omega_{\varepsilon},\\
								\gamma_0 T_a(W) & \mbox{in } \holdAll \setminus \omega_{\varepsilon},
							\end{cases} 
\end{align*}
where a different operator $T_a$ is used (a regularized version of the $p$-Laplace operator). On the one hand, many of the steps taken for the derivation of the topological derivative can be used analogously in our context. The function space setting and the majority of the estimations even simplify here since all involved quantities are defined in the Hilbert spaces $H^1(\holdAll)$ and $\mathcal H(\RN)$ rather than in the Sobolev space $W^{1,p}(\holdAll)$ or the corresponding weighted Sobolev space over $\RN$. On the other hand, especially the proof of Theorem \ref{theo:asymp_H}, which is based on Propositions \ref{prop:supersol1} and \ref{prop:subsol}, required some additional effort.

Furthermore, the work presented in this paper extends the results of \cite{AmstutzBonnafe2015, Bonnafe2013} in several directions. Our work is motivated by a concrete application from electrical engineering. Therefore, our focus is not only on the rigorous theoretical derivation of the correct formula for the topological derivative, but also on the practical applicability of this formula. In order to be able to use the derived formula for computational shape and topology optimization, we have to consider the following additional aspects:
\begin{itemize}
	\item We compute both sensitivities $G^{f \rightarrow air}$ (see Section \ref{sec:TopAsympExp_CaseI}) and $G^{air\rightarrow f}$ (see Section \ref{sec:CaseII}) in order to be able to apply a bi-directional optimization algorithms which is capable of both introducing and removing ferromagnetic material. Note that, in \cite{AmstutzBonnafe2015, Bonnafe2013}, the derivation of the topological derivative in the reverse scenario cannot be achieved by simply exchanging the values for $\gamma_0$ and $\gamma_1$ since the result that corresponds to our Theorem \ref{theo:asymp_H} assumes that $\gamma_1 < \gamma_0$.
	\item We derive explicit formulas for the matrices $\myTDMat$, $\myTDMat^{(2)}$, see \eqref{TDMat_CaseI} and \eqref{TDMat_CaseII}.
	\item It is a priori not clear, how the new term $J_2$ defined in \eqref{def_J2} and the corresponding term of $G^{air \rightarrow f}$ in \eqref{G2_final}, which account for the nonlinearity of the problem, can be computed numerically in an efficient way. In Section \ref{Sec:Numerics}, we show a way to efficiently evaluate these terms by precomputing values in an off-line stage and using interpolation during the optimization process.
\end{itemize}

The rest of this paper is organized as follows: In Section \ref{sec:ProblemDescription}, we introduce the problem from electrical engineering that serves as a motivation for our study. We collect some mathematical preliminaries of our problem in Section \ref{sec:Preliminaries} before deriving the topological asymptotic expansion for two different cases in Sections \ref{sec:TopAsympExp_CaseI} and \ref{sec:CaseII}. In Section \ref{sec:TopAsympExp_CaseI} we consider the case where an inclusion of air is introduced into a domain of ferromagnetic material, whereas Section \ref{sec:CaseII} deals with the reverse scenario. In Section \ref{sec_Appendix}, we derive the explicit formulas of the matrices $\myTDMat$, $\myTDMat^{(2)}$ appearing in the final formulas, and in Section \ref{Sec:Numerics}, the numerical evaluation of the derived formulas is considered. Finally, we show the application of an optimization algorithm based on the topological derivative to our model problem in Section \ref{sec:applicTDNL}.

\section{Problem description} \label{sec:ProblemDescription}
We consider the design optimization of an electric motor which consists of various components as depicted in Figure \ref{fig:elMotor}. Let $\holdAll \subset \mathbb R^2$ denote the whole computational domain and let $\Omega_f^{ref}$ be the ferromagnetic reference domain, which is the brown area in the left picture of Figure \ref{fig:elMotor}. We denote its complement by $\Omega_{air}^{ref}$, i.e., $\Omega_{air}^{ref} = \holdAll \setminus \overline{\Omega_{f}^{ref}}$. This subdomain also contains the coil areas $\Omega_c$, the magnet areas $\Omega_{ma}$ as well as the thin air gap region $\Omega_g$ between the rotor and the stator of the motor. Let $\Omega^d \subset \Omega_f^{ref}$ denote the design subdomain, which is the union of the highlighted regions in the right picture of Figure \ref{fig:elMotor}. We are interested in the optimal distribution of ferromagnetic material and air regions in $\Omega^d$ and denote the subdomain of $\Omega^d$ that is currently occupied with ferromagnetic material by $\Omega$ (the unknown set). 
For any given configuration of ferromagnetic material inside $\Omega^d$, the set of all points that are occupied with ferromagnetic material is then given by 
\begin{align}
	\overline{\Omega}_f := \overline{\left(\Omega_f^{ref} \setminus \overline{\Omega^d} \right)} \cup \overline{\Omega}. \label{def_Omega_f}
\end{align}
Then, introducing $\Omega_{air} = \holdAll \setminus \overline{\Omega_f}$, we always have that $\overline{\holdAll} = \overline{\Omega}_f \cup \overline{\Omega}_{air}$.

Our goal is to find a set $\Omega \subset \Omega^d$ such that a given domain-dependent shape functional $\mathcal J$ is minimized. In the case of electric motors, this objective function $\mathcal J$ is generally supported only in the air gap $\Omega_g \subset \Omega_{air}^{ref}$. Therefore, a perturbation of the material coefficient inside the design domain $\Omega^d$ will not directly affect the functional and we assume that the functional for the perturbed and the unperturbed configuration coincide, i.e., $\mathcal J_{\varepsilon} = \mathcal J_0 = \mathcal J$ in the expansion \eqref{def:topDerivative}. The functional depends on the configuration of the design subdomain $\Omega^d$ via the solution $u$ of the state equation, $\mathcal J = \mathcal J(u)$.

The optimization problem we consider reads as follows:
	\begin{align}
		\underset{\Omega \in \mathcal A}{\mbox{min } } \mathcal J(u) \label{minJ} 
	\end{align}
	\begin{align}
		\begin{aligned} \label{PDEconstraintUnperturbed}
		\mbox{s.t.}\quad -\mbox{div}(\nu_{\Omega}(x,|\nabla u|) \nabla u ) &= F &&\mbox{in } \holdAll, \\
		u &= 0 &&\mbox{on } \partial \holdAll, \\
		[u] &= 0 &&\mbox{on } \Gamma_f, \\
		[\nu_{\Omega} \nabla u \cdot n] &= 0 &&\mbox{on } \Gamma_f, \\
		\end{aligned}
	\end{align}
	where $\mathcal A$ denotes a set of admissible shapes, $\Gamma_f:=\partial \Omega_f$ denotes the interface between ferromagnetic subdomains and air regions, and $[\cdot]$ denotes the jump across the interface.
Here, the magnetic reluctivity $\nu_{\Omega}$ depends on the ferromagnetic subdomain $\Omega_f \subset \holdAll$, which is related to the design variable $\Omega$ by \eqref{def_Omega_f}, and is defined as
\begin{align} \label{def_nuOmega}
	\begin{aligned}
	\nu_{\Omega}(x,|\nabla u|) &= \chi_{\Omega_{f}}(x) \, \hat{\nu}(|\nabla u|) &+& \chi_{\Omega_{air}}(x) \, \nu_0, \\
    &= (\chi_{ \Omega_f^{ref} \setminus \Omega^d}(x) + \chi_{\Omega}(x) ) \, \hat{\nu}(|\nabla u|) &+& (\chi_{D \setminus \Omega_f^{ref}}(x) + \chi_{\Omega^d \setminus \Omega}(x)) \, \nu_0,
    \end{aligned}
\end{align}
where $\chi_S$ denotes the characteristic function of a set $S$ and $\Omega_{air}$ and $\Omega_f$ are as defined above. Furthermore, $\nu_0 = 10^7/(4 \pi)=const>0$ denotes the reluctivity of air and $\nuh$ is a nonlinear function, which, due to physical properties, satisfies the following assumption (cf. \cite{Pechstein2004}):
\begin{assumption} \label{assump_BH}
The function $\nuh: \mathbb R_0^+ \rightarrow \mathbb R_0^+$ is continuously differentiable and there exists $\numin>0$ such that we have for all $s \in \mathbb R_0^+$,
  \begin{subequations}\label{physProperties}
      \begin{align} 
              \numin \leq \nuh(s) &\leq \nu_0, \label{propNuBounded}\\
              \numinp \leq \left( \nuh(s) \, s \right)' &\leq \nu_0. \label{propNuPrimeBounded}
      \end{align}
  \end{subequations}
\end{assumption}

The right hand side $F \in H^{-1}(\holdAll)$ comprises the sources given by the permanent magnets and by the electric current induced in the coil regions of the motor, see Figure \ref{fig:elMotor}. For $\eta \in H^1(\holdAll)$, it is defined as
\begin{align} \label{eq_magnetostatic_weak_rhs}
	\langle F, \eta \rangle = \int_{\holdAll} M^{\perp} \cdot \nabla \eta + J_z \, \eta \, \mbox dx,
\end{align}
where $J_z$ represents the third component of the electric current that is impressed in the coils and vanishes outside the coil areas, and
$M^{\perp} = (-M_2, M_1)^T$ is the perpendicular of the permanent magnetization $M = (M_1, M_2)^T$ in the magnets, which likewise vanishes outside the magnet areas. 
The PDE constraint \eqref{PDEconstraintUnperturbed} can be rewritten in operator form as
\begin{align}  \label{eq_magnetostatic_weak}
	A_{\Omega}(u) = F
\end{align}
with the operator $A_{\Omega}: H_0^1(\holdAll) \rightarrow H^{-1}(\holdAll)$ defined by
\begin{align} \label{eq_magnetostatic_weak_lhs}
	\langle A_{\Omega}(u), \test \rangle = \int_D \nu_{\Omega}(x,|\nabla u|)\nabla u \cdot \nabla  \test \, \mbox dx
\end{align}
for $u, \test \in H_0^1(\holdAll)$ and $F$ as in \eqref{eq_magnetostatic_weak_rhs}. Existence and uniqueness of a solution to the boundary value problem \eqref{eq_magnetostatic_weak} can be shown by the theorem of Zarantonello \cite[Thrm. 25.B.]{Zeidler1990} since properties \eqref{propNuBounded} and \eqref{propNuPrimeBounded} yield the strong monotonicity and Lipschitz continuity of the operator $A_\Omega$, see e.g., \cite{Gangl2017, Heise1994, Pechstein2004}.

\begin{figure} 
	\begin{minipage}{7cm}
		\begin{tabular}{c}    
			\includegraphics[scale=0.4, trim=300 0 300 0, clip=true]{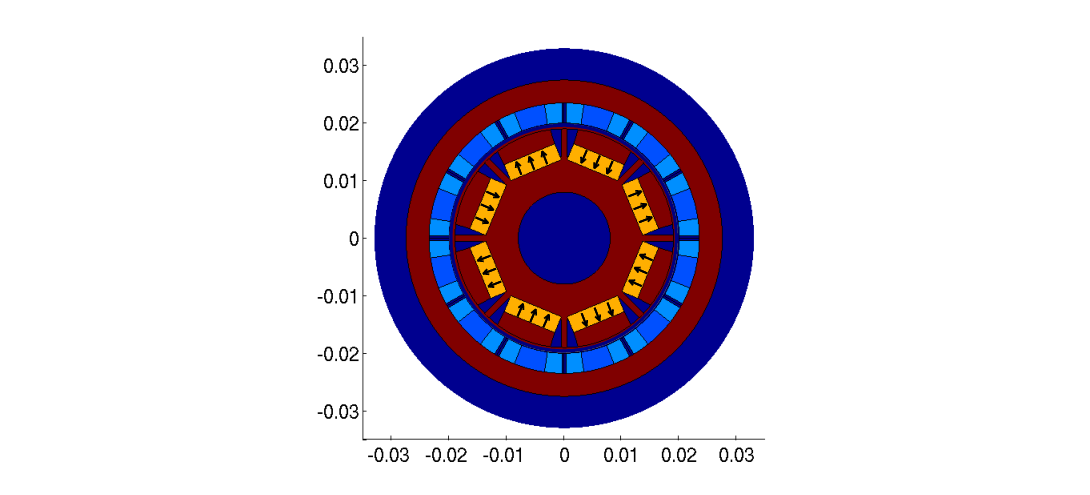}
		\end{tabular}
	\end{minipage}
	\hspace{1.2cm}
	\begin{minipage}{7cm}
		\begin{tabular}{c}    
			\hspace{-15mm}\includegraphics[scale=0.4, trim=300 0 300 0, clip=true]{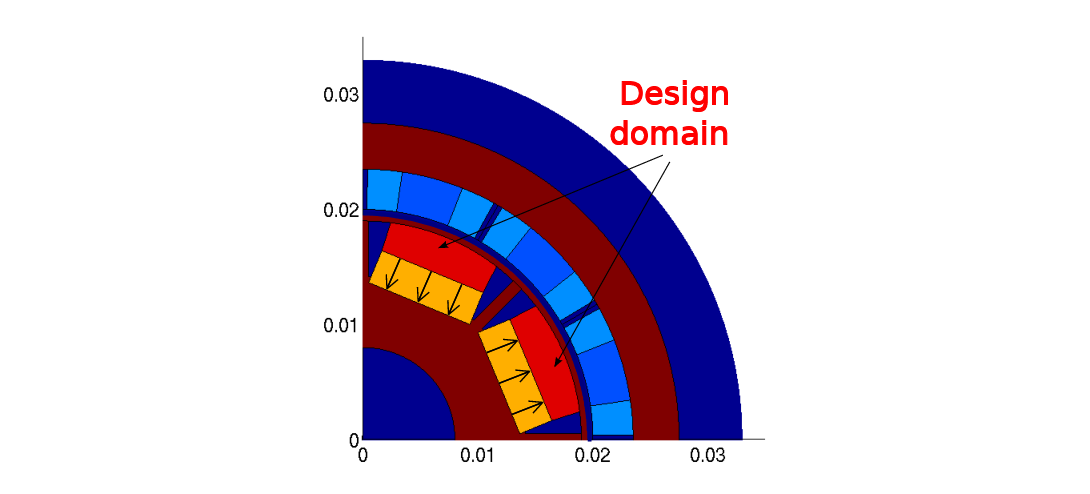}
		\end{tabular}
	\end{minipage}
	\caption{Left: Computational domain representing electric motor with different subdomains. Right: Zoom on upper left quarter with design regions $\Omega^d$ highlighted.}
	\label{fig:elMotor}	
\end{figure}


\section{Preliminaries} \label{sec:Preliminaries}

We aim at solving problem \eqref{minJ}--\eqref{PDEconstraintUnperturbed} by means of the topological derivative introduced in \eqref{def:topDerivative}. 
It is important to note that the topological derivative for introducing air in ferromagnetic material is different from that for introducing ferromagnetic material in a domain of air. Therefore, we distinguish between the following two cases:
\begin{enumerate}
	\item \textbf{Case I:} An \textit{inclusion of air} is introduced inside an area that is occupied with \textit{ferromagnetic material}, see Figure~\ref{fig:OmegaPertUnpert_caseI}.
	\item \textbf{Case II:} An \textit{inclusion of ferromagnetic material} is introduced inside an area that is occupied with \textit{air}, see Figure~\ref{fig:OmegaPertUnpert_caseII}.
\end{enumerate}
In order to distinguish these two sensitivities, we denote the topological derivative in Case I by $G^{f\rightarrow air}$ and in Case II by $G^{air\rightarrow f}$. It is important to have access to both these sensitivities for employing bidirectional optimization algorithms which are capable of both introducing and removing material at the most favorable positions. In \cite{Amstutz2006,GanglLanger2014} it is shown that, in the case of a linear state equation \eqref{PDEconstraintUnperturbed}, the two sensitivities $G^{f\rightarrow air}$ and $G^{air\rightarrow f}$ differ only by a constant factor. In the case introduced in Section \ref{sec:ProblemDescription} however, where the nonlinear material behavior of ferromagnetic material is accounted for, the two topological derivatives must be derived individually. We will rigorously derive $G^{f\rightarrow air}$ for Case I in Section~\ref{sec:TopAsympExp_CaseI} and comment on Case II in Section~\ref{sec:CaseII}.

Throughout this paper, for sake of more compact presentation, we will drop the differential $\mbox d x$ in the volume integrals whenever there is no danger of confusion.

\subsection{Notation} \label{subsec:notation}
For sake of better readability we introduce the operator $T: \RN \rightarrow \RN$ together with its Jacobian,
\begin{align}
	T(W) :=& \; \nuh(|W|)W \label{def_T} \\
	\DT(W) =& \begin{cases}
					\nuh(|W|)I + \frac{\nuh'(|W|)}{|W|}W \otimes W, &  W \neq (0,0)^\top \\
					\nuh(|W|) I&  W = (0,0)^\top ,
					\end{cases}\label{def_DT}
\end{align}
where $W\in \RN$, $I$ denotes the identity matrix in $\RN$ and $\otimes$ the outer product between two column vectors, $a \otimes b := a\, b^\top$, for $a,b\in \RN$. 
Note that $\DT$ is continuous also in $W=0$. Let further 
\begin{align} \label{def_T_Omega}
	T_{\Omega}(x,W) :&= \nu_{\Omega}(x,W)\,W = \chi_{\Omega_f}(x) \, T(W) + \chi_{D \setminus \Omega_f} (x) \,  \nu_0 W \\
	&= \left(\chi_{\Omfr \setminus \Omega^d}(x) +\chi_{\Omega}(x) \right) \,  T(W) + \left(\chi_{\holdAll \setminus \Omfr }(x) + \chi_{\Omega^d \setminus \Omega }(x)\right)\, \nu_0 W, \nonumber
\end{align}
such that $\langle A_{\Omega}(u), \test \rangle = \int_D T_{\Omega}(x,\nabla u) \cdot \nabla  \test $ for $u, \test \in H_0^1(\holdAll)$, and note that
\begin{align*}
	\DT_{\Omega}(x,W) = \chi_{\Omega_f}(x) \, \DT(W) + \chi_{D \setminus \Omega_f}(x) \,  \nu_0 I.
\end{align*}
The Fr\'echet derivative of the operator $A_{\Omega}: H_0^1(\holdAll) \rightarrow H^{-1}(\holdAll)$ is then given by
\begin{align}
	A_{\Omega}' : H_0^1(\holdAll) &\rightarrow \mathcal L( H_0^1(\holdAll), H^{-1}(\holdAll) ), \nonumber \\
	\langle A_{\Omega}'(u)  w, \test \rangle &= \int_{\holdAll} \DT_{\Omega}(x,\nabla u) \nabla w \cdot \nabla \test , \label{def_AOmegaPrime}
\end{align}
where $u,w, \test \in H_0^1(\holdAll)$. For $W = (w_1, w_2)^\top\in \RN$, $W\neq (0,0)^\top$, let $\theta_W$ be the angle between $W$ and the $x_1$-axis such that $W=|W|(\cos\theta_W,\sin\theta_W)$,
and denote $R_{\theta_W}$ the counter-clockwise rotation matrix around an angle $\theta_W$, i.e., $$R_{\theta_W} = \left[ \begin{array}{cc} \ccos \theta_W & -\ssin \theta_W \\ \ssin \theta_W & \ccos \theta_W \end{array} \right].$$ 
Denoting $\lambda_1(|W|) := \nuh(|W|)$ and $\lambda_2(|W|) := \nuh(|W|) + \nuh'(|W|)|W|$, it holds
	\begin{align} 	\label{DTW_eigenvalues}
	 \DT(W) = R_{\theta_W} \left[ \begin{array}{cc} \lambda_2(|W|) & 0 \\ 0 & \lambda_1(|W|) \end{array} \right] R_{\theta_W}^\top
	\end{align}
for all $W \in \RN$. 
Note that $\DT(W)$ is symmetric.

 For $V, W \in \RN$, we introduce the operator
\begin{align}
	S_W(V) := T(W+V) - T(W) - \DT(W)V. \label{def_S}
\end{align}
Note that, in the case of a linear state equation \eqref{PDEconstraintUnperturbed}, the operator $T$ is linear and thus $S_W(V)$ vanishes for all $V,W \in \RN$. The operator defined in \eqref{def_S} can be seen as a characterization of the nonlinearity of the problem.

\subsection{Simplified model problem} \label{sec_simplSetting}
In order to alleviate some calculations, we introduce a simplified model of the PDE constraint \eqref{PDEconstraintUnperturbed}. The model we introduce here, is meant for Case I. The analogous simplified model for Case II will be introduced in the beginning of Section \ref{sec:CaseII}.

The simplification consists in the fact that, in the unperturbed configuration, we assume the material coefficient $\nu$ to be homogeneous in the entire bounded domain $\holdAll$. In the notation of Section \ref{sec:ProblemDescription}, we assume that 
$\Omfr = \holdAll$
and, in the unperturbed case, 
$\Omega = \Omega^d$.
Then, the unperturbed state equation \eqref{PDEconstraintUnperturbed} simplifies to
\begin{align} \label{weakUnpert}
	&\mbox{Find } u_0 \in H^1_0(\holdAll) \mbox{ such that } \int_{\holdAll} T(\nabla u_0) \cdot \nabla \eta  = \langle F, \eta \rangle \quad \forall \eta \in H^1_0(\holdAll),
\end{align}
as can be seen from the definitions of the operator $A_{\Omega}$ \eqref{eq_magnetostatic_weak_lhs} and the reluctivity function $\nu_{\Omega}$ \eqref{def_nuOmega}, as well as the definition of the operator $T$ \eqref{def_T}. Here, $F\in H^{-1}(\holdAll)$ is as in \eqref{eq_magnetostatic_weak_rhs} and represents the sources due to the electric currents in the coil areas of the motor and the permanent magnetization in the magnets.

We will assume this simplified setting for the rest of this section and derive the formula for the topological derivative under these assumptions.
\begin{remark}
	The reason why we have to make this simplification will come clear in the proofs of Proposition \ref{proposition4412} and Lemma \ref{lemma455}. We remark that the topological derivative denotes the sensitivity of the objective function with respect to a perturbation inside an inclusion whose radius tends to zero. Therefore, the material coefficients ``far away'' from the point of perturbation, e.g., outside the design subdomain $\Omega^d$ when the point of perturbation is inside $\Omega^d$, should not influence the formula for the sensitivity and it is justified to use the same formula also for the realistic setting introduced in Section \ref{sec:ProblemDescription}. Note that, for all numerical computations, the realistic state equation \eqref{PDEconstraintUnperturbed} was solved.
\end{remark}
\subsection{Perturbed state equation} \label{sec_pertStateEqn}
We are interested in the sensitivity of the objective functional $\JJ$ with respect to a local perturbation of the material coefficient around a fixed point $x_0$ in the design subdomain $\Omega^d$. For that purpose, we introduce a perturbed version of the simplified state equation \eqref{weakUnpert}.

Let $\mbox{supp}(F)$ denote the support of the distribution $F \in H^{-1}(\holdAll)$. We assume that $\mbox{supp}(F)$ is compactly contained in $\holdAll$, $\mbox{supp}(F) \subset \subset \holdAll$, and that the design subdomain $\Omega^d$ is open and compactly contained in $\holdAll \setminus \mbox{supp}(F)$, $$\Omega^d \subset \subset D \setminus \mbox{supp}(F).$$
Let $x_0$ be a fixed point in $\Omega^d$. Let furthermore $\omega\subset \RN$ be a bounded open domain with $C^2$ boundary which contains the origin, and let $\omega_{\varepsilon} = x_0 + \varepsilon\, \omega$ represent the inclusion of different material in the physical domain. For simplicity and without loss of generality, we assume that $x_0 = (0,0)^\top$. Furthermore, let $0 < \rho < R$ and $\lambda \in (0,1]$ such that 
\begin{equation} \label{balls_rho_R}
	\lambda \, \omega \subset \subset B(0,\rho) \subset B(0,R) \subset \subset \holdAll \setminus \mbox{supp}(F).
\end{equation}
Note that such a choice of $\lambda$, $\rho$, $R$ is always possible if $x_0\in\Omega^d$.

Recall the notation introduced in Section \ref{sec:ProblemDescription}. In the perturbed configuration, the inclusion $\ome$ of radius $\varepsilon$ is occupied by air. Therefore, we have $\Omega = \Omega^d \setminus \ome$, and, according to \eqref{def_Omega_f}, $\Omega_f = \Omfr \setminus \Omega^d \cup \Omega = \holdAll \setminus \ome$ the set of points occupied by ferromagnetic material, see Figure~\ref{fig:OmegaPertUnpert_caseI}. Here we used that, in the simplified setting introduced in Section \ref{sec_simplSetting}, $\Omfr = \holdAll$.
We define the operator
\begin{align}
	T_{\varepsilon}(x,W) &:= \chi_{D\setminus \omega_{\varepsilon} }(x) T(W) + \chi_{\omega_{\varepsilon}}(x) \nu_0 W, \label{def_Teps_I} 
\end{align}
for $\varepsilon>0$, $x \in  \holdAll$ and $W \in \RN$ with its Jacobian given by
\begin{align*}
	\DT_{\varepsilon}(x, W) &= \chi_{D\setminus \omega_{\varepsilon} }(x) \DT(W) + \chi_{\omega_{\varepsilon}}(x) \nu_0 I.
\end{align*}
Note that, given the special setting introduced above, $ T_{\varepsilon}(x,W) = T_{\Omega^d \setminus \ome}(x,W)$ according to \eqref{def_T_Omega}.

Thus, in the simplified setting introduced in Section \ref{sec_simplSetting}, the perturbed state equation reads
\begin{align}\label{weakPert}
	&\mbox{Find } u_{\varepsilon} \in H^1_0(\holdAll) \mbox{ such that } \int_{\holdAll} T_{\varepsilon}(x,\nabla u_{\varepsilon})\cdot \nabla \eta  = \langle F,\eta \rangle \quad \forall  \eta \in H^1_0(\holdAll).
\end{align}
For $\varepsilon>0$ and $\ome = x_0 +\varepsilon \omega$ as in Section \ref{sec_pertStateEqn}, we define
\begin{align}
	S^{\varepsilon}_{W}(x, V) &:= \chi_{D \setminus \omega_{\varepsilon}}(x) S_W(V). \label{def_Seps_I}
\end{align}
Moreover, for $\varepsilon>0$, we define the scaled version of the domain $\holdAll$ as
\begin{align*}
	\holdAll / \varepsilon = \lbrace y = x / \varepsilon | x \in \holdAll \rbrace.
\end{align*}

\subsection{Expansion of cost functional} \label{sec:TopAsyExp_costFunc_caseI}
We assume that the functional to be minimized is of the following form: 
\begin{assumption} \label{assump_J}
	For $\varepsilon \geq 0$ small enough, let $\mathcal J:H^1_0(\holdAll) \rightarrow \mathbb R$ such that
	\begin{align}
		\JJ(\ue) = \JJ(u_0) + \langle \tilde G, \ue-u_0 \rangle + \delta_J \epsN + R(\varepsilon) \label{expansion_J}
	\end{align}
	where
	\begin{enumerate}
		\item $\tilde G$ denotes a bounded linear form on $H^1_0(\holdAll)$
		\item $\delta_J \in  \mathbb R$
		\item the remainder $R(\varepsilon)$ is
			of the form
				\begin{align}
					R(\varepsilon) = O \left( \int_{\holdAll \setminus B(0, \hat \alpha \varepsilon^{\tilde r} )} |\nabla(\ue-u_0)|^2 \right) \label{functional_remainder_b}
				\end{align}
				for a given $\hat \alpha >0$ and a given $\tilde r \in (0,1)$.
	\end{enumerate}
\end{assumption}

\subsection{Requirements} \label{sec:requirements}
In addition to Assumption \ref{assump_BH}, we have to make further assumptions on the nonlinear function $\nuh$ representing the magnetic reluctivity in the ferromagnetic subdomains.
\begin{assumption} \label{assump_c78}
	We assume that the nonlinear magnetic reluctivity function $\nuh$ satisfies the following:
	\begin{enumerate} 
		\item $\nuh \in C^3(\mathbb R_0^+)$. \label{nu_C3}
		\item There exists $\tilde c>0$ such that $\frac{|\nuh'(s)|}{s} \leq \tilde c$ for all $s\geq 0$. \label{nuhPrime0Zero}
		\item There exist non-negative constants $\cGtOld, \cHtOld, \tilde c', \tilde c''$ such that it holds
		\begin{align}
			\left\lvert \left( \nuh(s)s \right)'' \, \right\rvert &\leq \cGtOld, \quad \forall s \geq 0, \label{requirement7} \\
			\left\lvert \left(\nuh(s)s\right)''' \, \right\rvert &\leq \cHtOld, \quad \forall s \geq 0, \label{requirement8}\\
			| \nuh'(s)| &\leq \tilde c', \quad \, \forall s \geq 0,\label{bound_nup}\\
			| \nuh''(s)| &\leq \tilde c'', \quad \forall s \geq 0.\label{bound_nupp}
		\end{align}	
	\end{enumerate}
\end{assumption}
\begin{assumption} \label{assump_delta}
	Let $\delta_{\nuh} :=\underset{s>0}{\mbox{inf}} (\nuh'(s) s)/\nuh (s)$. We assume that
	\begin{align*}
	 	\delta_{\nuh} > \mbox{max} \, \left\lbrace \delta_{\nuh}^{\myP}, \delta_{\nuh}^{\myRsub} \right\rbrace 
	\end{align*}
	where $\delta_{\nuh}^{\myP} = -1/3$ and $\delta_{\nuh}^{\myRsub} =-\frac{(1+k_1)^2}{(1+k_1)^2+2}$ with $k_1=  (\numin-\nu_0)/\nu_0$.

\end{assumption}
Note that the first assumption implies that $\nuh$ is Lipschitz continuous on $[0, \infty)$ and we denote the Lipschitz constant by $L_{\nuh}$. Due to physical properties, the reluctivity function $\nuh$ is once continuously differentiable. However, in our asymptotic analysis, we will make use of derivatives of order up to three and thus assume $\nuh \in C^3(\mathbb R_0^+)$. This assumption is realistic in practice, since the function $\nuh$ is not known explicitly but only approximated from measured data by smooth functions, see \cite{PechsteinJuettler2006}.
The second point of Assumption \ref{assump_c78} does not automatically follow from physical properties, but it is satisfied for the (realistic) set of data we used for all of the numerical computations, see Figure~\ref{fig:nuHat}. Note that the second point of Assumption \ref{assump_c78} implies that $\nuh'(0) = 0$.

Assumption \ref{assump_delta} is needed to show Propositions {\ref{prop:supersol1} and \ref{prop:subsol} which will then yield the asymptotic behavior of the variation of the direct state at scale 1, see Theorem \ref{theo:asymp_H}. Due to the big contrast between maximum and minimum value of the magnetic reluctivity, see Figure~\ref{fig:nuHat}(a), the value of $\delta_{\nuh}^{\myRsub}<0$ is very close to zero. This means that, in order for $\nuh$ to fulfill Assumption \ref{assump_delta}, the function $\nuh$ would have to be almost monotone.
This would rule out a big class of reluctivity functions including the data used in the numerical experiments of this paper where $\delta_{\nuh } = -0.3091$.
However, we remark that Assumption \ref{assump_delta} is only a sufficient condition for the result of Theorem \ref{theo:asymp_H} and it may very well be possible to show the result with weaker assumptions on $\delta_{\nuh}$. The relaxation of Assumption \ref{assump_delta} is subject of future investigation.
\begin{figure} 
	\begin{tabular}{cccc}    
		\hspace{-6mm}\includegraphics[trim=10mm 68mm 10mm 68mm, clip,width=.25\textwidth]{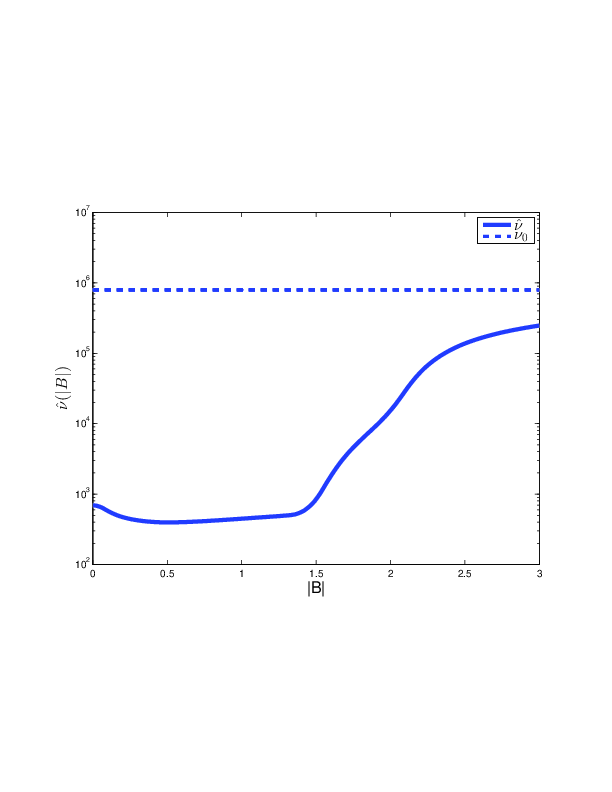} &\includegraphics[width=.25\textwidth]{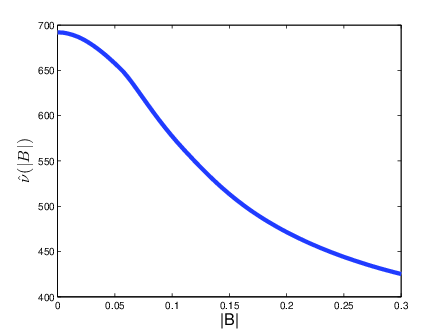} & \includegraphics[width=.25\textwidth]{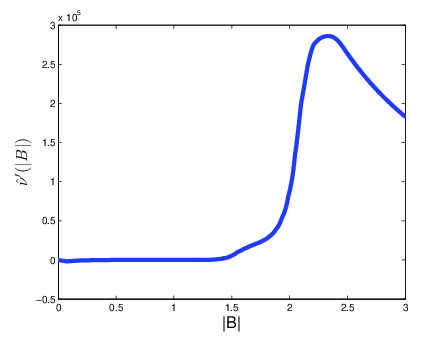} & \includegraphics[width=.25\textwidth]{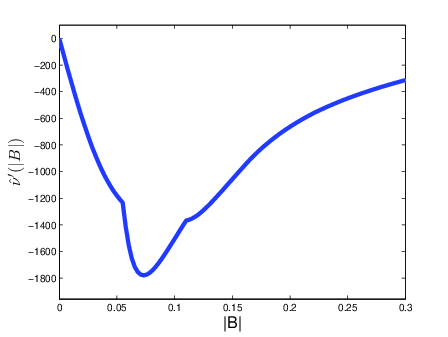}\\
			(a)&(b)&(c)&(d)
		\end{tabular}
	\caption[Magnetic reluctivity function in ferromagnetic subdomain]{(a) Magnetic reluctivity of ferromagnetic subdomain $\nuh$ in semilogarithmic scale. (b) Zoom of reluctivity $\nuh$. (c) First derivative of magnetic reluctivity, $\nuh'$. (d) Zoom of first derivative of magnetic reluctivity, $\nuh'$.}
	\label{fig:nuHat}	
\end{figure}

In Section \ref{sec_properties}, we will make use of the following estimates:
\begin{lemma}
	Let Assumption \ref{assump_c78} hold. Then there exist constants $\cGOld$, $\cHOld$ such that, for all $\varphi \in \RN$ the following estimates hold:
	\begin{align}
		4 | \nuh'(|\varphi|) | + |\nuh''(|\varphi|)| \, |\varphi| &\leq \cGOld, \label{est_c7} \\
		| \nuh'''(|\varphi|)| \, |\varphi| +9 |\nuh''(|\varphi|)| + 12 \frac{|\nuh'(|\varphi|)|}{|\varphi|} &\leq \cHOld. \label{est_c8}
	\end{align}
\end{lemma}
\begin{proof}
	Estimate \eqref{est_c7} can be easily seen using \eqref{requirement7} and \eqref{bound_nup}:
	\begin{align*}
		4 | \nuh'(|\varphi|) | + \left|\nuh''(|\varphi|)\right| \, |\varphi| &=  4 | \nuh'(|\varphi|) | + \left|\nuh''(|\varphi|) \, |\varphi| \, \right|\\
		& \leq  4 | \nuh'(|\varphi|) | + \left| \nuh''(|\varphi|) \, |\varphi| \, + 2 \nuh'(|\varphi|)\right|  + 2 |\nuh'(|\varphi|)|\\
		&\leq \cGtOld + 6 \tilde c' =: \cGOld.
	\end{align*}
	Similarly, from \eqref{requirement8} and \eqref{bound_nupp}, it follows that
	\begin{align*}
		| \nuh'''(|\varphi|)| \, |\varphi| &= \left| \nuh'''(|\varphi|) \, |\varphi| \, \right|\\
			& \leq \left| \nuh'''(|\varphi|) \, |\varphi| +3 \nuh''(|\varphi|)\right| + 3 \left|\nuh''(|\varphi|) \right|\\
			&\leq \cHtOld + 3 \tilde c'',
	\end{align*}
	which yields estimate \eqref{est_c8} with $\cHOld := \cHtOld + 12 \tilde c'' + 12 \tilde c$ by means of estimate \eqref{bound_nupp} and Assumption \ref{assump_c78}.\ref{nuhPrime0Zero}.
\end{proof}

\subsection{Properties} \label{sec_properties}
Given the physical properties of Assumption \ref{assump_BH} as well as the additional requirements of Section \ref{sec:requirements}, we can show the relations of Lemmas \ref{lemmaProperties} and \ref{lemmaProperties2}, which we will make use of throughout the next section.
\begin{lemma} \label{lemmaProperties}
	Let Assumption \ref{assump_BH} hold. Then, for the mapping $T: \RN \rightarrow \RN$ given in \eqref{def_T}, the following properties hold:
	\begin{enumerate}
		\item For all $\varphi \in \RN$, we have that
			\begin{equation*}
				\numin|\varphi| \leq |T(\varphi)| \leq \nu_0 |\varphi|.
			\end{equation*}
		\item There exist $0<\cDuOld \leq \cDoOld$ such that 
			\begin{equation}
				\cDuOld |\psi|^2 \leq \psi^\top \DT(\varphi) \psi \leq \cDoOld |\psi|^2 \quad \forall \varphi,\psi \in \RN. \label{prop4}
			\end{equation}
		\item There exists $\cEOld>0$ such that
			\begin{equation}
				(T(\varphi + \psi) - T(\varphi) ) \cdot \psi \geq \cEOld |\psi|^2 \quad \forall \varphi, \psi \in \RN. \label{prop5}
			\end{equation}
		\item There exists a Lipschitz constant $\cFOld>0$ such that
			\begin{equation}
				| T(\varphi+\psi) - T(\varphi)| \leq \cFOld |\psi|\quad \forall \varphi, \psi \in \RN.  \label{prop6}
			\end{equation}
	\end{enumerate}
\end{lemma}
\begin{proof}
	\begin{enumerate}
		\item By definition of the operator $T$ \eqref{def_T} and property \eqref{propNuBounded}, we have
			\begin{equation*}
				\numin |\varphi| \leq |T(\varphi)| = \nuh(|\varphi|)|\varphi| \leq \nu_0 |\varphi|.
			\end{equation*}
		\item Let $\varphi = (\varphi_1, \varphi_2)^\top \in \RN$. The eigenvalues of the $2 \times 2$ matrix $\DT(\varphi)$ \eqref{def_DT} are given by $\lambda_1 = \nuh(|\varphi|)$, $\lambda_2 = \nuh(|\varphi|) + \nuh'(|\varphi|)|\varphi|$. Properties \eqref{propNuBounded} and \eqref{propNuPrimeBounded} yield the claimed result with $\cDuOld = \numin$ and $\cDoOld = \nu_0$.
		\item Property \eqref{prop5} can be seen in the same way. It holds that
			\begin{equation*}
				(T(\varphi + \psi) - T(\varphi))\cdot \psi = \int_0^1 \psi^\top \DT(\varphi + t \psi) \, \psi \, \mbox d t \geq \cDuOld |\psi|^2,
			\end{equation*}
			which yields that \eqref{prop5} holds with $\cEOld = \cDuOld$.
		\item In a similar way, we obtain \eqref{prop6} with $\cFOld=\nu_0$,
			\begin{align*}
				| T(\varphi+\psi) - T(\varphi)| &= \left\lvert \int_0^1 \DT(\varphi + t \psi) \psi \mbox d t \right \rvert \leq \int_0^1 \left\lvert  \DT(\varphi + t \psi) \psi \right \rvert\mbox d t \\
				&\leq \int_0^1 \mbox{max}\, \lbrace \lambda_1(\varphi + t \psi), \lambda_2(\varphi + t \psi) \rbrace|\psi| \mbox dt \leq \nu_0 |\psi|.
			\end{align*}
			Here, we used the notation $\lambda_1(W)$, $\lambda_2(W)$ for the two eigenvalues of the matrix $DT(W)$ given in \eqref{def_DT} as well as \eqref{propNuBounded} and \eqref{propNuPrimeBounded}.
	\end{enumerate}
\end{proof} 

\begin{lemma} \label{lemmaProperties2}
	Let Assumption \ref{assump_c78} hold. Then, the mapping $T:\RN \rightarrow \RN, W \mapsto \nuh(|W|)W$ given in \eqref{def_T} has the following properties:
	\begin{enumerate}
	 	\item $T \in C^{3}(\RN)$. \label{lem:DTLipschitz}
		\item For the constant $\cGOld\geq0$ in \eqref{est_c7}, it holds
			\begin{equation}
				| S_{\varphi}(\psi_2)- S_{\varphi}(\psi_1)| \leq \cGOld |\psi_2 - \psi_1|(|\psi_1| + |\psi_2|) \quad \forall \varphi, \psi_1, \psi_2 \in \RN. \label{prop7}
			\end{equation}
		\item For the constant $\cHOld\geq0$ in \eqref{est_c8}, it holds
			\begin{equation}
				| S_{\varphi_2}(\psi)- S_{\varphi_1}(\psi)| \leq \cHOld |\varphi_2 - \varphi_1||\psi|^2  \quad \forall \varphi_1, \varphi_2 , \psi \in \RN. \label{prop8}
			\end{equation}
	\end{enumerate}
\end{lemma}
\begin{proof}
	\begin{enumerate}
		\item We consider the first, second and third derivative of $T$:
		\begin{itemize}
		 \item Recall the Jacobian of the mapping $T$ given in \eqref{def_DT},
			\begin{align*}
				\DT(\varphi) (\psi) = \begin{cases}
					\nuh(|\varphi|)\psi + \frac{\nuh'(|\varphi|)}{|\varphi|}(\varphi \otimes \varphi) \psi &  \varphi \neq (0,0)^\top, \\
					\nuh(|\varphi|) \psi&  \varphi = (0,0)^\top,
				\end{cases}
			\end{align*}
			for $\psi \in \RN$, and note that $\DT$ is continuous also in $\varphi=0$ due to
			\begin{align*}
				\underset{|\varphi|\rightarrow 0}{\mbox{lim} }\, | \DT(\varphi)(\psi) - \DT(0)(\psi) | &\leq \underset{|\varphi|\rightarrow 0}{\mbox{lim} }\,\left( | (\nuh(|\varphi|)-\nuh(0)) \psi |
				+ \left| \frac{\nuh'(|\varphi|)}{|\varphi|} (\varphi \otimes \varphi) \psi \right| 
				\right)\\
				&\leq \underset{|\varphi|\rightarrow 0}{\mbox{lim} }\,\left( | (\nuh(|\varphi|)-\nuh(0)) | \, |\psi|
				+ \left | \nuh'(|\varphi|)|\varphi| \, |\psi| \right| \right) = 0,
			\end{align*}
			because of \eqref{bound_nup} and the continuity of $\nuh$.
			\item
			For $\DtT$, we get for all $\varphi \neq 0$ and  all $\psi,\eta \in \RN$,
			\begin{align}
				\begin{aligned}
				\mathrm{D}^2 T(\varphi)(\psi,\eta) 
				=&\frac{\nuh'(|\varphi|)}{|\varphi|} \left( \varphi \otimes \eta  + (\varphi \cdot \eta) I + \eta \otimes \varphi \right) \psi\\
				&+ \left( \frac{\nuh''(|\varphi|)}{|\varphi|^2} - \frac{\nuh'(|\varphi|)}{|\varphi|^3} \right) (\varphi \cdot \eta)(\varphi \otimes \varphi) \psi. \label{D2T}
				\end{aligned}
			\end{align}
			In the point $\varphi =0$, the Fr\'echet derivative of $\DT$ is given by~${\mathrm{D}^2 T(0)(\psi,\eta) =0}$, which can be seen as follows:
			\begin{align*}
				\underset{| \eta| \rightarrow 0}{\mbox{lim}}\, \frac{| \DT(\eta)(\psi) - \DT(0)(\psi) - 0 | }{|\eta|} 
				=& \underset{| \eta| \rightarrow 0}{\mbox{lim}}\,\left| \frac{\nuh(|\eta|) - \nuh(0)}{|\eta|} \psi + \nuh'(|\eta|) \left(\frac{\eta}{|\eta|} \otimes \frac{\eta}{|\eta|}\right) \, \psi \right|\\
				\leq& \, 2\, |\nuh'(0)| |\psi| = 0,
			\end{align*}
			since $\nuh'(0) =0$ due to Assumption \ref{assump_c78}.\ref{nuhPrime0Zero}. Thus, we have, $\mathrm{D}^2 T(0)(\psi,\eta) = 0$.
			Also here, we can see the continuity in $\varphi = 0$ as, for any $\psi, \eta \in \RN$, it holds
			\begin{align*}
				\underset{|\varphi|\rightarrow 0}{\mbox{lim} }&\, | \mathrm{D}^2 T(\varphi)(\psi,\eta) - \DtT(0)(\psi, \eta) | \\
				=& \underset{|\varphi|\rightarrow 0}{\mbox{lim} }\, \left[ \nuh'(|\varphi|) \left( \frac{\varphi}{|\varphi|} \otimes \eta  + \left(\frac{\varphi}{|\varphi|} \cdot \eta\right) I + \eta \otimes \frac{\varphi}{|\varphi|} -  \left( \frac{\varphi}{|\varphi|} \cdot \eta\right)\left( \frac{\varphi}{|\varphi|} \otimes  \frac{\varphi}{|\varphi|} \right) \right) \right. \\
				& \qquad \quad + \left. \nuh''(|\varphi|) (\varphi \cdot \eta) \left(\frac{\varphi}{|\varphi|} \otimes \frac{\varphi}{|\varphi|} \right)\right] \psi = \, 0,
			\end{align*}
			where, we used that $\varphi /|\varphi|$ is a unit vector and, again, that $\nuh'(0)=0$ due to Assumption~\ref{assump_c78}.\ref{nuhPrime0Zero}, as well as \eqref{bound_nupp}.
			
			\item
			By differentiating \eqref{D2T}, we obtain for all $\varphi, \psi,\eta, \xi \in \RN$ with $\varphi\neq0$ that
			\begin{align} \label{def_D3T}
				\begin{aligned}
				\mathrm{D}^3T&(\varphi)(\psi, \eta, \xi)=\frac{1}{|\varphi|^2} \left( \nuh''(|\varphi|) - \frac{\nuh'(|\varphi|)}{|\varphi|} \right) \varphi \cdot \xi \left[ (\varphi\cdot \eta) \psi + (\varphi\cdot \psi) \eta + (\psi\cdot \eta) \varphi \right] \\
				&+ \frac{\nuh'(|\varphi|)}{|\varphi|} \left[ (\xi \cdot \eta) \psi + (\xi \cdot \psi) \eta +(\psi \cdot \eta) \xi \right]\\
				&+\frac{1}{|\varphi|^2} \left( \nuh''(|\varphi|) - \frac{\nuh'(|\varphi|)}{|\varphi|} \right) \left[ (\xi \cdot \eta)(\varphi \cdot \psi) \varphi + (\varphi \cdot \eta)(\xi \cdot \psi) \varphi + (\varphi \cdot \eta)(\varphi \cdot \psi) \xi \right] \\
				&+ \left( \frac{\nuh'''(|\varphi|)}{|\varphi|^3} - 3 \frac{\nuh''(|\varphi|)}{|\varphi|^4} + 3 \frac{\nuh'(|\varphi|)}{|\varphi|^5} \right) (\varphi \cdot \xi) (\varphi \cdot \eta)( \varphi \cdot \psi) \varphi.
				\end{aligned}
			\end{align}
			We show that, under Assumption \ref{assump_c78}.\ref{nuhPrime0Zero}, the Fr\'echet derivative of $\mathrm{D}^2 T$ at the point $\varphi = 0$ is given by
			\begin{align}
				\mathrm{D}^3 T(0)(\psi, \eta,\xi) = \nu''(0) \left( \xi \otimes \eta  + (\xi \cdot \eta) I + \eta \otimes \xi \right) \psi. \label{def_D3T0}
			\end{align}
			Exploiting that $\mathrm{D}^2 T(0)(\psi,\eta) = 0$, we get
			\begin{align*}	
				\underset{|\xi| \rightarrow 0}{\mbox{lim}}&\, \frac{|\mathrm{D}^2 T(\xi)(\psi,\eta) - \mathrm{D}^2 T(0)(\psi,\eta) - \mathrm{D}^3 T(0)(\psi, \eta,\xi)|}{|\xi|} \\
				=&\underset{|\xi| \rightarrow 0}{\mbox{lim}}\, \left| \frac{\nuh'(|\xi|)}{|\xi|} \left( \frac{\xi}{|\xi|} \otimes \eta  + (\frac{\xi}{|\xi|} \cdot \eta) I + \eta \otimes \frac{\xi}{|\xi|} \right) \psi \right. \\
				&\qquad \; +  \left( \nuh''(|\xi|) - \frac{\nuh'(|\xi|)}{|\xi|} \right) (\frac{\xi}{|\xi|} \cdot \eta)(\frac{\xi}{|\xi|} \otimes \frac{\xi}{|\xi|}) \psi\\
				& \qquad \; - \left. \nu''(0) \left( \frac{\xi}{|\xi|} \otimes \eta  + (\frac{\xi}{|\xi|} \cdot \eta) I + \eta \otimes \frac{\xi}{|\xi|} \right) \psi \right|.
			\end{align*}
			Noting that, under Assumption \ref{assump_c78}.\ref{nuhPrime0Zero}, we have $\nuh'(0) = 0$ and thus $$\underset{t \rightarrow 0}{\mbox{lim}} \, \nu'(t) / t  = \underset{t \rightarrow 0}{\mbox{lim}} \, (\nu'(t) - \nu'(0)) / (t- 0)  = \nu''(0),$$
			we see that the above expression vanishes, which proves the form \eqref{def_D3T0}.
			The continuity of $\mathrm{D}^3 T(\varphi)(\psi,\eta,\xi)$ is clear for $\varphi\neq 0$ and can be seen for the point $\varphi =0 $ noting that $\underset{t \rightarrow 0}{\mbox{lim}} \, \nu'(t) / t = \nu''(0)$ which finishes the proof of statement \ref{lem:DTLipschitz} of Lemma~\ref{lemmaProperties2}.
		\end{itemize}

		\item We follow the lines of the proof of Proposition 4.1.3(7) on page 73 of \cite{Bonnafe2013}:\\
			Since $T \in C^3(\RN)$, we can apply the fundamental theorem of calculus and get
			\begin{align*}
				S_{\varphi}&(\psi_2) - S_{\varphi}(\psi_1) = T(\varphi + \psi_2) - T(\varphi + \psi_1) - \DT(\varphi) (\psi_2 - \psi_1) \\
				&= \int_0^1 \left[ \DT(\varphi + \psi_1 + t(\psi_2 - \psi_1)) - \DT(\varphi) \right] (\psi_2 - \psi_1) \mbox d t \\
				&= \int_0^1 \int_0^1 \DtT(\varphi + s \left[ (1-t)\psi_1 + t \psi_2 \right] )  ((1-t)\psi_1 + t \psi_2, \psi_2 - \psi_1) \mbox d t\, \mbox d s.
			\end{align*}
			From \eqref{D2T}, it follows that
			\begin{align}
				|\mathrm{D}^2 T(\varphi)(\psi, \eta)| 
				&\leq 3 \frac{|\nuh'(|\varphi|)|}{|\varphi|}  |\varphi| \, |\psi| \, |\eta| + \left( \frac{|\nuh''(|\varphi|)|}{|\varphi|^2} + \frac{|\nuh'(|\varphi|)|}{|\varphi|^3} \right) |\varphi|^3 |\psi|\, |\eta| \nonumber \\
				&= \left( 4 |\nuh'(|\varphi|)|+ |\nuh''(|\varphi|)|\,  |\varphi| \right) |\psi|\, |\eta|.  \label{est_D2T}
			\end{align}
			Estimate \eqref{est_D2T}, together with requirement \eqref{est_c7}, yields that
			\begin{align*}
				|&S_{\varphi}(\psi_2) - S_{\varphi}(\psi_1)|  \\
				&\leq \int_0^1 \int_0^1 \left|\DtT(\varphi + s \left[ (1-t)\psi_1 + t \psi_2 \right] ) )\right| \, |(1-t)\psi_1 + t \psi_2| \,|\psi_2 - \psi_1| \mbox d t \mbox d s \\
				& \leq \cGOld |(1-t)\psi_1 + t \psi_2| |\psi_2 - \psi_1| \\
				&\leq \cGOld (|\psi_1| + |\psi_2|) |\psi_2 - \psi_1|.
			\end{align*}
		\item Since $T$ is three times continuously differentiable, we get by the fundamental theorem of calculus that
			\begin{align*}
				| S_{\varphi_2}&(\psi) - S_{\varphi_1}(\psi) | \\
				=& | T(\varphi_2 + \psi) - T(\varphi_2) - \DT(\varphi_2)\psi  - \left(T(\varphi_1 + \psi) - T(\varphi_1) - \DT(\varphi_1)\psi \right) | \\
				=& \left|  \int_0^1 \left( \DT(\varphi_2 + t\, \psi) - \DT(\varphi_2)   \right) \psi \, \mbox d t -  \int_0^1 \left( \DT(\varphi_1 + t\, \psi) - \DT(\varphi_1)   \right) \psi \, \mbox d t \right| \\
				=& \left|\int_0^1 \int_0^1  \mathrm{D}^2 T(\varphi_2 + s\, t\, \psi )(\psi, t \psi)     \, \mbox d s \, \mbox d t
				-\int_0^1 \int_0^1  \mathrm{D}^2 T(\varphi_1 + s\, t\, \psi )(\psi, t \psi)     \, \mbox d s \, \mbox d t
				\right| \\
				\leq& \int_0^1 \int_0^1 \left| \mathrm{D}^2 T(\varphi_2 + s\, t\, \psi )(\psi, t \psi)  -\mathrm{D}^2 T(\varphi_1 + s\, t\, \psi )(\psi, t \psi)   \right| \, \mbox d s \, \mbox d t \\
				\leq& \int_0^1 \int_0^1 \int_0^1 t | \mathrm{D}^3 T(\varphi_1 + s\, t\, \psi + r (\varphi_2 - \varphi_1) )(\psi, \psi, \varphi_2 - \varphi_1) | \, \mbox d r \, \mbox d s \, \mbox d t.
			\end{align*}
			From \eqref{def_D3T}, it is seen that
			\begin{align*} 
				\begin{aligned}
				|\mathrm{D}^3T(\varphi)(\psi,\psi, \xi)| \leq& 6 \left| \left( \nuh''(|\varphi|) - \frac{\nuh'(|\varphi|)}{|\varphi|} \right) \right| \, |\psi|^2 |\xi| + 3 \frac{|\nuh'(|\varphi|)|}{|\varphi|} |\psi|^2 |\xi|\\ 
				&+ \left|\left( \nuh'''(|\varphi|)|\varphi| - 3 \nuh''(|\varphi|) + 3 \frac{\nuh'(|\varphi|)}{|\varphi|} \right) \right| \, |\psi|^2 |\xi| \\
				\leq& \left( |\nuh'''(|\varphi|)|\,|\varphi| + 9 | \nuh''(|\varphi|)| + 12 \frac{|\nuh'(|\varphi|)|}{|\varphi|}\right)|\psi|^2 |\xi|.
				\end{aligned}
			\end{align*}
			Denoting $\varphi_{r,s}^t = \varphi_1 + s\, t\, \psi + r (\varphi_2 - \varphi_1)$ for $0\leq r,s,t \leq 1$, we get
			\begin{align*}
				&|S_{\varphi_2}(\psi) - S_{\varphi_1}(\psi)| 
				\leq \int_0^1 \int_0^1 \int_0^1 |\mathrm{D}^3T(\varphi_{r,s}^t)(\psi, \psi, \varphi_2 - \varphi_1)| \, \mbox d r \, \mbox d s\, \mbox d t \\
				&\leq \int_0^1 \int_0^1 \int_0^1  \left| |\nuh'''(|\varphi_{r,s}^t|)|\,|\varphi_{r,s}^t| + 9 |\nuh''(|\varphi_{r,s}^t|)| + 12 \frac{|\nuh'(|\varphi_{r,s}^t|)|}{|\varphi_{r,s}^t|} \right| \,|\psi|^2 |\varphi_2 - \varphi_1|  \, \mbox d r \, \mbox d s \, \mbox d t \\
				&\leq \cHOld|\psi|^2 |\varphi_2 - \varphi_1|,
			\end{align*}
			where we used \eqref{est_c8}, which holds under Assumption \ref{assump_c78}.
	\end{enumerate}
	This concludes the proof of Lemma \ref{lemmaProperties2}
\end{proof}

\begin{remark} \label{rem:propXDep}
	It is easy to see that the statements of Lemma \ref{lemmaProperties} and Lemma \ref{lemmaProperties2} also hold for the $x$-dependent operators $T_{\varepsilon}$, $\DT_{\varepsilon}$ introduced in \eqref{def_Teps_I} and the operator $S^{\varepsilon}$ defined in \eqref{def_Seps_I} for each point in their domain of definition with the same constants.
\end{remark}

\begin{remark}
	Note that relation \eqref{prop7} implies that
		\begin{align}
			|S_{\varphi}(\psi)| \leq \cGOld |\psi|^2. \label{prop7_followUp}
		\end{align}
\end{remark}


\subsection{Weighted Sobolev spaces} \label{sec_weightedSoboSpaces}
In order to analyze the asymptotic behavior of the variation of the direct and adjoint state at scale 1 in Sections \ref{sec:VarDirect} and \ref{sec:VarAdjoint}, we need to define an appropriate function space. We follow the presentation given in \cite{Bonnafe2013}. For more details on weighted Sobolev spaces, we refer the reader to \cite{AmroucheGiraultGiroire1994}.

Let the weight function $w:\RN \rightarrow \mathbb R$ be defined as
\begin{equation}
	w(x) = \frac{1}{(1+|x|^2)^{1/2}( \mbox{log}(2+|x|))}. \label{def_weight}
\end{equation}
Note that $w \in L^2(\RN)$ and $w(x)>0$ for all $x \in \RN$ with 
\begin{align*}
	\underset{x \in \RN}{\mbox{inf}} w(x) = 0 \quad \mbox{ and } \quad \underset{x \in \RN}{\mbox{sup}} w(x) < \infty.
\end{align*}
For all open $\mathcal O \subset \RN$, the space of distributions in $\mathcal O$ is denoted by $\mathcal D'(\mathcal O)$. We define the weighted Sobolev space
\begin{equation*}
	\HH^w(\mathcal O) := \left \lbrace u \in \mathcal D'(\mathcal O): w\, u \in L^2(\mathcal O), \nabla u \in L^2(\mathcal O)	\right\rbrace,
\end{equation*}
together with the inner product
\begin{align*}
		\langle u, v \rangle_{\HH^w(\mathcal O)} := \langle w \, u, w\, v \rangle_{L^2(\mathcal O)} + \langle \nabla u, \nabla v \rangle_{L^2(\mathcal O)}, \quad  u, v \in \HH^w(\mathcal O),
\end{align*}
and the norm
\begin{align*}
	\| u \|_{\HH^w(\mathcal O)} := \langle u,u \rangle_{\HH^w(\mathcal O)}^{\frac{1}{2}},  \quad u \in \HH^w(\mathcal O).
\end{align*}
The following result is shown in \cite{Bonnafe2013}.
\begin{lemma}
	The space $\HH^w(\mathcal O)$ endowed with the inner product $\langle \cdot, \cdot \rangle_{\HH^w(\mathcal O)}$ is a separable Hilbert space.
\end{lemma}

We define the weighted quotient Sobolev space
\begin{equation*}
	\HH(\RN) := \HH^w(\RN)/_{\mathbb R}
\end{equation*}
where we factor out the constants, and equip it with the quotient norm
\begin{equation*}
	\| [u] \|_{\HH(\RN)} := \underset{m \in \mathbb R}{\mbox{inf}} \| \tilde u + m \|_{\HH^w(\RN)}, \quad [u] \in \HH(\RN),
\end{equation*}
where $\tilde u \in \HH^w(\RN)$ is any element of the class $[u]$. We note that $\HH(\RN)$ is a Hilbert space because $\HH^w(\RN)$ is a Hilbert space and $\mathbb R$ is a closed subspace. For the space $\HH(\RN)$, we can state the following Poincar\'{e} inequality, which is proved in \cite{Bonnafe2013}:
\begin{lemma} \label{lem:PoincareH}
	There exists $c_P>0$ such that
	\begin{equation*}
		\| [u] \|_{\HH(\RN)} \leq c_P \| \nabla \tilde u \|_{L^2(\RN)}, \quad \forall [u] \in \HH(\RN),
	\end{equation*}
	where $\tilde u \in \HH^w(\RN)$ is any element of the class $[u]$.
\end{lemma}

For all $[u]\in \HH(\RN)$, let $\tilde u\in \HH(\RN)$ denote any element of the class $[u]$. We endow $\HH(\RN)$ with the semi-norm
\begin{equation*}
	|[u]|_{\HH(\RN)} := \| \nabla \tilde u \|_{L^2(\RN)}.
\end{equation*}
The following corollary follows directly from Lemma \ref{lem:PoincareH}.
\begin{corollary} \label{cor:CoercivityH}
	The semi-norm $|[\cdot]|_{\HH(\RN)}$ is a norm and is equivalent to the norm $\|[\cdot]\|_{\HH(\RN)}$ in~$\HH(\RN)$. 
\end{corollary}
\section{Topological asymptotic expansion: case I} \label{sec:TopAsympExp_CaseI}
In this section, we derive the topological asymptotic expansion \eqref{def:topDerivative} for the introduction of an inclusion of air, which has linear material behavior, inside ferromagnetic material, which behaves nonlinearly. 
For the reader's convenience, we moved all longer, technical proofs of this section to Section \ref{sec_proofsTDNL}.

By Assumption \ref{assump_J}, the expansion \eqref{def:topDerivative} reduces to showing that
\begin{align*}
	\langle \tilde G, \ue - u_0 \rangle = \varepsilon^2 \, G(x_0) + o(\varepsilon^2) \quad \mbox{ and } \quad R(\varepsilon) = o(\varepsilon^2)
\end{align*}
with the remainder $R(\varepsilon)$ of the form \eqref{functional_remainder_b}, where we chose $f(\varepsilon) = \varepsilon^2$. In order to show these relations, we investigate in detail the difference $\ue - u_0$, called the \textit{variation of the direct state}. After rescaling, we introduce an approximation of this variation which is independent of the small parameter~$\varepsilon$. This approximation, which we will denote by $H$, is the solution to a transmission problem on the entire plane $\RN$ and is an element of a weighted Sobolev space as introduced in Section~\ref{sec_weightedSoboSpaces}. We establish relations between this approximation $H$ and the variation $\ue - u_0$ on the domain $\holdAll$. An important ingredient for this is to show that $H$ satisfies a sufficiently fast decay towards infinity, meaning that this approximation to the difference between the perturbed and unperturbed state is small ``far away'' from the inclusion. This result, which is rather technical, is obtained in Theorem \ref{theo:asymp_H}. 
All of these steps are shown in detail in Section~\ref{sec:VarDirect}.

Similar results are needed for the \textit{variation of the adjoint state} $\adje - \adjz$ which is approximated by the $\varepsilon$-independent function $K$. Again, a sufficiently fast decay of $K$ towards infinity is important. We remark that, also in the case of a nonlinear state equation, the boundary value problem defining the adjoint state is always linear. Therefore, the treatment of the variation of the adjoint state is less technical. These steps are carried out in Section \ref{sec:VarAdjoint}.

Given the relations of Sections \ref{sec:VarDirect} and \ref{sec:VarAdjoint}, a topological asymptotic expansion of the form \eqref{def:topDerivative} is shown in Section \ref{sec_topAsympExp_I}. 

\subsection{Variation of direct state} \label{sec:VarDirect}
\begin{figure} 
	\begin{minipage}{.5\textwidth}
		\begin{tabular}{c}    
			\hspace{15mm} \includegraphics[scale=0.55]{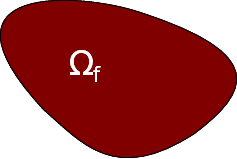}
		\end{tabular}
	\end{minipage}
	\hspace{1.2cm}
	\begin{minipage}{.5\textwidth}
		\begin{tabular}{c}    
			\hspace{-7mm}\includegraphics[scale=0.55]{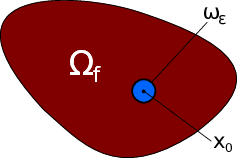}
		\end{tabular}
	\end{minipage}
	\caption[Illustration of topological derivative in Case I]{Left: Unperturbed configuration for Case I, $\holdAll = \Omega_f$. Right: Perturbed configuration for Case I, $\overline \holdAll = \overline{\Omega_f} \cup \overline{\ome}$. }
	\label{fig:OmegaPertUnpert_caseI}	
\end{figure}

\subsubsection{Regularity assumptions} \label{sec:regularityAssumptions}
In order to perform the asymptotic analysis for the derivation of the topological derivative, we need some regularity of the solution to the unperturbed state problem \eqref{weakUnpert} in a neighborhood of the point of the perturbation $x_0 \in \Omega^d$. Henceforth, we make the following assumption:

\begin{assumption} \label{assump_reg_u0}
	There exists $\beta >0$ such that
	\begin{align*}
		u_0|_{\Omega^d} \in C^{1,\beta}(\Omega^d).
	\end{align*}
\end{assumption}
\begin{remark}
    In the case of the model problem introduced in Section \ref{sec:ProblemDescription}, the right hand side is a distribution $F \in H^{-1}(\holdAll)$, which is, however, only supported outside the design area. Therefore, the assumption that the solution $u_0$ is smooth in the design area is reasonable. \unsure{talk to somebody..}
\end{remark}

\subsubsection{Step 1: variation $\ue - u_0$} \label{sec:Step1Direct_I}
Subtracting the perturbed problem \eqref{weakUnpert} from the unperturbed problem \eqref{weakPert} and defining ${\uet:= \ue - u_0 }$,
we get the boundary value problem defining the variation of the direct state at scale $\varepsilon$:
\begin{align} \label{variationDirect_scaleEps}
	\begin{aligned}
		\mbox{Find } \uet \in H^1_0(\holdAll):
		&\int_{\holdAll} \left( \Te(x,\nabla u_0 + \nabla \uet) - \Te(x,\nabla u_0) \right) \cdot \nabla \test \\
		&=- \int_{\ome}(\nu_0 - \hat{\nu}(|\nabla u_0|)) \nabla u_0 \cdot \nabla \test  \quad \forall \test \in H^1_0(\holdAll).
	\end{aligned}
\end{align}

\subsubsection{Step 2: approximation of variation $\ue - u_0$} \label{sec:approxStep1}
We approximate problem \eqref{variationDirect_scaleEps} by the same boundary value problem where we replace the function $\nabla u_0$
by its value at the point of interest $x_0$, i.e., we replace $\nabla u_0$ by the constant $U_0 := \nabla u_0(x_0)$. Note that this point evaluation makes sense due to Assumptions \ref{assump_reg_u0}. Denoting the solution to the arising boundary value problem by $\he$, we get
\begin{align} \label{variationApprox1}
	\begin{aligned}
		\mbox{Find } \he \in H^1_0(\holdAll):
		&\int_{\holdAll} \left( \Te(x,U_0 + \nabla \he)  - \Te(x,U_0) \right) \cdot \nabla \test \\
		& = - \int_{\ome}(\nu_0 - \hat{\nu}(|U_0|)) U_0 \cdot \nabla \test \quad \forall \test \in H^1_0(\holdAll).
	\end{aligned}
\end{align}
The relation between the solutions to boundary value problems \eqref{variationDirect_scaleEps} and \eqref{variationApprox1} will be investigated in Proposition \ref{proposition4413}.

\subsubsection{Step 3: change of scale} \label{sec:changeOfScale}
Next, we make another approximation to boundary value problem \eqref{variationApprox1}. First, we perform a change of scale, i.e., we go over from the domain $\holdAll$ with the inclusion $\ome$ of size $\varepsilon$ to the much larger domain $\holdAll / \varepsilon$ with the inclusion $\omega$ of unit size, e.g., $\omega = B(0,1)$. In a second step we approximate this scaled version of \eqref{variationApprox1} by sending the boundary of the ``very large'' domain $D / \varepsilon$ to infinity. This yields a transmission problem on the plane $\RN$ which is independent of~$\varepsilon$.

We introduce the $\varepsilon$-independent operators corresponding to \eqref{def_Teps_I} and \eqref{def_Seps_I} at scale 1,
\begin{align}
	\tilde T (x,W) &:= \chi_{\RN \setminus \omega }(x) T(W) + \chi_{\omega}(x) \nu_0 W, \nonumber\\
	\tilde S_{W}(x, V) &:= \chi_{\RN \setminus \omega}(x) S_W(V), \label{def_Stil_I}
\end{align}
for $x \in  \RN$ and $V,W \in \RN$ with $T$ and $S$ given in \eqref{def_T} and \eqref{def_S}, respectively, and note that
\begin{align*}
	\mbox{D}\tilde T(x,W)= \chi_{\RN \setminus \omega }(x) \DT(W) + \chi_{\omega}(x) \nu_0 I.
\end{align*}

\begin{remark} \label{rem:propXDep_tilde}
	Again, the statements of Lemma \ref{lemmaProperties} and Lemma \ref{lemmaProperties2} also hold for the $x$-dependent operators $\tilde T$, $ \mbox{D}\tilde T$, $\tilde S$ for each point in their domain of definition with the same constants.
\end{remark}

With this notation, we arrive at the nonlinear transmission problem on $\RN$ defining $H$, the variation of the direct state at scale 1:
\begin{align} \label{variationApprox2}
	\begin{aligned}
		\mbox{Find } H \in \HH(\RN):
	& \int_{\RN} \left( \Tt(x,U_0 + \nabla H) - \Tt(x,U_0) \right)\cdot \nabla \test \\
	&	=-\int_{\omega} \left(\nu_0 - \nuh (|U_0|) \right) U_0 \cdot \nabla \test \quad \forall \test \in \HH(\RN).	
	\end{aligned}
\end{align}

Next, we show existence and uniqueness of a solution to \eqref{variationApprox2} using \cite[Thrm. 25.B.]{Zeidler1990} which was shown by Zarantonello in 1960.
\begin{proposition} \label{propExUnH}
	Let Assumption \ref{assump_BH} hold. Then there exists a unique solution $H \in \HH(\RN)$ to problem \eqref{variationApprox2}.
\end{proposition}
\begin{proof}
	We apply theorem of Zarantonello \cite[Thrm. 25.B.]{Zeidler1990} to problem \eqref{variationApprox2} rewritten in the form
	\begin{align*}
		 \mbox{Find }H \in \HH(\RN) \mbox{ such that } A \, H = L,
	\end{align*}
	where the operator $A:\HH(\RN) \rightarrow \HH(\RN)^*$ and the right hand side $L \in \HH(\RN)^{*}$ are defined by
	\begin{align*}
		\langle A \test_1, \test_2 \rangle 
			&= \int_{\RN} \left(\tilde T(x,U_0+\nabla \test_1) - \tilde T(x,U_0) \right) \cdot \nabla \test_2, \\
		\langle L, \test \rangle &= \int_{\omega} \left( \nuh (|U_0|) - \nu_0 \right) U_0 \cdot \nabla \test,
	\end{align*}
	for $\test_1, \test_2, \test \in \HH(\RN)$.
	We verify the strong monotonicity and Lipschitz continuity of the operator A.

	Property \eqref{prop5} together with Remark \ref{rem:propXDep_tilde} gives
	\begin{align*}
		\langle A \test_1 - A \test_2, \test_1 - \test_2 \rangle 
		&= \int_{\RN} \left( \Tt(x, U_0 + \nabla \test_1) - \Tt(x, U_0 + \nabla \test_2) \right) \left( \nabla \test_1 - \nabla \test_2 \right) \\
		&\geq \cEOld \int_{\RN} |\nabla \test_1-\nabla \test_2|^2 = \cEOld \| \nabla \test_1 - \nabla \test_2 \|_{L^2(\RN)}^2.
	\end{align*}
	The Poincar\'{e} inequality of Lemma \ref{lem:PoincareH} yields the strong monotonicity property in~$\HH(\RN)$.

	For the Lipschitz condition, we get by property \eqref{prop6} together with Remark \ref{rem:propXDep_tilde}, Cauchy's inequality and the norm equivalence of Corollary \ref{cor:CoercivityH} that
	\begin{align*}
		\| A \test_1 - A \test_2 \|_{\HH(\RN)^*} &= \underset{\test \neq 0}{\mbox{sup}} \frac{1}{\| \test \|_{\HH(\RN)}} \left|\langle A \test_1 - A \test_2, \test \rangle \right|\\
		&= \underset{\test \neq 0}{\mbox{sup}} \frac{1}{\| \test \|_{\HH(\RN)}} \left| \int_{\RN} \left(\tilde T(x,U_0+\nabla \test_1) - \tilde T(x,U_0+\test_2) \right) \cdot \nabla \test \right| \\
		&\leq \cFOld \, \underset{\test \neq 0}{\mbox{sup}} \frac{1}{\| \test \|_{\HH(\RN)}}  \int_{\RN} |\nabla \test_1 - \nabla \test_2| \, |\nabla \test| \\
		&\leq \cFOld \, \underset{\test \neq 0}{\mbox{sup}} \frac{1}{\| \test \|_{\HH(\RN)}}  \| \test_1 -  \test_2\|_{\HH(\RN)} \, \| \test\|_{\HH(\RN)} \\
		&= \cFOld \, \| \test_1 -  \test_2\|_{\HH(\RN)}.
	\end{align*}
	
	Therefore, the theorem of Zarantonello \cite[Thrm. 25.B.]{Zeidler1990} yields the existence of a unique solution $H \in \HH(\RN)$ to the variational problem \eqref{variationApprox2} since $L \in \HH(\RN)^* $.
\end{proof}

\begin{remark} \label{rem_U0_neq_0}
We note that, for $U_0=(0,0)^\top$, problems \eqref{variationApprox1} and \eqref{variationApprox2} only admit the trivial solution which yields that $\nabla H$, $\nabla h_{\varepsilon}$ are identical zero and also $S_{U_0}(\nabla H)$ and $S_{U_0}(\nabla h_{\varepsilon})$ vanish. In this case, many computations simplify significantly. For the rest of this section, we exclude the trivial case and assume that $U_0 \neq (0,0)^\top$.
\end{remark}
\subsubsection{Step 4: asymptotic behavior of variations of direct state} \label{sec_asymptBehH}
In this section, we investigate the asymptotic behavior of the solution $H$ to problem \eqref{variationApprox2} as $|x|$ goes to infinity. 
 
For $H \in \HH(\RN)$ the solution to $\eqref{variationApprox2}$, $\hat H \in \HH^w(\RN)$ a given element of the class $H$ and $\varepsilon >0$, we define the function $\He: \holdAll \rightarrow \mathbb R$ by
\begin{align}
	\He(x):= \varepsilon \hat H(\varepsilon^{-1}x), \quad x \in \holdAll. \label{def_Heps}
\end{align}
Note that, when one is only interested in the gradient of $\He$, the specific choice of $\hat H$ in the class $H$ does not matter.
Noting that
\begin{equation*}
	\underline w:= \underset{x \in \holdAll}{\mbox{inf}} \,w\left(\frac{x}{\varepsilon}\right) >0,
\end{equation*}
it is easy to see that $\hat H \in \HH^w(\RN)$ implies $H_{\varepsilon}\in H^1(\holdAll)$. We can show some first estimates which we will make use of in later estimations:

\begin{lemma} \info{(corresponds to Lemma 4.4.5 in \cite{Bonnafe2013})}\label{lem445}
Let Assumption \ref{assump_BH} and Assumption \ref{assump_reg_u0} hold. Then
	\begin{align} 
		\| \nabla \ut_{\varepsilon} \|_{L^2(\holdAll)}^2 = \mathcal O(\epsN),  \label{asympt_uEpsTilde}\\
		\| \nabla h_{\varepsilon} \|_{L^2(\holdAll)}^2 = \mathcal O(\epsN),  \label{asympt_hEps}\\
		\| \nabla H_{\varepsilon} \|_{L^2(\holdAll)}^2 = \mathcal O(\epsN).  \label{asympt_HEps}
	\end{align}
\end{lemma}
This proof is following the lines of \cite{Bonnafe2013, AmstutzBonnafe2015} and can be found in \cite{Gangl2017}.
 
	In order to show estimate \eqref{hEps_minus_Heps_L2} in Section \ref{subsubsec:justification}, we need that there exists a representative $\tilde H$ of the solution $H$ to \eqref{variationApprox2} which satisfies a sufficiently fast decay for $|x| \rightarrow \infty$.
	For that purpose, let the nonlinear operator $Q: \HH(\RN) \rightarrow \HH(\RN)^*$ be defined by
	\begin{align} \label{defQ}
		\hspace{-10mm}\langle Q \test_1, \test_2 \rangle := \int_{\RN} \left[ \Tt(x, U_0 + \nabla \test_1) - \Tt(x,U_0) \right] \cdot \nabla \test_2 + \int_{\omega} (\nu_0 - \nuh(|U_0|)) U_0 \cdot \nabla \test_2.
	\end{align}
	Note that for $H$ the solution of \eqref{variationApprox2}, we have that 
	\begin{align} \label{eq_QHzero}
		\langle Q\, H, \test \rangle = 0 \quad \mbox{ for all } \test \in \HH(\RN).
	\end{align}
	
	In the following, we show that there exist a supersolution $\myP$ satisfying ${\langle Q\, \myP , \test \rangle \geq 0}$ and a subsolution $\myRsub$ such that ${\langle Q\, \myRsub , \test \rangle \leq 0}$ for all test functions $\test$ in a subset of $\HH(\RN)$, both of which satisfy a sufficient decay at infinity. Then we make use of a comparison principle to show that there exists a representative $\tilde H$ of the solution $H$ of \eqref{eq_QHzero} which satisfies $\myRsub(x) \leq \tilde H(x) \leq \myP(x)$ almost everywhere and conclude that $\tilde H$ must have the same decay at infinity as $\myP$ and $\myRsub$.

	For this purpose, we first introduce a coordinate system that is aligned with the fixed vector $U_0$.
	Since we excluded the trivial case where $U_0 = (0,0)^\top$ (see Remark \ref{rem_U0_neq_0}), we can introduce the unit vector $e_1 = U_0 /|U_0|$ and the orthonormal basis $(e_1, e_2)$ of $\RN$. We denote $(x_1, x_2)$ the system of coordinates in this basis and introduce the half space  $\RN_+ := \lbrace x \in \RN: U_0 \cdot x \geq 0 \rbrace$. We first show that there exists a representative $\tilde H$ of the solution $H$ to \eqref{variationApprox2} that is odd with respect to the first coordinate.
	
	\begin{lemma} \label{lem:Hodd}
		Let $H\in \HH(\RN)$ be the unique solution to the operator equation $QH=0$ with $Q$ defined in \eqref{defQ} and assume that $\omega$ is symmetric with respect to the line $\lbrace x \in \RN: U_0 \cdot x = 0 \rbrace$. Then there exists an element $\tilde H$ of the class $H$ such that, for all $(x_1, x_2) \in \RN$,
		\begin{align*}
			\tilde H(-x_1,x_2) = -\tilde H(x_1, x_2).
		\end{align*}
		In particular, $\tilde H(0,x_2) = 0$ for all $x_2\in \mathbb R$.
	\end{lemma}
    
	A proof of Lemma \ref{lem:Hodd} can be found in Section \ref{sec_proofsTDNL} on page \pageref{lem:Hodd_proof}.

	Lemma \ref{lem:Hodd} allows us to investigate the asymptotic behavior of $\tilde H$ only in the half plane $\RN_+$. However, Lemma \ref{lem:Hodd} is based on the assumption that $\omega$ is symmetric with respect to the line $\lbrace x \in \RN: U_0 \cdot x = 0 \rbrace$. In order to fulfill this assumption for any given $U_0 \in \RN$, we will from now on restrict ourselves to the case where $\omega = B(0,1)$.

	\begin{proposition} \label{prop:supersol1}
		\info{(corresponds to Proposition 4.4.6 in \cite{Bonnafe2013})} 
		Let $\omega = B(0,1)$ and let Assumption \ref{assump_delta} hold. Then there exists $\sigma > \Nhalf$ and $k\in(0,1]$ such that the function $\myP \in \HH(\RN)\cap L^{\infty}(\RN)$ defined by 
		\begin{align} \label{defP}
			\myP(x) := \begin{cases}
						 k(U_0 \cdot x) |x|^{-\sigma} & x \in \RN \setminus \omega, \\
						 k(U_0 \cdot x) &  x \in \omega,
			        \end{cases}
		\end{align}
		satisfies
		\begin{align}
			\langle Q\, \myP, \test \rangle \geq 0  \quad \forall \test \in \HH(\RN):\mbox{ supp}(\test) \subset \RN_+, \, \test \geq 0 \mbox{  a.e.} \label{QPgreaterEqualZero}
		\end{align}
		for the operator $Q$ defined in \eqref{defQ}.
	\end{proposition}
    
	The proof of Proposition \ref{prop:supersol1} is very technical and can be found on page \pageref{prop:supersol1_proof}.

	Next, we provide a subsolution $\myRsub$ satisfying $\langle Q \, \myRsub , \test \rangle \leq 0$ for all $\test$ from the same set of test functions.
	\begin{proposition} \label{prop:subsol}
	Let $\omega = B(0,1)$ and let Assumption \ref{assump_delta} hold. Then there exists $\sigma > \Nhalf$ such that the function $\myRsub \in \HH(\RN)\cap L^{\infty}(\RN)$ defined by 
		\begin{align} \label{defR}
			\myRsub (x):= \begin{cases}
					k (U_0 \cdot x) |x|^{-\sigma} & x \in \RN \setminus \omega,\\
					k (U_0 \cdot x) & x \in \omega,
			     \end{cases}
		\end{align}
		with 
		\begin{align*}
			k = \frac{\numin-\nu_0}{\nu_0 + \numin(\sigma-1)} \in (-1,0)
		\end{align*}
		satisfies
		\begin{align}
			\langle Q \, \myRsub, \test \rangle \leq 0 \quad \forall \test \in \HH(\RN): \mbox{supp}(\test) \subset \RN_+, \test \geq 0 \mbox{  a.e.} \label{QRgreaterEqualZero}
		\end{align}
		for the operator $Q$ defined in \eqref{defQ}.
	\end{proposition}
    
	The technical proof of Proposition \ref{prop:subsol} can be found in Section \ref{sec_proofsTDNL} on page \pageref{prop:subsol_proof}.

	Now, we can show that there exists an element $\tilde H$ of the class $H$, where $H \in \HH(\RN)$ is the solution to \eqref{variationApprox2}, which has the same asymptotic behavior as $\myP$ and $\myRsub$ defined in \eqref{defP} and \eqref{defR}, respectively, by means of a comparison principle.
	\begin{proposition} \label{prop:compPrinciple}
		Let $\omega = B(0,1)$ and let Assumption \ref{assump_delta} hold. Let $\myP$ the supersolution defined in Proposition \ref{prop:supersol1}, $\myRsub$ the subsolution defined in Proposition \ref{prop:subsol} and $H\in \HH(\RN)$ the unique solution to the operator equation $QH=0$ with $Q$ defined in \eqref{defQ}. Then there exists an element $\tilde H$ of the class $H$ such that 
		\begin{align}
			\myRsub(x) \leq \tilde H(x) \leq \myP(x) \quad \forall x \in \RN_+ \mbox{ a.e.} \label{est_R_H_P}
		\end{align}
	\end{proposition}
Proposition \ref{prop:compPrinciple} can be proved by an adaptation of the proof in \cite{Bonnafe2013, AmstutzBonnafe2015}, which can be found in \cite{Gangl2017}.

	Finally, collecting the results of Propositions \ref{prop:supersol1}, \ref{prop:subsol} and \ref{prop:compPrinciple}, we can state the following result:
	\begin{theorem} \label{theo:asymp_H}
		Let $\omega = B(0,1)$, let Assumption \ref{assump_delta} hold and let $H\in \HH(\RN)$ be the unique solution to the operator equation $QH=0$ with $Q$ defined in \eqref{defQ}. Then there exists an element $\tilde H$ of the class $H\in \HH(\RN)$ and $\tau > \Nhalfmone$ such that
		\begin{align}
			 \tilde H(y) = \mathcal O \left( |y|^{-\tau} \right) \; \mbox{as} \; |y| \rightarrow \infty. \label{asymp_H}
		\end{align}
	\end{theorem}

	\begin{remark}
		As $\myP$ and $\myRsub$ are both in $L^{\infty}(\RN)$, we also have that $H \in L^{\infty}(\RN)$.
	\end{remark}

	From now on, the function $H_{\varepsilon}$ is defined choosing $\hat H = \tilde H$ in \eqref{def_Heps} where $\tilde H$ is as in Theorem \ref{theo:asymp_H}, i.e., $\He(x) := \varepsilon \Ht (\varepsilon^{-1}x)$ for $x \in \holdAll.$

	\subsubsection{Estimates for the variations of the direct state} \label{subsubsec:justification}
    Exploiting the asymptotic behavior of Theorem \ref{theo:asymp_H}, the following estimates can be obtained.
	\begin{proposition}\label{proposition4412}
		\info{(corresponds to Proposition 4.4.12 in \cite{Bonnafe2013})}
		Let $\omega = B(0,1)$ and let Assumptions \ref{assump_BH}, \ref{assump_delta} and \ref{assump_reg_u0} hold. Then
		\begin{align}
			\| \nabla h_{\varepsilon} - \nabla H_{\varepsilon} \|^2_{L^2(\holdAll)} &= o(\epsN),		\label{hEps_minus_Heps_L2} \\
			\forall \alpha>0 \, \forall r\in(0,1): \int_{\holdAll \setminus B(0, \alpha \varepsilon^r)} |\nabla h_{\varepsilon}|^2 &= o(\epsN), \label{estimate_4422} \\
			\int_{\holdAll} |\nabla u_0 - U_0| |\nabla h_{\varepsilon}|^2 &= o(\epsN),  \label{estimate_4423} \\
			\int_{\holdAll} |\nabla h_{\varepsilon} - \nabla H_{\varepsilon}|(|\nabla h_{\varepsilon}| + |\nabla H_{\varepsilon}|)&= o(\epsN). \label{estimate_4426} 
		\end{align}
	\end{proposition}
	Proposition \ref{proposition4412} can be shown in a similar way as it was done in \cite{Bonnafe2013, AmstutzBonnafe2015}, see \cite{Gangl2017} for an adaptation to our case.

	Note that, in the proof of \eqref{hEps_minus_Heps_L2}, the structure of the simplified setting introduced in Section~\ref{sec_simplSetting} is exploited.
	
	\begin{proposition} \label{proposition4413}
		\info{(corresponds to Proposition 4.4.13 in \cite{Bonnafe2013})} 
		Let $\omega = B(0,1)$ and let Assumptions \ref{assump_BH}, \ref{assump_delta} and \ref{assump_reg_u0} hold. Then
		\begin{align}
			\| \nabla \ut_{\varepsilon} - \nabla h_{\varepsilon} \|_{L^2(\holdAll)}^2 &= o(\epsN), \label{uEpsTilde_minus_hEps} \\
			\int_{\holdAll} |\nabla \ut_{\varepsilon} - \nabla h_{\varepsilon}| (|\nabla \ut_{\varepsilon}| + |\nabla h_{\varepsilon}|) &= o(\epsN),\label{estimate_4429} \\
			\forall \alpha >0 \, \forall r \in(0,1): \int_{\holdAll \setminus B(0,\alpha \varepsilon^r)} |\nabla \ut_{\varepsilon}|^2 &= o(\epsN) \label{estimate_4431}.
		\end{align}
	\end{proposition}
	The proof of this proposition is following the lines of \cite{Bonnafe2013, AmstutzBonnafe2015}, too, and is adapted to our case in \cite{Gangl2017}.

	\subsection{Variation of adjoint state} \label{sec:VarAdjoint}
	For $\varepsilon > 0$, we introduce the perturbed adjoint equation to the PDE-constrained optimization problem \eqref{minJ}--\eqref{PDEconstraintUnperturbed} in the simplified setting of Section \ref{sec_simplSetting},
	\begin{align}
		\mbox{Find } \adje \in H^1_0(\holdAll): 
		&\int_{\holdAll}  \mbox{D}\Te(x, \nabla u_0)  \nabla \adje \cdot \nabla \test = - \langle \tilde G, \test \rangle \quad  \forall \test \in H^1_0(\holdAll), \label{adjointEquation_scaleEps}
	\end{align}
	where $\Te$ is given in \eqref{def_Teps_I} and $\tilde G$ fulfills Assumption \ref{assump_J} together with the objective function $\mathcal J$. Note that $ \mbox{D}\Te$ is a symmetric matrix. For $\varepsilon = 0$ we get the unperturbed adjoint equation,
	\begin{align}
		\mbox{Find } \adjz \in H^1_0(\holdAll):
		&\int_{\holdAll} \DT(\nabla u_0) \nabla \adjz \cdot \nabla \test = -\langle \tilde G, \test \rangle \quad \forall \test \in H^1_0(\holdAll), \label{adjoint_eps0}
	\end{align}
	where we used that $\DT_0(x,\nabla u_0) = \DT(\nabla u_0)$ according to the definition of $T_{\varepsilon}$ \eqref{def_Teps_I}.
	For $\varepsilon >0$, we call $\adje$ the perturbed adjoint state, and $\adjz$ the unperturbed adjoint state. Note that we use the same right hand side $\tilde G$, independently of the parameter $\varepsilon \geq0$.
	
	\subsubsection{Regularity assumptions} \label{sec:regularityAssumptions_adjoint}
	Similarly to Assumption \ref{assump_reg_u0}, we also need the unperturbed adjoint state $\adjz$ to be sufficiently regular in a neighborhood of the point of the perturbation $x_0 \in \Omega^d$. We assume the following:
	
	\begin{assumption} \label{assump_reg_v0}
	There exists $\tilde \beta >0$ such that 
		\begin{align*}
			\adjz|_{\Omega^d} \in C^{1,\tilde \beta}(\Omega^d).
		\end{align*}
	\end{assumption}
	\begin{remark}
        In our case, the right hand side of the adjoint equation \eqref{adjoint_eps0} is a distribution $\tilde G \in H^{-1}(\holdAll)$ satisfying an expansion of the form \eqref{expansion_J}. However, since we assume that $\mathcal J$ is not supported in $\Omega^d$, $\tilde G$ is not supported in $\Omega^d$ and thus the assumption that $\adjz$ is smooth in $\Omega^d$ is reasonable. \unsure{talk to somebody..}
	\end{remark}
		
	\subsubsection{Step 1: variation $\adje - \adjz$}
	We proceed in an analogous way to Section \ref{sec:Step1Direct_I}.
	Subtracting \eqref{adjoint_eps0} from \eqref{adjointEquation_scaleEps}, we get
	 the boundary value problem defining the variation of the adjoint state at scale $\varepsilon$, $\adjet := \adje - \adjz$:
	\begin{align} 
		&\mbox{Find } \adjet \in H^1_0(\holdAll) :\nonumber\\
		&\int_{\holdAll} \DT_{\varepsilon}(x,\nabla u_0) \nabla \adjet \cdot \nabla \test = -\int_{\ome} (\nu_0 \, I - \DT(\nabla u_0) ) \nabla \adjz \cdot \nabla \test \;  \forall \test \in H^1_0(\holdAll).\label{variationAdjoint_scaleEps}
	\end{align}
	
	\subsubsection{Step 2: approximation of Variation $\adje - \adjz$} \label{sec:approxStep1adj}
	Analogously to Section \ref{sec:approxStep1}, we approximate boundary value problem \eqref{variationAdjoint_scaleEps} by the same boundary value problem where the functions $\nabla u_0, \nabla \adjz$ are replaced by their values at the point $x_0$, $U_0 = \nabla u_0(x_0)$ and $\Adjz := \nabla \adjz(x_0)$, respectively. Again, note that this point evaluation makes sense due to Assumptions \ref{assump_reg_u0} and \ref{assump_reg_v0}. We denote the solution to the arising boundary value problem by $k_{\varepsilon}$:
	\begin{align}
		&\mbox{Find } k_{\varepsilon} \in H^1_0(\holdAll) :\nonumber\\
		&\int_{\holdAll} \DT_{\varepsilon}(x,U_0) \nabla k_{\varepsilon} \cdot \nabla \test = -\int_{\ome} (\nu_0 \, I - \DT(U_0) ) \Adjz \cdot \nabla \test \quad  \forall \test \in H^1_0(\holdAll). \label{variationAdjoint_scaleEpsApprox}
	\end{align}

	\subsubsection{Step 3: change of scale} \label{sec:approxStep2adj}
	Also here, we proceed analogously to the case of the variation of the direct state presented in \ref{sec:changeOfScale}. We perform a change of scale and then approximate boundary value problem \eqref{variationAdjoint_scaleEpsApprox} by sending the outer boundary to infinity, which yields the linear transmission problem
	\begin{align} \label{variationAdjoint_scale1}
		&\mbox{Find } K \in \HH(\RN) :\nonumber\\
		&\int_{\RN} \mbox{D}\Tt(x, U_0) \nabla K \cdot \nabla \test = -\int_{\omega} (\nu_0 \, I - \DT(U_0) ) \Adjz \cdot \nabla \test \quad  \forall \test \in \HH(\RN).
	\end{align}
	Note that \eqref{variationAdjoint_scale1} is independent of $\varepsilon$.
	
	It is straightforward to establish the well-posedness of problem \eqref{variationAdjoint_scale1}:
	\begin{lemma} \label{lem_exUniqK}
		Let Assumption \ref{assump_BH} hold. Then, there exists a unique solution $K \in \HH(\RN)$ to problem \eqref{variationAdjoint_scale1}.
	\end{lemma}
	\begin{proof}
		We show existence and uniqueness of a solution $H\in\HH(\RN)$ to \eqref{variationAdjoint_scale1} by the lemma of Lax-Milgram. The coercivity and boundedness of the left hand side of \eqref{variationAdjoint_scale1} can be shown by exploiting \eqref{prop4} together with Remark \ref{rem:propXDep_tilde}, the physical properties \eqref{physProperties} and the norm equivalence of Corollary \ref{cor:CoercivityH}. The right hand side of \eqref{variationAdjoint_scale1} is obviously a bounded linear functional on $\HH(\RN)$.
	\end{proof}

	\begin{remark}
	We remark that, for $\Adjz=(0,0)^\top$, problems \eqref{variationAdjoint_scaleEpsApprox} and \eqref{variationAdjoint_scale1} only admit the trivial solution such that $\nabla K$, $\nabla k_{\varepsilon}$ are identical zero. In this case, many computations simplify significantly. For the rest of this work, we exclude the trivial case and assume that $\Adjz \neq (0,0)^\top$.
	\end{remark}

	\subsubsection{Step 4: Asymptotic behavior of variations of the adjoint state} \label{sec_step4adjoint}
	Let $K \in \HH(\RN)$ be the unique solution in $\HH(\RN)$ to \eqref{variationAdjoint_scale1} and let $\hat K \in \HH^w(\RN)$ denote a given element of the class $K$. For $\varepsilon>0$, let $\Ke: \holdAll \rightarrow \mathbb R$ be defined by
	\begin{align} \label{def_Keps}
		\Ke(x):= \varepsilon \hat K(\varepsilon^{-1} x).
	\end{align}
	As in the case of the variation of the direct state, making the change of scale backwards, it follows from $\hat K \in \HH^w(\RN)$ that $\Ke \in H^1(\holdAll)$, since
	\begin{align*}
		\underset{x \in  \holdAll}{\mbox{inf}} \, w\left(\frac{x}{\varepsilon}  \right) > 0.
	\end{align*}
	
	\begin{lemma}\info{corresponds to Lemma 4.5.3 \cite{Bonnafe2013} } \label{lemma453}
		Let Assumptions \ref{assump_BH}, \ref{assump_c78}, \ref{assump_reg_u0} and \ref{assump_reg_v0}. Then it holds
		\begin{align}
			\| \nabla \adjet \| _{L^2(\holdAll)}^2 &= O(\epsN), \label{norm_vepstilde}\\
			\| \nabla k_{\varepsilon} \| _{L^2(\holdAll)}^2 &= O(\epsN), \label{norm_keps}\\
			\| \nabla K_{\varepsilon} \| _{L^2(\holdAll)}^2 &= O(\epsN).\label{norm_Keps}
		\end{align}
	\end{lemma}
This proof is following the lines of \cite{Bonnafe2013, AmstutzBonnafe2015} and can be found in \cite{Gangl2017}.

	Next, we show an asymptotic behavior of an element of the class $K \in \HH(\RN)$ similar to \eqref{asymp_H}.
	\begin{proposition}\info{corresponds to Prop 4.5.4 in \cite{Bonnafe2013} } \label{prop454}
		Let Assumption \ref{assump_BH} hold and let $K \in \mathcal H(\RN)$ the unique solution to \eqref{variationAdjoint_scale1} according to Lemma \ref{lem_exUniqK}. Then there exists an element $\tilde{K}$ of the class $K$ such that
		\begin{align}
			\tilde{K}(y)= \mathcal O(|y|^{\onemN}) \quad \mbox{as} \quad |y|\rightarrow \infty. \label{asympt_K}
		\end{align}
	\end{proposition}
This proposition can be shown in a standard way, see e.g., \cite{Ammari2008, Bonnafe2013, AmstutzBonnafe2015} and the proof for our setting can be found in \cite{Gangl2017}.
		
	Let, from now on, the function $K_{\varepsilon}$ \eqref{def_Keps} be defined by choosing $\hat K = \tilde K$ where $\tilde K$ is the element of the class $K\in\HH(\RN)$, which satisfies the asymptotic behavior \eqref{asympt_K}. Here, $K \in \HH(\RN)$ is the unique solution to \eqref{variationAdjoint_scale1} according to Lemma \ref{lem_exUniqK}.
    	
	\begin{lemma}\info{corresponds to Lemma 4.5.5 in \cite{Bonnafe2013} } \label{lemma455}
		Let Assumption \ref{assump_BH} hold. Then, it holds
		\begin{align}
			\|\nabla k_{\varepsilon} - \nabla K_{\varepsilon}\|_{L^2(\holdAll)}^2 =  o(\epsN) \mbox{ and} \label{estimate_keps_Keps}
		\end{align}
		\begin{align}
			\forall \alpha > 0 \, \forall r\in(0,1): \int_{\holdAll \setminus B(0, \alpha \varepsilon^r)} |\nabla k_{\varepsilon}|^2 =  o(\epsN). \label{estimate_kepsOutsideball}
		\end{align}	 
	\end{lemma}
This proof is an adaptation of \cite{Bonnafe2013, AmstutzBonnafe2015} which can be found in \cite{Gangl2017}.

	\begin{lemma}\info{corresponds to Lemma 4.5.6 in \cite{Bonnafe2013} } \label{lemma456}
		Let Assumptions \ref{assump_BH}, \ref{assump_c78}, \ref{assump_reg_u0} and \ref{assump_reg_v0} be satisfied. Then, it holds
		\begin{align}
			\| \nabla \adjet - \nabla k_{\varepsilon} \|_{L^2(\holdAll)}^2 = o(\epsN). \label{estimate_vteps_keps}
		\end{align}
	\end{lemma}
This proof is following the lines of \cite{Bonnafe2013, AmstutzBonnafe2015} and is adapted to our case in \cite{Gangl2017}.
		
	\subsection{Topological asymptotic expansion} \label{sec_topAsympExp_I}
	Recall Assumption \ref{assump_J}.
   	By estimate \eqref{estimate_4431}, it follows that 
	\begin{align} \label{Reps_eps2}
		R(\varepsilon) = o(\epsN).
	\end{align}
	We have a closer look at the term $\langle \tilde G, \uet \rangle$. Testing adjoint equation \eqref{adjointEquation_scaleEps} for $\varepsilon>0$ with $\test = \uet$ and exploiting the symmetry of $\DT_{\varepsilon}$, we get
	\begin{align*}
		 \langle \tilde G, \uet \rangle =& -\int_{\holdAll} \DT_{\varepsilon}(x,\nabla u_0) \nabla \tilde u_{\varepsilon} \cdot \nabla \adje  \\
		 =& -\int_{\holdAll} \DT_{\varepsilon}(x,\nabla u_0) \nabla \tilde u_{\varepsilon} \cdot \nabla \adje  \\
		 &+ \int_{\holdAll} \left( \Te(x, \nabla u_0 + \nabla \uet)  - \Te(x,\nabla u_0) \right) \cdot \nabla \adje + \int_{\ome}(\nu_0 - \hat{\nu}(|\nabla u_0|)) \nabla u_0 \cdot \nabla \adje,
	\end{align*}
	where we added the left and right hand side of \eqref{variationDirect_scaleEps} tested with $\test = \adje$. According to the definition of the operator $S^\varepsilon$ \eqref{def_Seps_I}, we get
	 \begin{align*}
		\langle \tilde G, \uet \rangle &= \int_{\ome}(\nu_0 - \hat{\nu}(|\nabla u_0|)) \nabla u_0 \cdot \nabla \adje + \int_{\holdAll} S^{\varepsilon}_{\nabla u_0}(x,\nabla \uet) \cdot \nabla \adje.
	 \end{align*}
	 Noting that $\adje = \adjz  +  \adjet$, and defining
	 \begin{align}
		j_1(\varepsilon) &:= \int_{\ome}(\nu_0 - \hat{\nu}(|\nabla u_0|)) \nabla u_0 \cdot ( \nabla \adjz + \nabla \adjet), \label{def_j1} \\
		j_2(\varepsilon) &:= \int_{\holdAll} S^{\varepsilon}_{\nabla u_0}(x,\nabla \uet) \cdot ( \nabla \adjz + \nabla \adjet),\label{def_j2}
	 \end{align}
	 together with \eqref{Reps_eps2}, we get from \eqref{expansion_J} that
	\begin{align}
		\mathcal J_{\varepsilon}(u_{\varepsilon}) - \mathcal J_{0}(u_{0}) = &j_1(\varepsilon) + j_2(\varepsilon) + \delta_J \epsN + o(\epsN). \label{topAsympExp_final}
	\end{align}
	Note that the operator $S_{\nabla u_0}^{\varepsilon}$ represents the nonlinearity of the problem. Therefore, the term $j_2$ vanishes in the linear case where the nonlinear function $\nuh$ is replaced by a constant $\nu_1$.

	In Sections \ref{sec:Expansion_j1} and \ref{sec:Expansion_j2} we will show that there exist numbers $J_1$, $J_2$ such that
	\begin{align*}
		j_1(\varepsilon) = \varepsilon^2 \, J_1 + o(\varepsilon^2) \quad \mbox{ and } \quad j_2(\varepsilon) = \varepsilon^2 \, J_2 + o(\varepsilon^2).
	\end{align*}
	Comparing expansion \eqref{topAsympExp_final} with \eqref{def:topDerivative} this will yield the final formula for the topological derivative,
	\begin{align*}
		G(x_0) = J_1 + J_2 + \delta_J
	\end{align*}
	in Theorem \ref{theo_MainResultCaseI}.
	
	\subsubsection{Expansion of linear term $j_1(\varepsilon)$} \label{sec:Expansion_j1}
	Following approximation steps 2 and 3 of Sections \ref{sec:approxStep1adj} and \ref{sec:approxStep2adj}, respectively, we define
	\begin{align}
		\tilde j_1(\varepsilon) &:=  \left( \nu_0 - \nuh(|U_0|) \right) \int_{\omega_{\varepsilon}} U_0 \cdot \left( \Adjz + \nabla k_{\varepsilon} \right), \label{def_j1tilde}  \\
		J_1 &:= \left( \nu_0 - \nuh(|U_0|) \right) \int_{\omega} U_0 \cdot \left( \Adjz + \nabla K \right). \label{def_J1}
	\end{align}
	
	\begin{lemma} \label{lemma_j1tilde_J1}
	\info{corresponds to Lemma 4.6.1 in \cite{Bonnafe2013}}
		Let Assumption \ref{assump_BH} hold. Then it holds
		\begin{align}
			\tilde j_1(\varepsilon) - \epsN J_1 = o(\epsN). \label{lemma461}
		\end{align}
	\end{lemma}
This proof is following the lines of \cite{Bonnafe2013, AmstutzBonnafe2015} and is adapted to our case in \cite{Gangl2017}.

	\begin{lemma} \label{lemma_j1_j1tilde}
	\info{corresponds to Lemma 4.6.2 in \cite{Bonnafe2013}}
		Let Assumptions \ref{assump_BH}, \ref{assump_c78}, \ref{assump_reg_u0} and \ref{assump_reg_v0} hold. Then it holds
		\begin{align}
			j_1(\varepsilon) - \tilde j_1(\varepsilon) = o(\epsN). \label{lemma462}
		\end{align}
	\end{lemma}
This proof is following the lines of \cite{Bonnafe2013, AmstutzBonnafe2015} and is adapted to our case in \cite{Gangl2017}.

	Considering \eqref{def_J1}, it follows from the linearity of equation \eqref{variationAdjoint_scale1} that the mapping
	\begin{align*}
		\Adjz \mapsto \left( \nu_0 - \nuh(|U_0|) \right) \int_{\omega} \left(\Adjz +  \nabla K \right)
	\end{align*}
	is linear from $\RN$ to $\RN$. It only depends on the set $\omega$, and on the positive definite matrix~$\DT(U_0)$. Hence, there exists a matrix
	\begin{align}
		\myTDMat  = \myTDMat(\omega, \DT(U_0) ) \in \mathbb R^{2\times2}, \label{intro_M1}
	\end{align}
	such that 
	\begin{align*}
		\left( \nu_0 - \nuh(|U_0|) \right)  \int_{\omega} \left( \Adjz + \nabla  K \right) = \myTDMat \, \Adjz.
	\end{align*}
	This matrix $\myTDMat$ is related to the concept of polarization matrices, see, e.g., \cite{AmmariKang2007}.
	Eventually, it follows that
	\begin{align}
		J_1 = U_0^\top \, \myTDMat \, \Adjz. \label{J1_polMat}
	\end{align}
	In Section \ref{sec_AppendixCaseI}, an explicit formula for the matrix $\myTDMat = \myTDMat(\omega, \DT(U_0) )$ will be derived.
	
	Summing up estimates \eqref{lemma461} and \eqref{lemma462}, as well as \eqref{J1_polMat}, we get the following result:
	\begin{corollary}
	Let Assumptions \ref{assump_BH}, \ref{assump_c78}, \ref{assump_reg_u0} and \ref{assump_reg_v0} hold. Then, there exists a matrix $\myTDMat  = \myTDMat(\omega, \DT(U_0) ) \in \mathbb R^{2 \times 2}$ such that
		\begin{align}
			j_1(\varepsilon) = \epsN \; U_0^\top \, \myTDMat \, \Adjz + o(\epsN). \label{j1_final}
		\end{align}
	\end{corollary}
	
	\subsubsection{Expansion of nonlinear term $j_2(\varepsilon)$} \label{sec:Expansion_j2}
	
	According to the approximation steps taken for the variations of the direct and adjoint state in Sections \ref{sec:approxStep1}, \ref{sec:changeOfScale}, as well as \ref{sec:approxStep1adj} and \ref{sec:approxStep2adj}, we define
	\begin{align}
		\tilde j_2(\varepsilon) &:= \int_{\holdAll} S^{\varepsilon}_{U_0}(x,\nabla h_{\varepsilon}) \cdot ( \Adjz + \nabla k_{\varepsilon}), \label{def_j2tilde} \\
		J_2 &:= \int_{\RN} \tilde S_{U_0}(x,\nabla H) \cdot ( \Adjz + \nabla K). \label{def_J2} 
	\end{align}
	Note that, under Assumption \ref{assump_c78}, both $\tilde j_2(\varepsilon)$ and $J_2$ are well-defined due to growth condition~\eqref{prop7_followUp}.

	\begin{lemma} \label{lemma464}
		\info{Corresponds to lemma 4.6.4 in \cite{Bonnafe2013} }
		Let $\omega = B(0,1)$ and let Assumptions \ref{assump_BH}, \ref{assump_c78}, \ref{assump_delta} and \ref{assump_reg_u0} hold. Then it holds
		\begin{align}
			\tilde j_2(\varepsilon) - \epsN J_2 = o(\epsN). \label{estimate_4618}
		\end{align}
	\end{lemma}
	This proof is following the lines of \cite{Bonnafe2013, AmstutzBonnafe2015} and is adapted to our case in \cite{Gangl2017}.

	\begin{lemma} \label{lemma465}
		\info{corresponds to lemma 4.6.5 \cite{Bonnafe2013} }
		Let $\omega = B(0,1)$ and let Assumptions \ref{assump_BH}, \ref{assump_delta}, \ref{assump_reg_u0} and \ref{assump_reg_v0} be satisfied. Then, it holds
		\begin{align}
			\int_{\holdAll} |\nabla \adjz - \Adjz| |\nabla h_{\varepsilon}|^2 = o(\epsN). \label{estimate_4619}
		\end{align}
	\end{lemma}
	This proof can be found in \cite{Bonnafe2013, AmstutzBonnafe2015, Gangl2017}.
		
	\begin{lemma} \label{lemma466}
		\info{Corresponds to lemma 4.6.6 \cite{Bonnafe2013}}
		Let $\omega = B(0,1)$ and let Assumptions \ref{assump_BH}, \ref{assump_c78}, \ref{assump_delta}, \ref{assump_reg_u0} and \ref{assump_reg_v0} be satisfied. Then, it holds
		\begin{align*}
			j_2(\varepsilon) - \tilde j_2(\varepsilon) = o(\epsN).
		\end{align*}
	\end{lemma}
	This proof is following the lines of \cite{Bonnafe2013, AmstutzBonnafe2015} and is adapted to our case in \cite{Gangl2017}.

	Eventually, combining Lemma \ref{lemma464} and Lemma \ref{lemma466}, we get the following result:
	\begin{corollary}
		Let $\omega = B(0,1)$ and let Assumptions \ref{assump_BH}, \ref{assump_c78}, \ref{assump_delta}, \ref{assump_reg_u0} and \ref{assump_reg_v0} be satisfied. Then we have
		\begin{align}
			j_2(\varepsilon) = \epsN  \left( \int_{\RN} \tilde S_{U_0}(x,\nabla H) \cdot (\Adjz + \nabla K) \right) + o(\epsN). \label{j2_final}
		\end{align}
	\end{corollary}
	
	\subsubsection{Main result}
	Finally, combining \eqref{topAsympExp_final} with \eqref{j1_final} and \eqref{j2_final}, we get the main result of this paper, i.e., the topological derivative for the introduction of linear material (air) inside a region of nonlinear (ferromagnetic) material according to \eqref{def:topDerivative}.
	We recall the notation needed for stating the result of Theorem \ref{theo_MainResultCaseI}:
	\begin{itemize}
		\item $x_0 \in \Omega^d$ denotes the point around which we perturb the material coefficient,
		\item $u_0 \in H_0^1(\holdAll)$ is the unperturbed direct state, i.e., the solution to \eqref{weakUnpert}, and $U_0 = \nabla u_0(x_0)$,
		\item $\adjz \in H_0^1(\holdAll)$ is the unperturbed adjoint state, i.e., the solution to \eqref{adjoint_eps0}, and $\Adjz = \nabla \adjz(x_0)$,
		\item $\myTDMat = \myTDMat(\omega, \DT(U_0))$ denotes the matrix defined in \eqref{TDMat_CaseI} where $\omega$ represents the shape of the inclusion and $\DT$ is the Jacobian of $T$ defined in \eqref{def_T},
		\item $H \in \HH(\RN)$ denotes the variation of the direct state at scale 1, i.e., the solution to~\eqref{variationApprox2},
		\item $K \in \HH(\RN)$ denotes the variation of the adjoint state at scale 1, i.e., the solution to~\eqref{variationAdjoint_scale1},
		\item $\tilde S$ is defined in \eqref{def_Stil_I},
		\item $\delta_J$ is according to \eqref{expansion_J}.
	\end{itemize}
	\begin{theorem} \label{theo_MainResultCaseI}
		Assume that 
		\begin{itemize}
		 \item[-] $\omega = B(0,1)$ the unit disk in $\RN$
		 \item[-] the ferromagnetic material is such that Assumptions \ref{assump_BH}, \ref{assump_c78} and \ref{assump_delta} are satisfied,
		 \item[-] the functional $\mathcal J_{\varepsilon}$ satisfies Assumption \ref{assump_J},
		 \item[-] the unperturbed direct state $u_0$ satisfies Assumption \ref{assump_reg_u0}, i.e., $u_0 \in C^{1,\beta}$ for some $\beta>0$,
		 \item[-] the unperturbed direct state $\adjz$ satisfies Assumption \ref{assump_reg_v0}, i.e., $\adjz \in C^{1,\tilde \beta}$ for some $\tilde \beta>0$.
		\end{itemize}
		Then the topological derivative for introducing air inside ferromagnetic material reads
		\begin{align}
			\begin{aligned}
				G^{f \rightarrow air}(x_0) 
				=& U_0^\top \, \myTDMat \, \Adjz \\
				&+ \int_{\RN} \tilde S_{U_0}(x,\nabla H) \cdot (\Adjz + \nabla K)  + \delta_J. \label{G1_final}
			\end{aligned}
		\end{align}
	\end{theorem}

	\begin{remark} \label{rem_shapeOfInclusion}
		The proof of Theorem \ref{theo_MainResultCaseI} is valid only under the assumption that $\omega = B(0,1)$. This is mainly due to the fact that the proof of Proposition \ref{prop:compPrinciple} uses Lemma \ref{lem:Hodd} which exploits the symmetry of $\omega$ with respect to the line $\lbrace U_0 \cdot x = 0 \rbrace$. Since we need to make sure that this condition is satisfied for any possible $U_0$, we have to assume that $\omega$ is a disk in the sequel. Thus, the condition $\omega = B(0,1)$ could be relaxed to arbitrarily-shaped inclusions with $C^2$ boundary if an asymptotic behavior of the form \eqref{asymp_H} can be guaranteed otherwise. Note that the proof of the asymptotic behavior of $K$ in \eqref{asympt_K} is independent of the shape of $\omega$.
		The second place where the shape of the inclusion influences the topological derivative is in the formula for the matrix $\myTDMat(\omega, \DT(U_0))$. Here, an extension to ellipse-shaped inclusions is possible.
	\end{remark}
	
	\subsection{Proofs} \label{sec_proofsTDNL}     
	

\textit{Proof of Lemma~\ref{lem:Hodd}:}. \label{lem:Hodd_proof}
\begin{proof}
	Let $H$ be the unique solution to $QH=0$ with $Q$ defined in \eqref{defQ}
	and $\tilde H$ a representative of the class $H$.
	Consider the transformation
	\begin{align*}
		\phi(x_1, x_2) &:= (-x_1, x_2) \quad  \\
		\mbox{with }J&=\nabla \phi(x_1, x_2) = \left[ \begin{array}{cc} -1 & 0 \\ 0 & 1  \end{array} \right] = J^{-\top} \quad \mbox{and}\quad |\mbox{det}J|=1,
	\end{align*}
	and define the function $\tilde H^s \in \HH(\RN)$ by
	\begin{align*}
		\tilde H^s(x_1, x_2) := - \tilde H(-x_1, x_2) = - (\tilde H \circ \phi)(x).
	\end{align*}
	We show that also $Q \, \tilde H^s = 0$. Thanks to the symmetry of $\omega$ with respect to the line $\lbrace x \in \RN: U_0 \cdot x = 0 \rbrace$, we have $\phi^{-1}(\RN \setminus \omega) = \RN \setminus \omega$ and $\phi^{-1}(\omega)=\omega$. Thus, we get
	\begin{align*}
		\langle &Q\, \tilde H^s, \test^s \rangle =\int_{\omega} \left[ \nu_0 \left(U_0 + J^{-\top}\nabla_y \hat{\tilde H}^s(y)\right) - \nuh(|U_0|)U_0 \right] \cdot J^{-\top}\nabla_y \hat{\test}^s(y) \mbox dy \\
		&+\int_{\RN \setminus \omega} \left[ \nuh\left( \left\lvert U_0 + J^{-\top}\nabla_y  \hat{\tilde H}^s(y)\right\rvert \right) \left(U_0 + J^{-\top}\nabla_y \hat{\tilde H}^s(y) \right) - \nuh\left( \left\lvert U_0 \right\rvert \right)U_0 \right] \cdot J^{-\top}\nabla_y \hat{\test}^s(y)\mbox dy,
	\end{align*}
	where $\hat{\tilde H}^s = \tilde H^s \circ \phi$ and $\hat{\test}^s = \test^s\circ \phi$.
	By definition of $\tilde H^s$ and due to the fact that $(\phi \circ \phi)(x) = x$, we have that $\hat{\tilde H}^s = -\tilde H$.
	Since the basis was chosen such that $e_1 = U_0 / |U_0|$ and $e_2 = e_1^\perp$, it holds that 
	\begin{align*}
		U_0 + J^{-\top}\nabla_y  \hat{\tilde H}^s(y) &= \left[ \begin{array}{c}
																|U_0| - \frac{\partial}{\partial y_1} (-\tilde H(y))\\
																0 + \frac{\partial}{\partial y_2} (-\tilde H(y))
															\end{array} \right], \\
		\left \lvert U_0 + J^{-\top}\nabla_y  \hat{\tilde H}^s(y) \right\rvert &= \left \lvert U_0 + \nabla_y  \tilde H(y) \right\rvert ,\\
		\left( U_0 + J^{-\top}\nabla_y  \hat{\tilde H}^s(y) \right) \cdot J^{-\top}\nabla_y \hat{\test}^s(y) &= \left(  U_0 + \nabla_y  \tilde H(y) \right) \cdot \nabla_y \left(-\hat{\test}^s(y) \right),\\
		U_0  \cdot J^{-\top}\nabla_y \hat{\test}^s(y) &= U_0   \cdot \nabla_y \left(-\hat{\test}^s(y) \right).
	\end{align*}
	Using these relations, we conclude that
	\begin{align*}
		\langle Q\, \tilde H^s, \test^s \rangle 
			=& \int_{\RN \setminus \omega} \left[ \nuh\left( \left \lvert U_0 + \nabla_y  \tilde H(y) \right\rvert \right) \left(  U_0 + \nabla_y  \tilde H(y) \right) - \nuh\left( \left\lvert U_0 \right\rvert \right)U_0 \right] \cdot \nabla_y \left(-\hat{\test}^s(y) \right) \mbox dy\\
			&+\int_{\omega} \left[ \nu_0 \left(  U_0 + \nabla_y  \tilde H(y) \right) - \nuh(|U_0|)U_0 \right] \cdot \nabla_y \left(-\hat{\test}^s(y) \right) \mbox dy\\
			=& \langle Q \, \tilde H, -\hat{\test}^s \rangle.
	\end{align*}
	Since $-\hat{\test}^s \in \HH(\RN)$ for $\test^s \in \mathcal H(\RN)$, \eqref{eq_QHzero} yields $\langle Q \tilde H,  -\hat{\test}^s \rangle = 0$ and therefore
	\begin{align*}
		\langle Q\, \tilde H^s, \test^s \rangle = 0 \quad  \mbox{ for all } \test^s \in \mathcal H(\RN).
	\end{align*}
	Thus, it follows from the uniqueness of a solution $H \in \mathcal H(\RN)$ to \eqref{variationApprox2} established in Proposition \ref{propExUnH} that also $\tilde H^s$ is a representative of $H$ and therefore $\tilde H^s(x_1, x_2) = \tilde H(x_1, x_2) + C$ for all $(x_1, x_2) \in \RN$ where $C$ is a constant.
	Restricted to the line $\lbrace x\in \RN: x_1 = 0 \rbrace$, this yields that $-2 \tilde H(0,x_2) = C$ for all $x_2 \in \mathbb R$. Thus, choosing the representative $\tilde H$ in such a way that $\tilde H(0,0) = 0$ yields that $C=0$ and thus $\tilde H^s = \tilde H$, which yields
	\begin{align*}
		\tilde H(-x_1,x_2) = -\tilde H(x_1, x_2)
	\end{align*}
	for all $(x_1, x_2) \in \RN$.
\end{proof} 

\textit{Proof of Proposition~\ref{prop:supersol1}:} \label{prop:supersol1_proof}
This proof follows the ideas of \cite{Bonnafe2013, AmstutzBonnafe2015}, but requires some different calculations.
	\begin{proof}
	Let $k\in(0,1]$ be defined depending on the lower bound $\delta_{\nuh} >-1/3$ from Assumption~\ref{assump_delta} as
	\begin{align}
		k = \begin{cases} \frac{1}{2} \left(-\frac{1}{\delta_{\nuh}(1+\varepsilon)} - 3 \right)  & -\frac{1}{3} < \delta_{\nuh} \leq -\frac{1}{5},  \\ 1 & \delta_{\nuh} > -\frac{1}{5}, \end{cases} \label{def_k_P}
	\end{align}
	for some $\varepsilon>0$.
	Furthermore, let $\overline \sigma$ be defined as
	\begin{align*}
		\overline \sigma := \mbox{min}\lbrace \sigma_0, \sigma_1, \sigma_2 \rbrace >1,	
	\end{align*}
	where $\sigma_0$ and $\sigma_1$ are given by
	\begin{align}
		\sigma_0 &:= 2, \nonumber\\
		\sigma_1 &:= 1 + \frac{ \numin }{\nu_0 + \tilde c' |U_0| \sqrt{5}} \in (1,2),\label{def_beta1}
	\end{align}
	with $\numin$ and $\tilde c'$ given by \eqref{physProperties} and \eqref{bound_nup}, respectively,
	and $\sigma_2$ is defined as the unique solution in $(1,2)$ to the equation $\tilde f(t) = \delta_{\nuh}$ with $\tilde f$ defined in \eqref{proof_P_ftilde} in the case where $\delta_{\nuh}<0$, and $\sigma_2:=2$ else.
	We show that property \eqref{QPgreaterEqualZero} holds for $k$ chosen according to \eqref{def_k_P} and any fixed $\sigma$ satisfying
	\begin{align}
		\sigma \in (1, \overline{\sigma}).\label{def_beta}
	\end{align}

	Let us first compute the first and second derivatives of the function $\myP$ given in \eqref{defP}. We use the notation $r = |x|$ and $e_r = x / |x|$. For $x \in \RN \setminus \overline \omega$, we have
	\begin{align}
		\myP(x) = \, & k\, (U_0 \cdot x) r^{-\sigma}, \nonumber\\
		\nabla \myP(x) \cdot \varphi = \,&  k\, r^{-\sigma} \left( U_0 - \sigma (U_0 \cdot e_r) e_r \right) \cdot \varphi,	 \label{nablaPext}	 \\
		\mathrm{D}^2 \myP(x)(\varphi, \psi) = \,  &  k\,  r^{-\sigma} \left(
		- \sigma \frac{1}{r^2} \left[ r (U_0 \cdot \psi) - (U_0 \cdot x) \frac{1}{r} (x\cdot \psi) \right] (e_r \cdot \varphi) \right. \nonumber\\
		& \left. \qquad \quad- \sigma (U_0 \cdot e_r) \frac{1}{r^2} \left[ r (\varphi \cdot \psi) - (\varphi \cdot x) \frac{1}{r} (x \cdot \psi) \right] \right) \nonumber \\
		& -  \sigma  \, k\, r^{-\sigma - 2} r (e_r \cdot \psi) (U_0\cdot \varphi - \sigma (U_0 \cdot e_r) (e_r \cdot \varphi) ),\nonumber
	\end{align}
	where we used that
	\begin{align*}
		\nabla(e_r \cdot \eta) = \nabla \left(\frac{x \cdot \eta}{|x|}\right) = \frac{1}{|x|^2}\left( |x| \eta - (x\cdot \eta) \frac{x}{|x|}\right)
	\end{align*}
	for $\eta \in \RN$.
	For $\psi = \varphi$ we get
	\begin{align}
		\mathrm{D}^2 \myP(x)(\varphi, \varphi) &=&  \sigma  \,k\, r^{-\sigma-2} \left( \vphantom{\frac{du}{mmy}} \right. & -r (U_0 \cdot \varphi)(e_r \cdot \varphi) &&+ (U_0 \cdot x)(e_r \cdot \varphi)^2 \nonumber \\
		&& & -r (U_0 \cdot e_r)(\varphi \cdot \varphi) &&+ (U_0 \cdot x) \frac{1}{r} (\varphi \cdot x)^2 \frac{1}{r} \nonumber\\
		&& & -r (e_r \cdot \varphi) (U_0 \cdot \varphi) && + r (e_r \cdot \varphi)^2 \sigma (U_0 \cdot e_r) \left. \vphantom{\frac{du}{mmy}} \right) \nonumber
	\end{align} \vspace{-5mm}
	\begin{align}
		\hspace{-8mm} = \sigma \,k\, r^{-\sigma-2} \lleft (\sigma+2) (U_0 \cdot x)(e_r \cdot \varphi)^2 - 2(x \cdot \varphi)(U_0\cdot\varphi) - (U_0 \cdot x)(\varphi \cdot \varphi) \rright. \label{proofP_D2P}
	\end{align}
	In particular, we obtain
	\begin{align}
		\Delta \myP(x) &= \sum_{i = 1}^{\myN} \mathrm{D}^2 \myP(x)(e_i, e_i) \nonumber \\
			&=  \sigma  \,k\,  r^{-\sigma-2} \lleft (\sigma+2) (U_0 \cdot x) - 2 (U_0 \cdot x) - \myN (U_0\cdot x) \rright \nonumber\\
			&= -  \sigma \,k \,  r^{-\sigma - 2} (U_0 \cdot x ) (\myN - \sigma) \label{proofP_LaplaceP}
	\end{align}
	where $e_i$ denotes the unit vector in Cartesian coordinates in direction $x_i$.

	Integration by parts yields
	\begin{align*} 
		\langle  Q\,  \myP, \test \rangle = \langle  Q_{int}\,  \myP, \test \rangle + \langle  Q_{trans}\,  \myP, \test \rangle + \langle  Q_{ext}\,  \myP, \test \rangle
	\end{align*}
	with
	\begin{align*}
		\langle  Q_{int}\,  \myP, \test \rangle &:= \int_{\omega} - \nu_0 \Delta \myP \, \test, \\
		\langle  Q_{trans}\,  \myP, \test \rangle &:= \int_{\partial \omega} \left[ - \nuh(|U_0 + (\nabla \myP)_{ext}|)(U_0 + (\nabla \myP)_{ext}) + \nu_0(U_0 + (\nabla \myP)_{int}) \right] \cdot n \, \test,\\
		\langle  Q_{ext}\,  \myP, \test \rangle &:= \int_{\RN \setminus \omega} - \ddiv \left( \nuh(|U_0 + \nabla \myP|) (U_0 + \nabla \myP) \right) \test,
	\end{align*}
	where $n$ denotes the unit normal vector pointing out of $\omega$.

	Thus, we have that $\langle Q \, \myP, \test \rangle \geq 0$ for all $\test \in \HH(\RN)$ with $\mbox{supp}(\test) \subset \RN_+$ such that $\test \geq 0$ almost everywhere, 
	if (and only if) the following three conditions are satisfied:
	\begin{align}
		- \nu_0 \Delta \myP \geq 0 \quad\; &\forall x \in \omega: \, U_0 \cdot x > 0, \label{Qint}\\
		\left[ - \nuh(|U_0 + (\nabla \myP)_{ext}|)(U_0 + (\nabla \myP)_{ext}) \right. \quad\; \qquad\quad& \nonumber\\
			\left.+ \nu_0(U_0 + (\nabla \myP)_{int}) \right] \cdot n \geq 0 \quad\; &\forall x \in  \partial \omega: \, U_0 \cdot x > 0, \label{Qtrans}\\
		-\ddiv \left( \nuh(|U_0 + \nabla \myP|) (U_0 + \nabla \myP) \right) \geq 0 \quad\;&\forall x \in \RN \setminus {\overline \omega}: \,  U_0 \cdot x > 0. \label{Qext}
	\end{align}	
	\begin{enumerate}
	 \item 
		The first condition \eqref{Qint} is satisfied by definition as $\myP$ is linear inside $\omega$ and thus $\Delta \myP = 0$ in $\omega$.
	
	\item
		Next, we investigate the transmission condition \eqref{Qtrans}. 
		For $x \in \partial \omega$, we have
		\begin{align*}
			(\nabla \myP)_{int} =  k\,  U_0
		\end{align*}
		and, by means of formula \eqref{nablaPext},
		\begin{align*}
			(\nabla \myP)_{ext} = k\, (U_0 - \sigma (U_0 \cdot x)x)
		\end{align*}
		because $r = |x| =1$.
		Using that
		\begin{align*}
			(U_0 + (\nabla \myP)_{int})\cdot n &= (1+k)(U_0 \cdot x) >0, \\
			\left(  -  (\nabla \myP)_{ext} + (\nabla \myP)_{int} \right) \cdot n &= \sigma \, k \, (U_0\cdot x) >0,  
		\end{align*}
		because $n(x) = x$ for $x\in \partial \omega$, and exploiting the physical property \eqref{propNuBounded}, we get
		\begin{align*}
		& \left(- \nuh(|U_0 + (\nabla \myP)_{ext}|)(U_0 + (\nabla \myP)_{ext}) + \nu_0(U_0 + (\nabla \myP)_{int})  \right) \cdot n\\
		& \geq\left( - \nuh(|U_0 + (\nabla \myP)_{ext}|)(U_0 + (\nabla \myP)_{ext}) + \nuh(|U_0 + (\nabla \myP)_{ext}|)(U_0 + (\nabla \myP)_{int})  \right) \cdot n\\
		& =\nuh(|U_0 + (\nabla \myP)_{ext}|)\left(-  (\nabla \myP)_{ext} + (\nabla \myP)_{int}  \right) \cdot n\\
		&\geq \numin\, \sigma \, k \,(U_0 \cdot x)  > 0.
		\end{align*} 
		for all $x \in \partial \omega$ with $(U_0 \cdot x)>0$.
	
	\item
		Now, we consider the exterior condition \eqref{Qext}. Let $(r,\theta)$ denote the polar coordinates of $x$ in the coordinate system aligned with $U_0$ such that 
        \begin{align*}
        	\mbox{cos}\, \theta = \frac{x}{|x|} \cdot \frac{U_0}{|U_0|}.
        \end{align*}
		We further introduce $\tilde \varphi = \tilde \varphi (r, \theta) := U_0 + (\nabla \myP)_{ext}$, and the auxiliary variables
		\begin{align} 
			q&= q(r) :=&& \hspace{-12mm} \sigma \,k\, r^{-\sigma}, \nonumber \\
			d_0 &= d_0(r) :=&&\hspace{-12mm} 1+ k\,r^{-\sigma}, & \nonumber\\
			d_1 &=d_1(r) :=&&\hspace{-12mm} 1+ k\,r^{-\sigma}(1-\sigma)
				& &\hspace{-12mm}= d_0 -   q, \label{def_d1}\\
			d_2 &= d_2(r, \theta) :=&& \hspace{-12mm} 1+ k\,r^{-\sigma}(1-\sigma \ccos^2 \theta)
				& &\hspace{-12mm}= d_0 - q \, \ccos^2 \theta, \label{def_d2} \\
			s &=s(r,\theta) :=&&\hspace{-12mm} \ssin^2 \theta \, d_0^2, \nonumber\\
			c &=c(r,\theta) :=&& \hspace{-12mm}\ccos^2 \theta \, d_1^2, \nonumber\\
			d &=d(r,\theta) :=&& \hspace{-12mm}d_1 \, d_2\nonumber.
		\end{align}
		Note that all of these symbols are actually functions of $r$ and possibly~$\theta$. For better presentation, we will drop these dependencies for the rest of the proof.
		It can be seen that
		\begin{align}
			e_r \cdot \tilde \varphi & = |U_0|\, \ccos \theta \, d_1, \nonumber\\
			x \cdot \tilde \varphi & = (U_0 \cdot x)\, d_1, \nonumber\\
			U_0 \cdot\tilde \varphi & =  |U_0|^2 \, d_2, \nonumber\\
			|\tilde \varphi|^2 &= |U_0|^2 \, \left( s+ c \right). \label{normPhiTilde}
		\end{align}
		
		Note that $d_0$ and $d_1$ are positive because of $\sigma < \sigma_0 = 2$, $k\in(0,1]$ and $r>1$. This implies that $|\tilde \varphi| > 0$ since not both $\ssin \theta$ and $\ccos \theta$ can vanish at the same time. 
		These relations, together with \eqref{proofP_D2P} and \eqref{proofP_LaplaceP}, yield that
		\begin{align*}
			-\ddiv &\left( \nuh(|U_0 + \nabla \myP|) (U_0 + \nabla \myP) \right) =- \ddiv \left( \nuh(|\tilde \varphi|) \tilde \varphi \right)\\
			&=- \left( \nuh (|\tilde \varphi|)  \Delta \myP + \frac{\nuh'(|\tilde \varphi|)}{|\tilde \varphi|} \mathrm{D}^2 \myP(\tilde \varphi, \tilde \varphi) \right)\\
			&=-  \sigma \, k \, r^{-\sigma-2} (U_0 \cdot x) \left( \nuh (|\tilde \varphi|) (\sigma - \myN) + \nuh'(|\tilde \varphi|) \frac{1}{|\tilde \varphi|} |U_0|^2 f(r, \theta ) \right)
		\end{align*}
		with
		\begin{align}
			f(r, \theta) =  (\sigma+ 1) \, c- 2 \,d - s. \label{proofP_f}
		\end{align}	
		Thus, condition \eqref{Qext} is satisfied if 
		\begin{align} \label{integrandQext}
			\nuh (|\tilde \varphi|) (\sigma - \myN) + \nuh'(|\tilde \varphi|)  \frac{1}{|\tilde \varphi|} |U_0|^2 f(r, \theta ) \leq 0
		\end{align}
		for all $x \in \RN \setminus \overline \omega$ with $U_0 \cdot x >0$, i.e., for all $(r, \theta)$ with $r>1, \, \ccos \theta \in (0,1)$.
		We distinguish three different cases:

		\textbf{Case 0}: The spatial coordinates $(r, \theta)$ are such that $\nuh'(|\tilde \varphi|) = 0$:\\
		Condition \eqref{integrandQext} is satisfied since $\sigma$ is smaller than $\sigma_0=\myN $ due to \eqref{def_beta}, and since $\nuh(|\tilde \varphi|)>0$ by physical property \eqref{propNuBounded}.

		\textbf{Case 1}: The spatial coordinates $(r, \theta)$ are such that $\nuh'(|\tilde \varphi|) > 0$:\\
		We insert \eqref{normPhiTilde} and \eqref{proofP_f} into the left hand side of \eqref{integrandQext} and get \label{prop:supersol1:pageref3Case1}
		\begin{align*}
			\nuh (|\tilde \varphi|) &(\sigma - \myN) + \nuh'(|\tilde \varphi|) \frac{1}{|\tilde \varphi|} |U_0|^2 f(r, \theta )  \\
			&=   \nuh (|\tilde \varphi|) (\sigma - \myN) + \nuh'(|\tilde \varphi|)|U_0| \frac{1}{ \sqrt{s+c}} ((\sigma+1)c-2d-s)  \\ 
			&\leq \nuh (|\tilde \varphi|) (\sigma - \myN) + \nuh'(|\tilde \varphi|)|U_0| \frac{((\sigma-1)c-s)}{ \sqrt{s+c}}   \\
			&=   \nuh (|\tilde \varphi|) (\sigma - \myN) + \nuh'(|\tilde \varphi|)|U_0| \left(\frac{\sigma c}{ \sqrt{s+c}} - \frac{c+s}{\sqrt{c+s}} \right)  \\
			&\leq   \nuh (|\tilde \varphi|) (\sigma - \myN) + \nuh'(|\tilde \varphi|)|U_0| \left(\frac{\sigma (c + s)}{ \sqrt{s+c}} - \sqrt{c+s} \right)  \\
			&=     \nuh (|\tilde \varphi|) (\sigma - \myN) + \nuh'(|\tilde \varphi|)|U_0| (\sigma-1) \sqrt{c+s},
		\end{align*}	
		where we used that
		\begin{align*}
			-2\,d = - 2 d_1 \, d_2 \leq -2 \, d_1^2 \leq -2 \, d_1^2 \, \ccos^2 \theta =  -2\,c.
		\end{align*}

		So, \eqref{integrandQext} holds if $\sigma$ satisfies
		\begin{align*}
			&\nuh (|\tilde \varphi|) (\sigma - \myN) + \nuh'(|\tilde \varphi|)|U_0| (\sigma-1) \sqrt{c+s}  \leq 0 \\
			&\Leftrightarrow \sigma( \nuh (|\tilde \varphi|)  + \nuh'(|\tilde \varphi|)|U_0| \sqrt{c+s})  \leq \myN \nuh (|\tilde \varphi|) + \nuh'(|\tilde \varphi|)|U_0| \sqrt{c+s} \\
			&\Leftrightarrow \sigma \leq \frac{\myN \nuh (|\tilde \varphi|) + \nuh'(|\tilde \varphi|)|U_0| \sqrt{c+s}}{ \nuh (|\tilde \varphi|)  + \nuh'(|\tilde \varphi|)|U_0| \sqrt{c+s}} = 1 + \frac{ \nuh (|\tilde \varphi|)}{ \nuh (|\tilde \varphi|)  + \nuh'(|\tilde \varphi|)|U_0| \sqrt{c+s}}.
		\end{align*}
		Since $\sigma \in (1, \sigma_1)$ with $\sigma_1$ defined in \eqref{def_beta1},	
		the above inequality is satisfied because		
		\begin{align*}
			\frac{ \nuh (|\tilde \varphi|)}{ \nuh (|\tilde \varphi|)  + \nuh'(|\tilde \varphi|)|U_0| \sqrt{c+s}} \geq 
			\frac{ \numin}{ \nuh (|\tilde \varphi|)   + \nuh'(|\tilde \varphi|)|U_0| \sqrt{c+s}} \geq 
			\frac{ \numin}{ \nu_0  + \tilde c' |U_0| \sqrt{5}},
		\end{align*}
		where we used property \eqref{propNuBounded} as well as the facts that $0<\nuh'(|\tilde \varphi|)$ by assumption, $\nuh'(|\tilde \varphi|) \leq \tilde c'$ by \eqref{bound_nup}, and $0\leq c \leq 1$, $0 \leq s \leq (1+k)^2 \leq 4$ for $\sigma\in (1,2)$ and $k\in (0,1)$ noting that $r>1$.
		Thus, choosing $\sigma$ according to \eqref{def_beta}, condition \eqref{Qext} is satisfied at points where $\nuh'(|\tilde \varphi|) > 0$.
			 
		\textbf{Case 2}: The spatial coordinates $(r, \theta)$ are such that $\nuh'(|\tilde \varphi|) < 0$:
		
		\underline{Case 2a:} $f(r, \theta ) \geq 0$: Condition \eqref{integrandQext} is satisfied since $\sigma <  \sigma_0 = \myN$ due to \eqref{def_beta} because both summands are non-positive. 
		
		\underline{Case 2b:} $f(r, \theta ) < 0$:\\
		In this case, we must show that the positive contribution of the second summand on the left hand side of \eqref{integrandQext} is compensated by the negative first term. This is possible if Assumption \ref{assump_delta} holds.
		
		We introduce $g(r, \theta ) := (|U_0|^2/|\tilde \varphi|^2) \, f(r,\theta)$, such that condition \eqref{integrandQext} can be rewritten as
		\begin{align*}
			\nuh (|\tilde \varphi|) (\sigma - \myN) + \nuh'(|\tilde \varphi|)|\tilde \varphi| \, g( r, \theta ) \leq 0,
		\end{align*}
		and find a lower bound $\underline g$ for $g(r, \theta )$,
		\begin{align*}
			\underline g \leq g(r, \theta) <0, \quad \forall r>1, \, \theta \in \left( -\frac{\pi}{2}, \frac{\pi}{2} \right).
		\end{align*}
		Then, since $\nuh'(|\tilde \varphi|) < 0$, it holds that 
		\begin{align*}
			\nuh (|\tilde \varphi|) (\sigma - \myN) + \nuh'(|\tilde \varphi|)|\tilde \varphi| \, g( r, \theta ) \leq \nuh (|\tilde \varphi|) (\sigma - \myN) + \nuh'(|\tilde \varphi|)|\tilde \varphi| \, \underline g
		\end{align*}
		and \eqref{integrandQext} follows if the right hand side of the given estimate is non-positive. The condition that the right hand side of the expression above is non-positive is equivalent to the condition
		\begin{align}
			\frac{\nuh'(|\tilde \varphi|)|\tilde \varphi|}{ \nuh (|\tilde \varphi|)}  &\geq \frac{\myN - \sigma}{\underline g}. \label{proofP_condftilde}
		\end{align}

		Let us now investigate the expression $g( r, \theta )$ and find a bound $\underline g$ from below. Using \eqref{normPhiTilde} and \eqref{proofP_f}, we have
		\begin{align*}
			g(r, \theta ) = \frac{(\sigma+ 1) \, \ccos^2 \theta \, d_1^2- 2 \, d_1 \, d_2 - \ssin^2 \theta \, d_0^2}{ d_0^2\,  \ssin^2 \theta + d_1^2 \, \ccos^2 \theta} =:  \frac{g_1}{g_2}.
		\end{align*}
		Rewriting the nominator $g_1$ in terms of $d_0$ using \eqref{def_d1} and \eqref{def_d2}, we get
		\begin{align*}
			g_1 =& (\sigma+1)\ccos^2 \theta (d_0^2 - 2 q d_0 + q^2) - 2(d_0^2 - q d_0 - q \ccos^2 \theta d_0 + q^2 \ccos^2 \theta) - \ssin^2 \theta d_0^2 \\
			=& d_0^2( (\sigma+2)\ccos^2 \theta - 3) + 2\,d_0 \, q(1- \sigma \ccos^2 \theta) + q^2 \ccos^2 \theta  (\sigma-1).
		\end{align*}
		Similarly, we get for the denominator $g_2$,
		\begin{align*}
			g_2 &= d_0^2\,  \ssin^2 \theta + d_1^2 \, \ccos^2 \theta\\
				&= d_0^2 - 2\, d_0 \, q \ccos^2 \theta + q^2 \ccos^2 \theta.
		\end{align*}
		Note that $g_2$ is positive and, therefore, $g_1$ must be negative by the assumption of Case 2b.
		Together, we get
		\begin{align*}
			\frac{g_1}{g_2} &= \frac{d_0^2( (\sigma+2)\ccos^2 \theta - 3) + 2\,d_0 \, q(1- \sigma \ccos^2 \theta) + q^2  \ccos^2 \theta  (\sigma-1)}{d_0^2 - 2\,d_0\, q \ccos^2 \theta  + q^2 \ccos^2 \theta} \\
			&= -\frac{g_2}{g_2} + \frac{1}{g_2}\left( d_0^2((\sigma+2)\ccos^2 \theta - 2) + 2\, d_0\,  q  (1 - \ccos^2 \theta (1+\sigma) ) + q^2 \ccos^2 \theta \sigma			\right) \nonumber \\
			&= -1 + \frac{1}{g_2}\left( \ccos^2\theta \left[d_0^2(\sigma+2) -2\,(1+\sigma) \, d_0\,  q + q^2 \sigma \right] -2 d_0 (d_0 - q)\right). \nonumber 
		\end{align*}
		Note that, for $r>1$ and $\sigma >1$, we have
		\begin{align*}
			0 \leq r^{-\sigma} \leq 1  \quad \mbox{ and } \quad
			0 \leq \ccos^2 \theta \leq 1,
		\end{align*}
		and thus, for $\sigma \in (1,\sigma_0)$ and $k\in (0,1]$,
		\begin{align*}
			1 \leq d_0 \leq 1+k,  \quad \mbox{ and } \quad 0 \leq q \leq \sigma \,k,  \quad \mbox{ and } \quad 0< 1-k(\sigma-1)  \leq d_0 - q = d_1 \leq 1.
		\end{align*}
		Hence, we can see that
		\begin{align*}
			d_0^2(\sigma+2) -(1+\sigma) d_0\, 2 q + q^2 \sigma = (d_0 - q)^2 + (\sqrt{\sigma}\,d_0 - \sqrt{\sigma}\, q)^2 + (d_0+q)(d_0-q) >0,
		\end{align*}
		and we can estimate
		\begin{align}
			\frac{g_1}{g_2} 
			&\geq -1 - \frac{2 d_0 (d_0 - q)}{g_2} \geq -1 - \frac{2(1+k)}{g_2}. \nonumber 
		\end{align}
		For the denominator $g_2$, it can be seen that
		\begin{align*}
			g_2 &= d_0^2\,  \ssin^2 \theta + d_1^2 \, \ccos^2 \theta\\
				&\geq   \ssin^2 \theta + (1-k(\sigma-1))^2 \, \ccos^2 \theta\\
				&= 1 + \ccos^2 \theta \left( (1-k(\sigma-1))^2 -1 \right) \\
				&\geq 1 + \left( (1-k(\sigma-1))^2 -1 \right) = (1-k(\sigma-1))^2,
		\end{align*}
		because $(1-k(\sigma-1))^2 - 1 < 0$ due to $\sigma < \sigma_0 = \myN$ and $k\in (0,1]$, and thus	
		\begin{align}
			g(r,\theta) = \frac{g_1}{g_2}&\geq  -1 - \frac{2(1+k)}{(1-k(\sigma-1))^2 } =  - \frac{(1-k(\sigma-1))^2+2(1+k)}{(1-k(\sigma-1))^2 } =: \underline g. \nonumber
		\end{align}
		Note that $\underline g$ depends on $\sigma$ and $k$. For $t\in (1,2)$, we define
		\begin{align}
			\tilde f(t) := -\frac{(2-t)\,(1-k(t-1))^2}{(1-k(t-1))^2+2(1+k)} , \label{proof_P_ftilde}
		\end{align}
		and note that $\tilde f(\sigma) = (2-\sigma)/\underline g$. If now $\sigma$ satisfies that 
		\begin{align*}
			\tilde f(\sigma) \leq  \underset{s>0}{\mbox{inf }} \frac{\nuh'(s) s}{\nuh(s)} = \delta_{\nuh},
		\end{align*}
		then \eqref{proofP_condftilde} is satisfied, which yields \eqref{integrandQext} and, therefore, \eqref{Qext} in Case 2b. 
		
		If $\delta_{\nuh}$ is non-negative, this is satisfied because $\tilde f(t)<0$ for any $t \in (1,2)$ since $k\in (0,1]$, and \eqref{proofP_condftilde} and thus also \eqref{Qext} holds because $\sigma < \sigma_2=2$ in this case.
		
		In the case where $\delta_{\nuh}$ is negative, recall that $\delta_{\nuh}>-1/3 $ by Assumption \ref{assump_delta}, so we have $-1/3 < \delta_{\nuh}<0$. 
		In \eqref{def_k_P}, we defined $k$ in such a way that $\tilde f(1) = -1/(3+2k) = \delta_{\nuh}(1+\varepsilon)< \delta_{\nuh}$.
		Since $\tilde f(1) < \delta_{\nuh}$, $\tilde f(2) = 0 > \delta_{\nuh}$ and since it can be seen that $\tilde f$ is continuous and increasing in the interval $(1,2)$, there exists a unique $\sigma_2\in (1,2)$ such that $\tilde f(\sigma_2) = \delta_{\nuh}$ and it holds that $\tilde f(t) <\delta_{\nuh}$ for all $t \in (1, \sigma_2)$. Thus, if $\sigma \in (1, \sigma_2)$, inequality \eqref{proofP_condftilde} is satisfied, which yields \eqref{integrandQext} and thus \eqref{Qext}.
	\end{enumerate}
	Hence, choosing $\sigma$ and $k$ according to \eqref{def_beta} and \eqref{def_k_P}, respectively, yields the statement of Proposition \ref{prop:supersol1}.
	\end{proof}

\textit{Proof of Proposition~\ref{prop:subsol}:} \label{prop:subsol_proof}

	\begin{proof}
		The proof is similar to the proof of Proposition \ref{prop:supersol1}. 
		Again, we define $\overline{\hat{\sigma}}$ as
		\begin{align*}
			\overline{\hat{\sigma}} := \mbox{min}\lbrace \hat \sigma_0, \hat \sigma_1, \hat \sigma_2 \rbrace >1,	
		\end{align*}
		where $\hat \sigma_0$ and $\hat \sigma_1$ are given by
		\begin{align}
			\hat\sigma_0 &:= 2, \nonumber\\
			\hat\sigma_1 &:= 1 + \frac{ \numin }{\nu_0 + \tilde c' |U_0| \sqrt{5}} \in (1,2),\nonumber
		\end{align}
        with the constant $\tilde c'$ defined in \eqref{bound_nup}. If the bound $\delta_{\nuh}$ from Assumption \ref{assump_delta} is non-negative, we define $\hat \sigma_2 :=2$. Otherwise, if $\delta_{\nuh}<0$, let
        \begin{align}
			\hat k(t) &:= \frac{\numin - \nu_0}{\nu_0 + \numin (t-1)}, \nonumber\\
			\hat{\tilde f}^{(\mbox{i})}(t) &:= -\frac{(2-t)(1+ \hat k(t))^2}{(1+ \hat k(t))^2 + 2(1+ \hat k(t)(1-t))}, \label{proofR_ftilde_i}\\ 
			\hat{\tilde f}^{(\mbox{ii})}(t) &:= 	-\frac{(2-t)(1+ \hat k(t))^2}{(1+ \hat k(t))^2 + (1+ \hat k(t)(1-t))^2}, 		\label{proofR_ftilde_ii}
		\end{align}
        be mappings from $[1,2]$ to $\mathbb R$. Note that, for $j \in \lbrace \mbox{i}, \mbox{ii} \rbrace$ we have that $\hat{\tilde f}^{(j)}(2) = 0$ and 
        \begin{align*}
        	-\frac{(1+\hat k(1))^2}{(1+\hat k(1))^2+1}=\hat{\tilde f}^{(\mbox{ii})}(1)<\hat{\tilde f}^{(\mbox{i})}(1)=-\frac{(1+\hat k(1))^2}{(1+\hat k(1))^2+2} = \delta_{\nuh}^{R_2} < \delta_{\nuh} < 0,
        \end{align*}       
        with $\delta_{\nuh}^{R_2}$ defined in Assumption \ref{assump_delta}. It can be seen that $\hat{\tilde f}^{(\mbox{i})}$ and $\hat{\tilde f}^{(\mbox{ii})}$ are continuous. Moreover, $\hat{\tilde f}^{(\mbox{i})}$ and $\hat{\tilde f}^{(\mbox{ii})}$ are monotonically increasing in the interval $(1,2)$ which can be seen as follows: Note that
\begin{align*}
		(\hat{\tilde f}^{(\mbox{i})})'(1) &= \frac{2 q+1}{\left(2 q^2+1\right)^2} > 0, \mbox{ and }
		(\hat{\tilde f}^{(\mbox{i})})''(t) = \frac{8 q \left(q^3-1\right)}{(2 q (q+t-1)+t)^3} > 0 \quad \forall \, t\in (1,2),
	\end{align*}
    with $q=\nu_0 / \underline{\nu}>1$. Thus, it holds that $(\hat{\tilde f}^{(\mbox{i})})'(t) > 0$ for all $t \in (1,2)$.	
	For $\hat{\tilde f}^{(\mbox{ii})}$, we get ${(\hat{\tilde f}^{(\mbox{ii})})'(t) = \frac{1}{q^2+1} > 0}$ for all $t\in (1,2)$.
        
Thus, for $j \in \lbrace \mbox{i}, \mbox{ii} \rbrace$, there exists a unique solution in $(1,2)$ to the equation $\hat{\tilde f}^{(j)}(t) = \delta_{\nuh}$ which we denote by $\hat \sigma_2^{(j)}$ for $j \in \lbrace \mbox{i}, \mbox{ii} \rbrace$. In this case we define $\hat \sigma_2 := \mbox{min} \lbrace \hat \sigma_2^{(\mbox{i})}, \hat \sigma_2^{(\mbox{ii})} \rbrace$.
       
		We show that property \eqref{QRgreaterEqualZero} holds for any fixed $\sigma$ satisfying
		\begin{align}
			\sigma \in (1, \overline{\hat\sigma}).\label{def_betaR}
		\end{align}
		
		Similarly to the proof of Proposition \ref{prop:supersol1}, we have to show the three conditions
		\begin{align}
			-\nu_0 \Delta \myRsub \leq 0 &\quad \forall x \in \omega: \,U_0 \cdot x > 0 \label{QRint}\\
			\left[ - \nuh(|U_0 + (\nabla \myRsub)_{ext}|)(U_0 + (\nabla \myRsub)_{ext}) \right.\quad \; & \nonumber\\
				+ \nu_0(U_0 + (\nabla \myRsub)_{int}) \left. \right] \cdot n \leq 0 &\quad \forall x \in  \partial \omega: \, U_0 \cdot x > 0 \label{QRtrans}\\
			-\ddiv \left( \nuh(|U_0 + \nabla \myRsub|) (U_0 + \nabla \myRsub) \right) \leq 0 &\quad \forall x \in \RN \setminus { \overline \omega}: \, U_0 \cdot x > 0 \label{QRext}
		\end{align}
		where $n$ denotes the unit normal vector pointing out of $\omega$.
		\begin{enumerate}
		 \item 
		As in the proof of Proposition \ref{prop:supersol1}, it is easily seen that the first condition \eqref{QRint} is trivially satisfied as $ \myRsub$ is linear inside $\omega$.
				
		\item
		Next we consider the transmission condition \eqref{QRtrans}. Exploiting that, for $x \in \partial \omega$ with $\omega = B(0,1)$, the outward unit vector $n$ is equal to $x$ and $|x|=r=1$, and noting the formulas for the gradient of $\myRsub$ inside and outside the inclusion $\omega$,
		\begin{align*}
			(\nabla \myRsub)_{int} &= k \, U_0,\\
			(\nabla \myRsub)_{ext} &= k\, r^{-\sigma} \left( U_0 - \sigma (U_0 \cdot e_r) e_r \right),
		\end{align*}
		we obtain 
		\begin{align*}
			\left( -\nuh(|U_0 \right.&\left.+ (\nabla \myRsub)_{ext}|)(U_0 + (\nabla \myRsub)_{ext}) + \nu_0(U_0 + (\nabla \myRsub)_{int}) \right) \cdot n \\
			&= ( - \nuh(|U_0 + \nabla \myRsub_{ext}|)(1+k(1-\sigma)) +\nu_0 (1+k) )\,(U_0 \cdot x) \\
			&\leq ( - \numin(1+k(1-\sigma)) + \nu_0 (1+k) ) \, (U_0 \cdot x) \\
			&=  \left(  - \numin\frac{\sigma\,\nu_0}{\nu_0 + \numin(\sigma-1)}) + \nu_0 \frac{\sigma \, \numin }{ \nu_0 + \numin (\sigma-1)}\right) \, (U_0 \cdot x) = 0.
		\end{align*}
		In the estimation, we used that $1+k(1-\sigma)>0$ since $k\in(-1,0)$ and $\sigma > 1$.
		
		\item
		For the exterior condition \eqref{QRext}, we need to verify that
		\begin{align*}
			-\ddiv \left( \nuh(|U_0 + \nabla \myRsub|) (U_0 + \nabla \myRsub) \right) \leq 0  \quad \forall x \in \RN \setminus {\overline \omega}: U_0 \cdot x > 0.
		\end{align*}
		Again, let $(r,\theta)$ denote the polar coordinates of $x$ in the coordinate system aligned with $U_0$ such that 
        \begin{align*}
        	\mbox{cos}\, \theta = \frac{x}{|x|} \cdot \frac{U_0}{|U_0|}.
        \end{align*}
        
        Again, for better readability, we introduce the symbols
		\begin{align*}
			\hat q &= \hat q (r):=&& \hspace{-19mm} k\,\sigma r^{-\sigma}\\
			\hat d_0 &= \hat d_0 (r):=&&\hspace{-19mm} 1 +k\, r^{-\sigma}, \\
			\hat d_1 &= \hat d_1(r) :=&& \hspace{-19mm}1+k\, r^{-\sigma}(1-\sigma) &&\hspace{-19mm}= \hat d_0 -\hat q, \\
			\hat d_2 &= \hat d_2(r, \theta) :=&&\hspace{-19mm} 1+  k\, r^{-\sigma}(1-\sigma \ccos^2 \theta)& &\hspace{-19mm}= \hat d_0 -\hat q \ccos^2 \theta,  \\
			\hat s &= \hat s(r, \theta) :=&&\hspace{-19mm} \ssin^2 \theta \, \hat d_0^2, \\
			\hat c &= \hat c(r, \theta):= &&\hspace{-19mm} \ccos^2 \theta \, \hat d_1^2, \\
			\hat d &= \hat d(r,\theta) :=&& \hspace{-19mm}\hat d_1 \, \hat d_2,
		\end{align*}
		and drop the dependencies on $r$ and $\theta$. Note that, due to $k <0$, the symbols introduced above are not the same as the corresponding symbols used in the proof of Proposition \ref{prop:supersol1}.
		Analogously to the proof of Proposition \ref{prop:supersol1}, we introduce the notation $\hat{\tilde \varphi}=\hat{\tilde \varphi}(r,\theta):= U_0 + (\nabla \myRsub)_{ext}$ and get the relations
		\begin{align*}
			e_r \cdot \hat{\tilde \varphi} & = |U_0| \, \ccos \theta \, \hat d_1, \\
			x \cdot \hat{\tilde \varphi} & = (U_0 \cdot x)\, \hat d_1, \\
			U_0 \cdot \hat{\tilde \varphi} & =  |U_0|^2 \, \hat d_2, \\
			|\hat{\tilde \varphi}|^2 &= |U_0|^2 \left( \hat s+ \hat c \right).
		\end{align*}
		Again, we can deduce
		\begin{align*}
			-\ddiv &\left( \nuh(|U_0 + \nabla \myRsub|) (U_0 + \nabla \myRsub) \right)\\
				&=- \left( \nuh (|\hat{\tilde \varphi}|)  \Delta \myRsub + \frac{\nuh'(|\hat{\tilde \varphi}|)}{|\hat{\tilde \varphi}|} \mathrm{D}^2 \myRsub(\hat{\tilde \varphi}, \hat{\tilde \varphi}) \right)\\
				&= - \sigma\,k\, r^{-\sigma-2} (U_0 \cdot x) \left( \nuh (|\hat{\tilde \varphi}|) (\sigma - \myN) + \nuh'(|\hat{\tilde \varphi}|) |\hat{\tilde \varphi}| \frac{1}{|\hat{\tilde \varphi}|^2} |U_0|^2 \hat f(r, \theta ) \right),
		\end{align*}
		with the function $\hat f$ defined as
		\begin{align*}
			\begin{aligned}
				\hat f(r, \theta ) = & (\sigma+ 1) \hat c - 2 \, \hat d_1\, \hat d_2 - \hat s.
			\end{aligned}
		\end{align*}
		Thus, since $k<0$, it again suffices to show that
		\begin{align}
			 \nuh (|\hat{\tilde \varphi}|) (\sigma - \myN) + \nuh'(|\hat{\tilde \varphi}|) \frac{1}{|\hat{\tilde \varphi}|} |U_0|^2 \hat f(r, \theta )  \leq 0 \label{integrandQRext}
		\end{align}
		for all $x \in \RN \setminus \overline \omega$ with $U_0 \cdot x >0$, i.e., for all $(r, \theta)$ with $r>1, \, \ccos \theta \in (0,1)$.
		Again, we distinguish three different cases:
		
		\textbf{Case 0}: $\nuh'(|\hat{\tilde \varphi}|) = 0$: Estimate \eqref{integrandQRext} obviously holds for $\sigma <  \hat \sigma_0 = 2$ since  $\nuh(|\hat{\tilde \varphi}|)>0$ by physical property \eqref{propNuBounded}.
		
		\textbf{Case 1}: $\nuh'(|\hat{\tilde \varphi}|) > 0$:\\		
		Also for $k \in (-1,0)$ and $\sigma > 1$, it holds that $-2\hat d \leq -2 \hat c$ since $d_1>0$ and
		\begin{align*}
			\hat d_2 = 1 +k\, r^{-\sigma} -\sigma\, k\, r^{-\sigma}  \ccos^2\theta &\geq (1+k\, r^{-\sigma})\ccos^2\theta -\sigma \,k\, r^{-\sigma} \ccos^2\theta\\
			&= \ccos^2\theta \hat d_1,
		\end{align*}
		because $r>1$. Thus, we can perform the analogous estimations as in the proof of Proposition \ref{prop:supersol1}.
		Using that, for $\sigma \in (\Nhalf, \myN)$ and $k\in(-1,0)$, we have
		\begin{align*}
			0 \leq \hat c \leq (1-k)^2 \leq 4, \qquad 0 \leq \hat s \leq 1, \quad \mbox{ and } \quad 0 \leq \sqrt{\hat s+\hat c} \leq \sqrt{5},
		\end{align*}
		we again conclude that \eqref{integrandQRext} and thus \eqref{QRext} hold since $\sigma < \hat{\sigma}_1$ due to \eqref{def_betaR}. 
		
		\textbf{Case 2}: $\nuh'(|\hat{\tilde \varphi}|) < 0$:
		
		\underline{Case 2a:} $\hat f(r, \theta ) \geq 0$: Estimate \eqref{integrandQRext} holds for any $\sigma \in (\Nhalf, \myN)$ because both summands are non-positive. 
		
		\underline{Case 2b:} $\hat f(r, \theta ) < 0$:\\
		Analogously to the proof of Proposition \ref{prop:supersol1}, we can introduce
		\begin{align*}
			g^{(2)}( r, \theta ) := \frac{|U_0|^2}{|\hat{\tilde \varphi}|^2} \, \hat f(r,\theta) = \frac{1}{\hat s + \hat c}  \, \hat f(r,\theta),
		\end{align*}
		and rewrite condition \eqref{integrandQRext} as
		\begin{align*}
			\nuh (|\hat{\tilde \varphi}|) (\sigma - \myN) + \nuh'(|\hat{\tilde \varphi}|) |\hat{\tilde \varphi}|  g^{(2)}(r, \theta ) \leq 0.
		\end{align*}
		Again, we have to find a lower bound on the expression
		\begin{align*}
			g^{(2)}( r, \theta ) = \frac{(\sigma+ 1) \, \ccos^2 \theta \, \hat d_1^2- 2 \, \hat d_1 \, \hat d_2 - \ssin^2 \theta \, \hat d_0^2}{ \hat d_0^2\,  \ssin^2 \theta + \hat d_1^2 \, \ccos^2 \theta} =:  \frac{g_1^{(2)}}{g_2^{(2)}}
		\end{align*}
		which satisfies condition \eqref{proofP_condftilde}.
		The manipulations of the terms $g_1^{(2)}$ and $g_2^{(2)}$ are analogous to the proof of Proposition \ref{prop:supersol1} and we arrive at the corresponding expression
		\begin{align*}
			\frac{g_1^{(2)}}{g_2^{(2)}} 
			= -1 + \frac{1}{g_2^{(2)}} &\left[ \ccos^2 \theta \left(\hat d_0^2(\sigma+2) - 2\, (1+\sigma) \,\hat d_0 \,\hat q   + \hat q^2 \, \sigma \right) -2 \hat d_0 (\hat d_0 - \hat q) \right] \\
			= -1 + \frac{1}{g_2^{(2)}} &\left[ \ccos^2 \theta \left( (\hat d_0 - \hat q)^2 + (\sqrt{\sigma} \, \hat d_0 - \sqrt{\sigma} \, \hat q)^2 + (\hat d_0 +  \hat q)(\hat d_0 - \hat q) \right) \right. \\
			& \left. \;\, -2 \, \hat d_0 \,(\hat d_0 - \hat q) \right] .
		 \end{align*}
		 Again, we will estimate this expression from below such that we can extract a condition on $\sigma>1$ that is sufficient for \eqref{integrandQRext}.
		 For the estimation, we will use that, for $r>1$, $\sigma \in (1, \hat \sigma_0)$ and $k\in(-1,0)$, we have
		 \begin{align*}
			\sigma \, k &\leq \hat q \leq 0,\\
			1+k &\leq \hat d_0 \leq 1, \\
			1 &\leq \hat d_0 - \hat q = \hat d_1 \leq 1+k\, (1-\sigma),\\
			1+k\, (1+\sigma) &\leq   \hat d_0 + \hat q \leq 1.
		 \end{align*}

		 Note that, for the denominator $g_2^{(2)}$, we have
		 \begin{align}
			g_2^{(2)} &= \hat d_0^2\,  \ssin^2 \theta + \hat d_1^2 \, \ccos^2 \theta \geq (1+k)^2 \ssin^2 \theta +  \ccos^2 \theta \nonumber \\
					& = \ccos^2 \theta\left(1-(1+k)^2\right) + (1+k)^2 \geq (1+k)^2 . \label{g22_lower}
		 \end{align}

		For the estimation of $g_1^{(2)}/g_2^{(2)}$, we need to distinguish two more cases:
		
		\underline{Case 2b (i)}: The spatial coordinates $(r,\theta)$ are such that $\hat d_0 + \hat q \geq 0$:\\
		In this case, recalling that $g_2^{(2)}>0$ since $-1<k<0$, we can estimate the above expression from below by dropping the positive cosine term and, taking into account \eqref{g22_lower}, we get
		\begin{align*}
			\frac{g_1^{(2)}}{g_2^{(2)}} &\geq -1 - \frac{2 \, \hat d_0 \, (\hat d_0 - \hat q) }{g_2^{(2)}} \geq -1 - \frac{2 (1+k\, (1-\sigma) )}{g_2^{(2)}} \\
				&\geq -1 - \frac{2 (1+k\, (1-\sigma) )}{(1+k)^2} =  \frac{-(1+k)^2-2 (1+k\, (1-\sigma) )}{(1+k)^2} =: \underline g^{(\mbox{i})}.
		\end{align*}

		\underline{Case 2b (ii)}: The spatial coordinates $(r,\theta)$ are such that $\hat d_0 + \hat q < 0$:\\
		In this case, we get the estimate
		\begin{align*}
			\frac{g_1^{(2)}}{g_2^{(2)}} &\geq -1 + \frac{\ccos^2 \theta (\hat d_0 + \hat q)(\hat d_0 - \hat q)-2 \, \hat d_0 \, (\hat d_0 - \hat q) }{g_2^{(2)}} \\
			 &\geq -1 + \frac{(\hat d_0 + \hat q)(\hat d_0 - \hat q)-2 \, \hat d_0 \, (\hat d_0 - \hat q) }{g_2^{(2)}} 
			 = -1 + \frac{- (\hat d_0 - \hat q)^2 }{g_2^{(2)}}  \\
			 &\geq = -1 -\frac{ (1+k\, (1-\sigma))^2 }{g_2^{(2)}}  
			\geq  \frac{-(1+k)^2-(1+k\, (1-\sigma))^2}{(1+k)^2} =: \underline g^{(\mbox{ii})}.
		\end{align*}

		Recall that $\sigma$ and $k$ are fixed numbers only depending on the given data $\numin$, $\nu_0$ and $\delta_{\nuh}$, and on $|U_0|$. Note that $\hat k(\sigma) = k$ and, for $j \in \lbrace \mbox{i}, \mbox{ii}\rbrace$, it holds $\hat{\tilde f}^{(j)}(\sigma) = (2- \sigma)/ \underline g^{(j)}$ for $j \in \lbrace \mbox{i}, \mbox{ii}\rbrace$ with $\hat{\tilde f}^{(\mbox{i})}$ and $\hat{\tilde f}^{(\mbox{ii})}$ defined in \eqref{proofR_ftilde_i} and \eqref{proofR_ftilde_ii}, respectively. Due to the monotonicity of $\hat{\tilde f}^{(\mbox{i})}$ and $\hat{\tilde f}^{(\mbox{ii})}$, we see that, in the case where $\delta_{\nuh} < 0$, the condition $ \sigma< \hat{\sigma}_2 = \mbox{min}\lbrace \hat{\sigma_2}^{(\mbox{i})},\hat{\sigma_2}^{(\mbox{ii})} \rbrace$ implies $\hat{\tilde f}^{(j)}( \sigma) < \delta_{\nuh}$ for $j \in \lbrace \mbox{i}, \mbox{ii} \rbrace$. Thus, for $ \sigma \in (1, \hat \sigma_2)$ and $j \in \lbrace \mbox{i},\mbox{ii} \rbrace$, we get
        \begin{align*}
        	\frac{2-\sigma}{g^{(2)}} \leq \frac{2-\sigma}{\underline{g}^{(j)}} = \hat{\tilde f}^{(j)}( \sigma) < \delta_{\nuh} \leq \frac{\nuh'(|\hat{\tilde \varphi}|)\, |\hat{\tilde \varphi}|}{\nuh(|\hat{\tilde \varphi}|)}
       \end{align*}
       which is equivalent to \eqref{integrandQRext} and thus \eqref{QRext}. 
       
       If $\delta_{\nuh}$ is non-negative, the inequalities above hold for all $ \sigma \in (1,2)$ since $\hat{\tilde f}^{(j)}$ is increasing in $(1,2)$ and $\hat{\tilde f}^{(j)}(2) = 0$.
		\end{enumerate}
		Again, the overall statement of Proposition \ref{prop:subsol} follows because $\sigma \in (1, \mbox{min} \lbrace \hat\sigma_0, \hat\sigma_1, \hat\sigma_2 \rbrace)$.
	\end{proof}


\section{Topological asymptotic expansion: case II} \label{sec:CaseII}

In a similar way to Section \ref{sec:TopAsympExp_CaseI}, it is possible to derive the topological derivative also in the case where we perturb a domain of linear material by an inclusion of nonlinear material, see Fig. \ref{fig:OmegaPertUnpert_caseII}. Since, here, the material outside the inclusion behaves linearly, the result corresponding to Theorem \ref{theo:asymp_H} about the asymptotic behavior of the variation of the direct state at scale 1 simplifies significantly. The behavior at infinity can be treated in the same way as it was done for the variation of the adjoint state at scale 1 in Proposition \ref{prop454}. The rest of the derivation is analogous to Section \ref{sec:TopAsympExp_CaseI}, see \cite{Gangl2017} for more details.

\begin{figure} 
	\begin{minipage}{.5 \textwidth}
		\begin{tabular}{c}    
			\hspace{15mm}\includegraphics[scale=0.55]{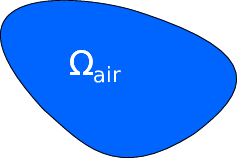}
		\end{tabular}
	\end{minipage}
	\hspace{1.2cm}
	\begin{minipage}{.5 \textwidth}
		\begin{tabular}{c}    
			\hspace{-7mm}\includegraphics[scale=0.55]{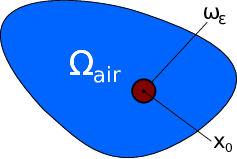}
		\end{tabular}
	\end{minipage}
	\caption[Illustration of topological derivative in Case II]{Left: Unperturbed configuration for Case II. Right: Perturbed configuration for Case II. }
	\label{fig:OmegaPertUnpert_caseII}	
\end{figure}

\subsection{Main result in case II}
Here, we only state the main result in Case II, i.e., the topological derivative for the introduction of nonlinear (ferromagnetic) material inside a region of linear material (air) according to the definition \eqref{def:topDerivative}. We introduce the notation to be used in the statement of Theorem \ref{theo_MainResultCaseII}:
	\begin{itemize}
		\item $x_0 \in \Omega^d$ denotes the point around which we perturb the material coefficient;
		\item $\uii \in H_0^1(\holdAll)$ is the unperturbed direct state, i.e., the solution to 
        \begin{align*}
	\mbox{Find } u_0^{(2)} \in H^1_0(\holdAll): \int_{\holdAll} \nu_0 \nabla u_0^{(2)} \cdot \nabla \test  = \langle F, \test \rangle \quad \forall \test \in H^1_0(\holdAll),
\end{align*}
and $U_0^{(2)} = \nabla \uii(x_0)$;
		\item $\adjii \in H_0^1(\holdAll)$ is the unperturbed adjoint state, i.e., the solution to 
        \begin{align*}
	&\mbox{Find } \vii \in H^1_0(\holdAll): \int_{\holdAll} \nu_0 \nabla \vii \cdot \nabla \test = -\langle \tilde G , \test \rangle \quad  \forall \test \in H^1_0(\holdAll),
\end{align*}
with $\tilde G$ according to \eqref{expansion_J}, and $\Adjzii = \nabla \adjii(x_0)$;
		\item $\myTDMat^{(2)} = \myTDMat^{(2)}(\omega, \DT(U_0^{(2)}))$ denotes the matrix defined in \eqref{TDMat_CaseII} where $\omega$ represents the shape of the inclusion and $\DT$ is the Jacobian of $T$ defined in \eqref{def_T};
		\item $H^{(2)} \in \HH(\RN)$ denotes the variation of the direct state at scale 1, i.e., the solution to
        \begin{align*} 
		\mbox{Find } \Hii \in \mathcal H(\RN):
        &\int_{\RN\setminus \omega} \nu_0 \Hii  \cdot \nabla \test + \int_{\omega} ( T(\Uoii + \nabla \Hii) - T( \Uoii) ) \cdot \nabla \test \\
		&= \int_{\omega} (\nu_0 - \nuh(|\Uoii|)) \Uoii \cdot \nabla \test \quad \forall \test \in \mathcal H(\RN);
\end{align*}
		\item $K^{(2)} \in \HH(\RN)$ denotes the variation of the adjoint state at scale 1, i.e., the solution to
        \begin{align} 
        \begin{aligned}
        \label{variationAdjoint_scale1_ii}
	\mbox{Find } \Kii \in \mathcal H(\RN):
	&\int_{\RN \setminus \omega} \nu_0 \nabla \Kii \cdot \nabla \test + \int_{\omega} \DT(\Uoii)\nabla \Kii \cdot \nabla \test \\
    &=  \int_{\omega} (\nu_0 I -\DT(\Uoii) )  \Voii \cdot \nabla \test \quad \forall \test \in \mathcal H(\RN);
    \end{aligned}
\end{align}
		\item $\delta_J$ is according to \eqref{expansion_J}.
	\end{itemize}
	\begin{theorem} \label{theo_MainResultCaseII}
		Assume that 
		\begin{itemize}
		 \item[-] $\omega = B(0,1)$ the unit disk in $\RN$
		 \item[-] the ferromagnetic material is such that Assumptions \ref{assump_BH} and \ref{assump_c78} are satisfied,
		 \item[-] the functional $\mathcal J_{\varepsilon}$ satisfies Assumption \ref{assump_J},
		 \item[-] the unperturbed direct state $\uii$ satisfies $\uii \in C^{1,\beta}$ for some $\beta>0$,
		 \item[-] the unperturbed direct state $\adjii$ satisfies $\adjii \in C^{1,\tilde \beta}$ for some $\tilde \beta>0$.
		\end{itemize}
		Then the topological derivative for introducing air inside ferromagnetic material reads
		\begin{align}
			\begin{aligned}
			G^{air \rightarrow f}(x_0) 
			=& (\Uoii)^\top \, \myTDMat^{(2)} \, \Voii \\
				&+  \int_{\omega} S_{\Uoii}(x,\nabla H^{(2)}) \cdot (\Voii + \nabla \Kii) + \delta_J. \label{G2_final}	
			\end{aligned}
		\end{align}
	\end{theorem}


\section{Polarization matrices} \label{sec_Appendix}

In this section we present a way to explicitly compute the matrices $\myTDMat$ and $\myTDMat^{(2)}$ arising in the first terms of the topological derivatives \eqref{G1_final} and \eqref{G2_final}, respectively, based on the notion of anisotropic polarization tensors \cite[Sect. 4.12]{AmmariKang2007}.
 
Let $\omega$ be the unit ball, $\omega = B(0,1)$, and let the conductivities in $\omega$ and in $\RN \setminus \overline \omega$ be denoted by the symmetric and positive definite $2\times 2$ matrices $\tilde A$ and $A$, respectively. We further assume that the matrix $\tilde A - A$ is either positive definite or negative definite.

From \cite[Def. 4.29]{AmmariKang2007} we obtain the first order anisotropic polarization tensor in two dimensions
\begin{align}
	\polMat(A, \tilde A; \omega):=(\polMat_{ij})_{i,j=1,2}. \label{def_PolMat}
\end{align}
where, for $i,j\in \lbrace 1,2 \rbrace$,
\begin{align}
	\polMat_{ij} = \int_{\omega} (\tilde A - A) e_j \cdot \nabla \theta_i(A, \tilde A; \omega), \label{appendix_Mij}
\end{align}
and $\theta_i(A, \tilde A; \omega)$ is the solution to the transmission problem
\begin{align}
	\left \lbrace
	\begin{aligned} \label{transmissionThetai}
	&\nabla \cdot (A \nabla \theta_i) = 0 \; \mbox{ in } \RN \setminus \overline \omega,\\
	&\nabla \cdot (\tilde A \nabla \theta_i) = 0 \; \mbox{ in } \omega, \\
	&\theta_i|_-  - \theta_i|_+ = x_i  \; \mbox{ on } \partial \omega, \\
	&n \cdot \tilde A \nabla \theta_i |_- - n \cdot A \nabla \theta_i|_+ = n \cdot A e_i  \; \mbox{ on } \partial \omega, \\
	&\theta_i - \frac{1}{2 \pi \sqrt{\mbox{det}(A)} } \mbox{ln} \| A^{-1/2}x\| \int_{\partial \omega} \theta_i(y) \mbox d\sigma(y) \rightarrow 0 \mbox{ as } |x| \rightarrow \infty.
	\end{aligned} \right.
\end{align}

For the case where $A=I$ and $\omega$ is an ellipse which is aligned with the coordinate system, an explicit formula for the polarization matrix is available:
\begin{proposition}[\hspace{1sp}\cite{AmmariKang2007}, Proposition 4.31] \label{prop_PolMatExplicit_AK}
	If $\omega$ is an ellipse whose semi-axes are aligned with the $x_1$- and $x_2$-axes and of length $a$ and $b$, respectively, then the first-order APT, $\polMat(I, \tilde A; \omega)$, takes the form
	\begin{align}
		\polMat(I, \tilde A; \omega) = |\omega| \left( I + (\tilde A - I)(\frac{1}{2}I-C) \right)^{-1} (\tilde A - I),  \label{formulaM_ellipse}
	\end{align}
	with the matrix
	\begin{align} \label{prop_PolMatExplicit_AK_C}
		C = \frac{a-b}{2(a+b)} \left[ \begin{array}{cc}
		                               1 & 0 \\
		                               0 & -1
		                              \end{array}  \right].
	\end{align}
\end{proposition}
In particular, if $\omega$ is a disk, then
\begin{align}
	\polMat(I,\tilde A; \omega) = 2 |\omega| (\tilde A +I)^{-1} (\tilde A - I).	\label{formulaM_disk}
\end{align}
Furthermore, we will use the following relation:
\begin{lemma}[\hspace{1sp}\cite{AmmariKang2007}, Lemma 4.30] \label{lem:RotPolMat}
	For any unitary transformation $R$, the following holds:
		\begin{align*}
			\polMat(A, \tilde A; \omega) = R \, \polMat( R^\top A R, R^\top \tilde A R; R^{-1} \omega) \, R^\top.
		\end{align*}
\end{lemma}

In order to apply Proposition \ref{prop_PolMatExplicit_AK} to the case of an anisotropic background conductivity $A$, we need to perform a change of variables such that the background conductivity $A$ becomes the identity.
We can show following relation:
\begin{lemma} \label{lem_coordinateTrans}
	Let $\omega$ be bounded with smooth boundary, $A, \tilde A \in \mathbb R^{2\times 2}$ positive definite, symmetric and such that $\tilde A - A$ is either positive definite or negative definite. Let the polarization matrix $\polMat(A, \tilde A; \omega)$ be defined by \eqref{def_PolMat}, \eqref{appendix_Mij} and \eqref{transmissionThetai}. Then it holds
	\begin{align}
		\polMat(A, \tilde A; \omega) = \mbox{det}(A^{1/2})\, (A^{1/2})^\top \, \polMat(I, A^{-1/2}\tilde A A^{-1/2}; A^{-1/2}\omega) \, A^{1/2} . \label{polMatMcal}
	\end{align}
\end{lemma}
\begin{proof}
We only give a sketch of the proof here and refer the reader to \cite{Gangl2017} for more details. The idea of the proof is to consider \eqref{transmissionThetai} in a weak form by noting that, for $\Adjz= (\adj_1, \adj_2)^\top$,
\begin{align}
	\theta(A, \tilde A; \omega;\Adjz) &:= \sum_{i=1}^2 \adj_i \, \theta_i(A, \tilde A; \omega) = K(A, \tilde A; \omega; \Adjz) + \chi_{\omega} (\Adjz \cdot x), \label{appendix_theta_K}
\end{align}
where $K(A, \tilde A; \omega; \Adjz)$ is the solution to the transmission problem
\begin{align}
	&\mbox{Find } K \in \mathcal H(\RN) \mbox{ such that} \nonumber\\
	&\int_{\RN \setminus \omega} A \, \nabla K \cdot \nabla \test \,\mbox dy + \int_{\omega} \tilde A \, \nabla K \cdot \nabla \test \,\mbox dy = - \int_{\omega} (\tilde A - A) \, \Adjz \cdot \nabla \test \, \mbox dy \; \forall \test \in \mathcal H(\RN). \label{transmissionK}
\end{align}
The coordinate transformation $x = A^{1/2} y$ yields that
\begin{align}
	\theta (A, \tilde A; \omega; \Adjz)(A^{1/2} y) &= \theta (I, A^{-1/2}\tilde A A^{-1/2}; A^{-1/2} \omega; A^{1/2} \Adjz)(y), \label{theta_coordTransform}
\end{align}
and, therefore, 
\begin{align*}
	\Adjz^\top \, \polMat(A, \tilde A; \omega) \, U_0 &= \int_{\omega} \left( (\tilde A - A) \, U_0 \right) \cdot \nabla_x \theta(A, \tilde A; \omega;\Adjz)(x) \, \mbox d x \\
	&= \mbox{det}(A^{1/2}) \Adjz ^\top \, (A^{1/2})^\top \, \polMat(I, A^{-1/2}\tilde A A^{-1/2}; A^{-1/2} \omega) \, A^{1/2} U_0
\end{align*}
for $\Adjz$, $U_0 \in \mathbb R^2$.
\end{proof}

\begin{remark} \label{rem_Psymm}
	The polarization tensor \eqref{def_PolMat} is well-known to be symmetric, see e.g., \cite{AmmariKang2007}. The symmetry of $\polMat(A, \tilde A; \omega)$ for symmetric matrices $A, \tilde A$ can also be seen from \eqref{polMatMcal} noting that the matrix \eqref{formulaM_ellipse} is symmetric.
\end{remark}

\subsection{Case I} \label{sec_AppendixCaseI}
Now we are in the position to derive the polarization matrix $\polMat(A, \tilde A; \omega)$ with $A=\DT(U_0)$ as defined in \eqref{DTW_eigenvalues} and $\tilde A = \nu_0 I$, which we will later use to rewrite the first term of the topological derivative $J_1$ \eqref{def_J1}. Recall the representation of $\DT(U_0)$ \eqref{DTW_eigenvalues}, i.e.,
\begin{align} \label{DTU0_eigenvalues}
	\DT(U_0) = R \left[ \begin{array}{cc} \eigvTwo & 0 \\ 0 & \eigvOne \end{array} \right] R^\top
\end{align}
with $\eigvOne:= \lambda_1(|U_0|) = \nuh(|U_0|)$ and $\eigvTwo:= \lambda_2(|U_0|) = \nuh(|U_0|) + \nuh'(|U_0|)|U_0|$. Using Lemma \ref{lem_coordinateTrans}, Proposition \ref{prop_PolMatExplicit_AK} and Lemma \ref{lem:RotPolMat}, we get the following result:

\begin{proposition}
	Let $\omega = B(0,1)$ the unit disk in $\mathbb R^2$, and let Assumption \ref{assump_BH} hold. Then, we get
	\begin{align} \label{prop_polMat_CaseI}
		\polMat(\DT(U_0), \nu_0 I;  \omega) &= |\omega| \, R \,  
	 \left[  \begin{array}{cc} \frac{ \left( \eigvTwo +\sqrt{\eigvOne\eigvTwo}\right)(\nu_0 - \eigvTwo) }{\nu_0  +\sqrt{\eigvOne\eigvTwo}  } & 0 \\ 0&  \frac{\left(\eigvOne + \sqrt{\eigvOne\eigvTwo}\right)(\nu_0-\eigvOne) }{\nu_0 +\sqrt{\eigvOne\eigvTwo}}  \end{array} \right] \, R^{\top}.
	\end{align}
\end{proposition}
\begin{proof}
	The result follows by application of Lemma \ref{lem_coordinateTrans}, Proposition \ref{prop_PolMatExplicit_AK} and Lemma \ref{lem:RotPolMat}. Note that, due to \eqref{physProperties} which follows from Assumption \ref{assump_BH}, it holds $\eigvOne>0$, $\eigvTwo>0$, $\nu_0 - \eigvOne>0$ and $\nu_0 - \eigvTwo > 0$, and, therefore, $A$ and $\tilde A - A$ are positive definite. For details of the proof, we refer the reader to \cite{Gangl2017}.
    \end{proof}

Now, we finally obtain an explicit form of the matrix $\myTDMat$ in \eqref{j1_final}.
\begin{theorem}
	Let $\omega = B(0,1)$ and let Assumptions \ref{assump_BH}, \ref{assump_J}, \ref{assump_c78}, \ref{assump_delta}, \ref{assump_reg_u0} and \ref{assump_reg_v0} hold. Then the matrix $\myTDMat = \myTDMat(\omega, \DT(U_0))$ in the topological derivative \eqref{G1_final} reads
    \begin{align}
  	  \myTDMat&=  (\nu_0 - \eigvOne) |\omega| \, R
	 \left[  \begin{array}{cc} \frac{  \eigvTwo +\sqrt{\eigvOne\eigvTwo} }{\nu_0  +\sqrt{\eigvOne\eigvTwo}  } & 0 \\ 0&  \frac{\eigvOne + \sqrt{\eigvOne\eigvTwo} }{\nu_0 +\sqrt{\eigvOne\eigvTwo}}  \end{array} \right] R^{\top}, \label{TDMat_CaseI}
     \end{align}
     where $\eigvOne$, $\eigvTwo$ and $R$ are according to \eqref{DTU0_eigenvalues}.    
\end{theorem}
\begin{proof}
	We start out from the definition of $J_1 = J_1(U_0, \Adjz)$ in \eqref{def_J1}. Note that \eqref{variationAdjoint_scale1} is the same as \eqref{transmissionK} with $A=\DT(U_0)$, $\tilde A = \nu_0 I$ and $e_i$ replaced by $\Adjz$. Therefore, $K$ appearing in \eqref{def_J1} equals $K(\DT(U_0), \nu_0 I;\omega; \Adjz)$ according to \eqref{transmissionK}. By \eqref{appendix_theta_K} and \eqref{appendix_Mij}, we get
\begin{align*}
	J_1(U_0, \Adjz) &= (\nu_0 - \eigvOne) U_0 \cdot \int_{\omega} \Adjz + \nabla K \\
	&= (\nu_0 - \eigvOne) U_0^\top (\nu_0 I - \DT(U_0))^{-1}  \int_{\omega}(\nu_0 I - \DT(U_0)) ( \nabla \theta(\DT(U_0), \nu_0 I; \omega; \Adjz)) \\
	&= (\nu_0 - \eigvOne)  U_0 ^\top \,(\nu_0 I - \DT(U_0))^{-1} \, \polMat(\DT(U_0), \nu_0 I; \omega)^\top \, \Adjz ,
\end{align*}
where $\polMat(\DT(U_0), \nu_0 I; \omega)$ is given in \eqref{prop_polMat_CaseI}.
Thus, exploiting the symmetry of $\polMat(\DT(U_0), \nu_0 I; \omega)$ according to Remark \ref{rem_Psymm}, we finally get \eqref{J1_polMat} with the matrix
\begin{align}
	\myTDMat &= (\nu_0 - \eigvOne )(\nu_0 I - \DT(U_0))^{-1} \polMat(\DT(U_0), \nu_0 I; \omega) \nonumber \\
	  &= (\nu_0 - \eigvOne) |\omega| \, R
	 \left[  \begin{array}{cc} \frac{  \eigvTwo +\sqrt{\eigvOne\eigvTwo} }{\nu_0  +\sqrt{\eigvOne\eigvTwo}  } & 0 \\ 0&  \frac{\eigvOne + \sqrt{\eigvOne\eigvTwo} }{\nu_0 +\sqrt{\eigvOne\eigvTwo}}  \end{array} \right] R^{\top}. 
\end{align} 
\end{proof}

Note that, in the linear case where $\eigvTwo = \eigvOne>0$, it holds $(\nu_0 - \eigvOne)(\nu_0 I - \DT(U_0))^{-1}  = I$, and we obtain
\begin{align*}
	\myTDMat = \polMat(\DT(U_0), \nu_0 I; \omega)
	 &=2 \, |\omega| \,  \eigvOne \, \frac{  \nu_0 - \eigvOne  }{\nu_0 + \eigvOne  } I,
\end{align*}
which coincides with the well-known formula derived in, e.g., \cite{Amstutz2006}.
Thus, here, unlike in the nonlinear case, the matrix appearing in the topological derivative is actually the polarization matrix according to \cite{AmmariKang2007}.

\begin{remark} \label{rem_MatExplicit_CaseI}
	Finally, we remark that the explicit form of the matrix $\myTDMat$ satisfying relation \eqref{J1_polMat} can also be obtained directly without exploiting Proposition \ref{prop_PolMatExplicit_AK} in the following way: Starting out from the transmission problem defining $K$ \eqref{transmissionK}, after a coordinate transformation $x=\DT(U_0)^{1/2} y$ we can compute the solution $K$ explicitly by a special ansatz similarly to \cite[Proposition 4.6]{AmmariKang2007}. Noting that, by the coordinate transformation the circular inclusion $\omega$ becomes an ellipse $\tilde \omega$, we make the ansatz in elliptic coordinates. For that purpose, let $\tilde R \in \mathbb R$ and $(r,\varphi) \in \mathbb R_0^+ \times [0, 2\pi]$ be such that $x_1(r,\varphi) = \tilde R \, \ccos \varphi \ccosh r$ and $x_2(r,\varphi) = \tilde R \, \ssin \varphi \ssinh r$. For $i=1,2$ we make the ansatz
	\begin{align*}
		K_{e_1}(r, \varphi) = \begin{cases}
								a_1 \tilde R \, \ccos \varphi \, \ccosh r & \mbox{in } \tilde \omega, \\
								b_1 \mathrm{e}^{-r} \ccos\varphi  &  \mbox{in } \RN \setminus \overline{\tilde \omega },
							  \end{cases}
		\quad
		K_{e_2}(r, \varphi) = \begin{cases}
										a_2 \tilde R \, \ssin \varphi \, \ssinh r & \mbox{in } \tilde \omega, \\
										b_2 \mathrm{e}^{-r} \ssin \varphi &  \mbox{in } \RN \setminus \overline{\tilde \omega },
									\end{cases}
	\end{align*}
	for the transformed version of problem \eqref{transmissionK} involving the unit vector $e_i$, and choose the constants $a_i$, $b_i$ such that $K_{e_i}$ is continuous and satisfies the correct interface jump condition on $\partial \tilde \omega$ which are incorporated in the variational formulation for $i=1,2$. For a given $\Adjz \in \RN$, the solution $K$ to \eqref{variationAdjoint_scale1} is then obtained as a linear combination of $K_{e_1}$ and $K_{e_2}$. Plugging in this explicit solution K into \eqref{J1_polMat}, the matrix $\myTDMat$ can be identified. In particular, the behavior of $K$ as $|x| \rightarrow \infty$, cf. \eqref{asympt_K}, can be seen from this explicit formula.
\end{remark}

\subsection{Case II} \label{sec_AppendixCaseII}
In this case, we compute the polarization matrix $\polMat(A, \tilde A; \omega)$ with $A= \nu_0 I$ and $\tilde A = \DT(\Uoii)$. Again, using Lemma \ref{lem_coordinateTrans}, Proposition \ref{prop_PolMatExplicit_AK} and Lemma \ref{lem:RotPolMat}, we obtain the following result:

\begin{proposition} \label{prop_polMatii}
	Let $\omega = B(0,1)$ the unit disk in $\mathbb R^2$, and let Assumption \ref{assump_BH} hold. Then, we have
	\begin{align*}
		\polMat( \nu_0 I, \DT(\Uoii);  \omega) &=2 |\omega| \nu_0 \, R \, 
	\left[ \begin{array}{cc} \frac{\eigvTwo - \nu_0}{\eigvTwo + \nu_0} & 0\\ 0 & \frac{\eigvOne- \nu_0}{\eigvOne + \nu_0} \end{array} \right] \, R^{\top},
	\end{align*}
    where $\eigvOne = \hat\nu(|\Uoii|)$, $\eigvTwo =\hat\nu(|\Uoii|) + \hat\nu'(|\Uoii |)|\Uoii |$, and $R$ denotes the rotation matrix around the angle between $\Uoii $ and the x-axis such that
\begin{align*}
	\Uoii = R \binom{|\Uoii|}{0}.
\end{align*}
\end{proposition}
\begin{proof}
	This case is simpler since, here, the material coefficient outside the inclusion is proportional to the identity matrix. Therefore, the inclusion remains a disk even after the corresponding coordinate transformation. The result follows by application of Lemma \ref{lem_coordinateTrans} and Proposition \ref{prop_PolMatExplicit_AK}, noting that $\DT(\Uoii)$ is positive definite and $\DT(\Uoii) - \nu_0 I$ is negative definite due to Assumption \ref{assump_BH}.
\end{proof}

In the same way as in Case I, we obtain an explicit formula for the matrix $\myTDMat^{(2)}$ in the topological derivative \eqref{G2_final}.
\begin{theorem}
	Let $\omega = B(0,1)$ and let Assumptions \ref{assump_BH}, \ref{assump_J}, \ref{assump_c78}, \ref{assump_delta}, \ref{assump_reg_u0} and \ref{assump_reg_v0} hold. Then the matrix $\myTDMat^{(2)} = \myTDMat^{(2)}(\omega, \DT(\Uoii))$ in the topological derivative \eqref{G2_final} reads
    \begin{align}
  	  \myTDMat^{(2)}&=   2 |\omega| \nu_0 \, R
	 \left[ \begin{array}{cc} \frac{\eigvOne - \nu_0}{\eigvTwo + \nu_0} & 0\\ 0 & \frac{\eigvOne - \nu_0}{\eigvOne + \nu_0} \end{array} \right] R^{\top} . \label{TDMat_CaseII}
     \end{align}
     where $\eigvOne$, $\eigvTwo$ and $R$ are as in Proposition \ref{prop_polMatii}.
\end{theorem}

\begin{remark} \label{rem_MatExplicit_CaseII}
	Similarly to Remark \ref{rem_MatExplicit_CaseI}, also here we can compute the solution to transmission problem \eqref{variationAdjoint_scale1_ii} explicitly by making a special ansatz. Unlike in Case I, here the conductivity matrix outside the inclusion is a scaled identity matrix, $A = \nu_0 I$, and the circular inclusion $\omega$ does not become an ellipse. Therefore, the solution can be obtained by the following ansatz:
	\begin{align*}
			\Kii_{e_1}(x_1, x_2)= 
			\begin{cases}
									a_1^{(2)} \, x_1& \mbox{ in } \omega, \\
									b_1^{(2)} \, \frac{x_1}{x_1^2+x_2^2} & \mbox{ in } \RN \setminus \overline{\omega},
			                      \end{cases}
			                     \quad 
			\Kii_{e_2}(r,\varphi)= 
			\begin{cases}
									a_2^{(2)} \, x_2 & \mbox{ in } \omega, \\
									b_2^{(2)} \, \frac{x_2}{x_1^2+x_2^2}  & \mbox{ in } \RN \setminus \overline{\omega} .
			                      \end{cases}
 	\end{align*}    
	Again, the constants $a_i^{(2)}$, $b_i^{(2)}$ must be chosen such that the interface conditions are satisfied, and the matrix $\myTDMat^{(2)}$ can be identified. 
\end{remark}


\section{Computational aspects}\label{Sec:Numerics}
	In order to make use of formulas \eqref{G1_final} and \eqref{G2_final} in applications of shape and topology optimization, an efficient method to evaluate these formulas for every point $x_0$ in the design region of the computational domain is of utter importance. For the rest of this section, we restrict our presentation to Case I, noting that analogous results hold true for Case II.
	
	In particular, the evaluation of the second term $J_2$ in \eqref{G1_final} seems to be computationally very costly, as it involves the solutions $H$ and $K$ to the transmission problems \eqref{variationApprox2} and \eqref{variationAdjoint_scale1}, respectively. Both of these problems are defined on the unbounded domain $\RN$ and depend on $U_0 = \nabla u_0 (x_0)$, i.e., the gradient of the unperturbed direct state $u_0$ evaluated at the point of interest $x_0$. In addition, problem \eqref{variationAdjoint_scale1} also depends on $\Adjz = \nabla \adjz(x_0)$, i.e., the gradient of the unperturbed adjoint state $\adjz$ at point $x_0$.
	Recall the second term $J_2$ defined in \eqref{def_J2},
	\begin{align}
		J_2 = J_2(x_0) =J_2(U_0, \Adjz) = \int_{\RN} \tilde S_{U_0}(x,\nabla H(U_0) ) \cdot (\Adjz + \nabla K(U_0, \Adjz)), \label{J2_U0_V0}
	\end{align}
	where $\tilde S$ is defined in \eqref{def_Stil_I}, $H = H(U_0)$ is the solution to problem \eqref{variationApprox2},
	and $K = K(U_0, \Adjz)$ the solution to \eqref{variationAdjoint_scale1}.
	At the first glance, this means that, for each point $x_0 \in \Omega^d$ where one wants to evaluate the term $J_2$, one has to solve problems \eqref{variationApprox2} and \eqref{variationAdjoint_scale1} in order to get the value for $J_2$. Topology optimization algorithms which are based on topological sensitivities usually require the values of these sensitivities at all points of the design domain $\Omega^d$ simultaneously, which would, of course, result in extremely inefficient optimization algorithms. This enormous computational effort can be reduced with the help of the following observations.
	\begin{lemma} \label{propertiesJ2}
		Let $U_0$, $\Adjz \in \RN$, $R \in \mathbb R^{\myN \times \myN}$ an orthogonal matrix. Let $H(U_0) \in \HH(\RN)$ be the solution to \eqref{variationApprox2} and $K(U_0, \Adjz)$ the solution to \eqref{variationAdjoint_scale1} for given $U_0$, $\Adjz$. Let further $J_2 = J_2(U_0, \Adjz)$ be defined by \eqref{J2_U0_V0}. 
		Then the following properties hold:
		\begin{enumerate}
			\item $J_2$ is linear in the second argument, i.e., for all $a,b \in \mathbb R$ and $U_0, \Adj_1, \Adj_2 \in \RN$, \label{statementJ2linearV}
				\begin{align}
					J_2(U_0, a\,\Adj_1 + b \Adj_2) = a \, J_2(U_0,\Adj_1) + b\,J_2(U_0, \Adj_2). \label{eqnJ2linearV}
				\end{align}
			\item Let $y=R\,x$. For the solution $H = H(U_0)$ to \eqref{variationApprox2}, we have \label{statementHInvar}
				\begin{align}
					R^\top\, \nabla_y H(U_0)( y) &= \nabla H(R^\top \, U_0)( x). \label{nablaHInvar}
				\end{align}
			\item Let $y=R\,x$. For the solution $K = K(U_0, \Adjz)$ to \eqref{variationAdjoint_scale1}, we have \label{statementKInvar}
				\begin{align}
					R^\top\, \nabla_y K(U_0, \Adjz)( y) &= \nabla K(R^\top U_0, R^\top \, \Adjz)( x).  \label{nablaKInvar}
				\end{align}
			\item It holds \label{statementJ2rotInvar}
				\begin{align}	
					J_2(R^\top U_0, R^\top \Adjz ) = J_2(U_0, \Adjz). \label{eqnJ2rotInvar}
				\end{align}
		\end{enumerate}
	\end{lemma}

	\begin{proof}
		\begin{enumerate}
			\item It can easily be seen from \eqref{variationAdjoint_scale1} that $K$ depends linearly on $\Adjz$.
			
			\item For $x \in \RN$, let $H(U_0)( R\,x)$ the solution to problem \eqref{variationApprox2} after a coordinate transformation. Define $\tilde H(U_0)( x) := H(U_0)( y(x))$, $x \in \RN$, with $y(x) = R\,x$. Then we have
			\begin{align*}
				\nabla_y H(U_0)( y) = R^{-\top} \nabla_x \tilde H(U_0)( x) = R \nabla_x \tilde H(U_0)( x),
			\end{align*}
			since $R$ is orthogonal. Similarly, for a test function $\test$ and $x \in \RN$, we define $\tilde \test(x) := \test(y(x))$ and get
			\begin{align*}
				\nabla_y \test(y) = R \nabla_x \tilde \test(x).
			\end{align*}
			For the left hand side of transmission problem \eqref{variationApprox2} we obtain
			\begin{align}
				\int_{\RN} &\left( \Tt(y,U_0 + \nabla_y H(y)) - \Tt(y,U_0) \right)\cdot \nabla_y \test(y) \, \mbox d y \nonumber \\
                \label{lhs_coordTrans} 
                =& \int_{\RN} \left( \Tt(x,R^\top U_0 + \nabla_x \tilde H(x)) - \Tt(x,R^\top U_0) \right)\cdot \nabla_x \test(x) \, \mbox d x  
			\end{align}
			where we used that $R \omega = \omega$, $R^\top R = I$, $|R^\top U_0| = |U_0|$ and that $|$det$R|=1$.
            
            Similarly, we get for the right hand side of \eqref{variationApprox2},
			\begin{align}
				-\int_{\omega} \left(\nu_0 - \nuh (|U_0|) \right) U_0 \cdot& \nabla_y \test(y) \mbox d y 
				=  -\int_{\omega} \left(\nu_0 - \nuh (|R^\top U_0|) \right) R^\top U_0 \cdot  \nabla_x \tilde \test(x)  \mbox d x. \label{rhs_coordTrans}
			\end{align}
			On the other hand, consider the transmission problem obtained by replacing $U_0$ in \eqref{variationApprox2} by $R^\top U_0$ and denote its solution by $H(R^\top U_0)$. We note that the left and right hand side are equal to \eqref{lhs_coordTrans} and \eqref{rhs_coordTrans}, respectively. 
			Thus, it follows from the uniqueness of a solution in $\HH(\RN)$ to \eqref{variationApprox2} (where $U_0$ is replaced by $R^\top U_0$) stated in Proposition \ref{propExUnH}, that the solution of the original problem after a coordinate transformation $y = Rx$ equals the solution to the problem in the original coordinates with the vector $U_0$ rotated by application of $R^\top$ in $\HH(\RN)$, i.e.,
			\begin{align}
				H(U_0)( R\,x) = H(R^\top U_0)(x)\label{invariance_H}
			\end{align}
			for all $x \in \RN$. From \eqref{invariance_H}, it follows that
			\begin{align*}
				\nabla_x \left( H(R^\top U_0)(x) \right) = \nabla_x \left(H(U_0)( R\,x) \right) = R^\top \nabla_y H(U_0)(y),
			\end{align*}
			which finishes the proof of statement \ref{statementHInvar}.
			
			\item Statement \ref{statementKInvar} is shown in an analogous way by comparing the left and right hand side of transmission problem \eqref{variationAdjoint_scale1} first after the coordinate transformation $y = R\, x$ and second after replacing $U_0$ and $\Adjz$ by $R^\top U_0$ and $R^\top \Adjz$, respectively. Note that 
			\begin{align*}
				\mathrm{D}\tilde T(x, R^\top U_0) = R^\top \mathrm{D}\tilde T(x, U_0) R \quad \mbox{ and } \quad 
				\mathrm{D} T(R^\top U_0) = R^\top \DT(U_0) R.
			\end{align*}
			
			\item 
			Note that, for $V,W \in \RN$, $T$ defined in \eqref{def_T} and $S$ defined in \eqref{def_S}, we have
			\begin{align*}
				T(R^\top W ) = R^\top T(W), \;
				\DT(R^\top  W) = R^\top \DT(W) R, \;
				S_{R^\top V}(R^\top W) = R^\top S_V(W).
			\end{align*}
			Then, it follows by \eqref{nablaHInvar} and \eqref{nablaKInvar} and the fact that $|\mbox{det}R|= 1$ that
			\begin{align*}
				J_2&(R^\top U_0, R^\top \Adjz ) \\
				=& \int_{\RN \setminus \omega} S_{R^\top U_0}(\nabla_x H(R^\top U_0)(x)) \cdot \left( R^\top \Adjz + \nabla_x K(R^\top U_0, R^\top \Adjz)(x) \right) \mbox dx \\
				=& \int_{\RN\setminus \omega} S_{R^\top U_0}(R^\top \nabla_y H( U_0)(y)) \cdot \left( R^\top \Adjz + R^\top \nabla_y  K( U_0, \Adjz)( y) \right) \mbox dy \\
				=& \int_{\RN\setminus \omega} R^\top S_{U_0}(\nabla_y H(U_0)( y) ) \cdot R^\top (\Adjz + \nabla_y K(U_0, \Adjz)( y)) \mbox d y \\
				=&	J_2(U_0, \Adjz).
			\end{align*}
		\end{enumerate}
	\end{proof} 
	
	By means of properties \ref{statementJ2linearV} and \ref{statementJ2rotInvar} of Lemma \ref{propertiesJ2}, it is possible to efficiently evaluate $J_2$ by
	first precomputing values in an offline stage and then looking them up and interpolating between them during the optimization procedure.
	Let $t := |U_0|$, $s:=|\Adjz|$, $e_i$ the unit vector in $x_i$-direction for $i=1,2$, and $\theta$ and $\varphi$ the angles between $U_0$ and $e_1$ and between $\Adjz$ and $e_1$, respectively, i.e.,
	\begin{align*}
		U_0 = t\, R_{\theta} e_1 \quad \mbox{ and } \quad \Adjz = s \,R_{\varphi} e_1,
	\end{align*}
	where $R_{\alpha}$ denotes the counter-clockwise rotation matrix around an angle $\alpha$, i.e., $$R_{\alpha} = \left[ \begin{array}{cc} \ccos \alpha & -\ssin \alpha \\ \ssin \alpha & \ccos \alpha \end{array} \right].$$ Then, by \eqref{eqnJ2rotInvar} and \eqref{eqnJ2linearV}, we have
	\begin{align}
		J_2(U_0, \Adjz) &= J_2(t \, R_{\theta} e_1, s\, R_{\varphi} e_1)  \nonumber \\
			&= J_2(t \, e_1, s\, R_{\varphi- \theta} e_1) \nonumber \\
			&= J_2(t \, e_1, s \, \mbox{cos}(\varphi-\theta) e_1 + s\, \mbox{sin}(\varphi- \theta) e_2) \nonumber \\
			&= s \,\mbox{cos}(\varphi-\theta) J_2(t \, e_1,  e_1) + s \,\mbox{sin}(\varphi- \theta) J_2(t\, e_1, e_2). \label{J2_precompute}
	\end{align}
	Thus, by precomputing the values of $J_2(t \, e_1,  e_i)$ for $i=1,2$ for a typical range of values of $t = |U_0| = |\nabla u(x_0)| = |\mathbf B(x_0)|$ where $\mathbf B$ denotes the magnetic flux density, the values of the term $J_2$ can be efficiently approximated for any $U_0$ and $\Adjz$ by interpolation, without the need to solve a nonlinear problem for every evaluation.

\subsection{Numerical experiments}

For given $U_0,\Adjz \in \RN $, we compute approximate solutions $H_h$ and $K_h$ to the elements $ \tilde H(U_0) \in H(U_0)$ and $ \tilde K(U_0, \Adjz)\in K(U_0, \Adjz)$ which satisfy the asymptotic behaviors \eqref{asymp_H} and \eqref{asympt_K}, respectively, by the finite element method. We approximate problems \eqref{variationApprox2} and \eqref{variationAdjoint_scale1}, which are defined on the plane $\RN$, by restricting the computational domains to a circular domain of radius $1000$ which is centered at the origin where the inclusion $\omega$ is the unit disk, $\omega = B(0,1)$. 
We use homogeneous Dirichlet boundary conditions for both problems. This approximation is justified by the asymptotic behavior of the solutions $H$ and $K$ derived in \eqref{asymp_H} and \eqref{asympt_K}, respectively; see also the explicit expression of $K$ given in Remarks \ref{rem_MatExplicit_CaseI} and \ref{rem_MatExplicit_CaseII}.
We use piecewise linear finite elements on a triangular mesh.
Figure \ref{fig:HK_caseI} shows the obtained solutions for $H_h \approx H(U_0)$, 
$K_{h,10} \approx K(U_0, (1,0)^\top)$ 
and $K_{h,01} \approx K(U_0, (0,1)^\top)$ with $U_0 = (0.1, 0)^\top$. Note that, for $\Adjz = (\adj_1, \adj_2)^\top$, an approximation to $K(U_0, \Adjz)$ is given by the linear combination $\adj_1 K_{h,10}  + \adj_2 K_{h,01}$. 
The difference between the numerical approximation $K_{h,01}$ and the analytical formula for $K(U_0,(0,1)^\top)$ can be seen from the last picture in Figure \ref{fig:HK_caseI}. We observed that the difference between the exact values and the approximated values with Dirichlet conditions is rather small and decreases when further refining the mesh.
    \begin{figure}
    	\begin{tabular}{cc}
        	\includegraphics[width=.66\textwidth, trim=30mm 0 30mm 0 0, clip=true]{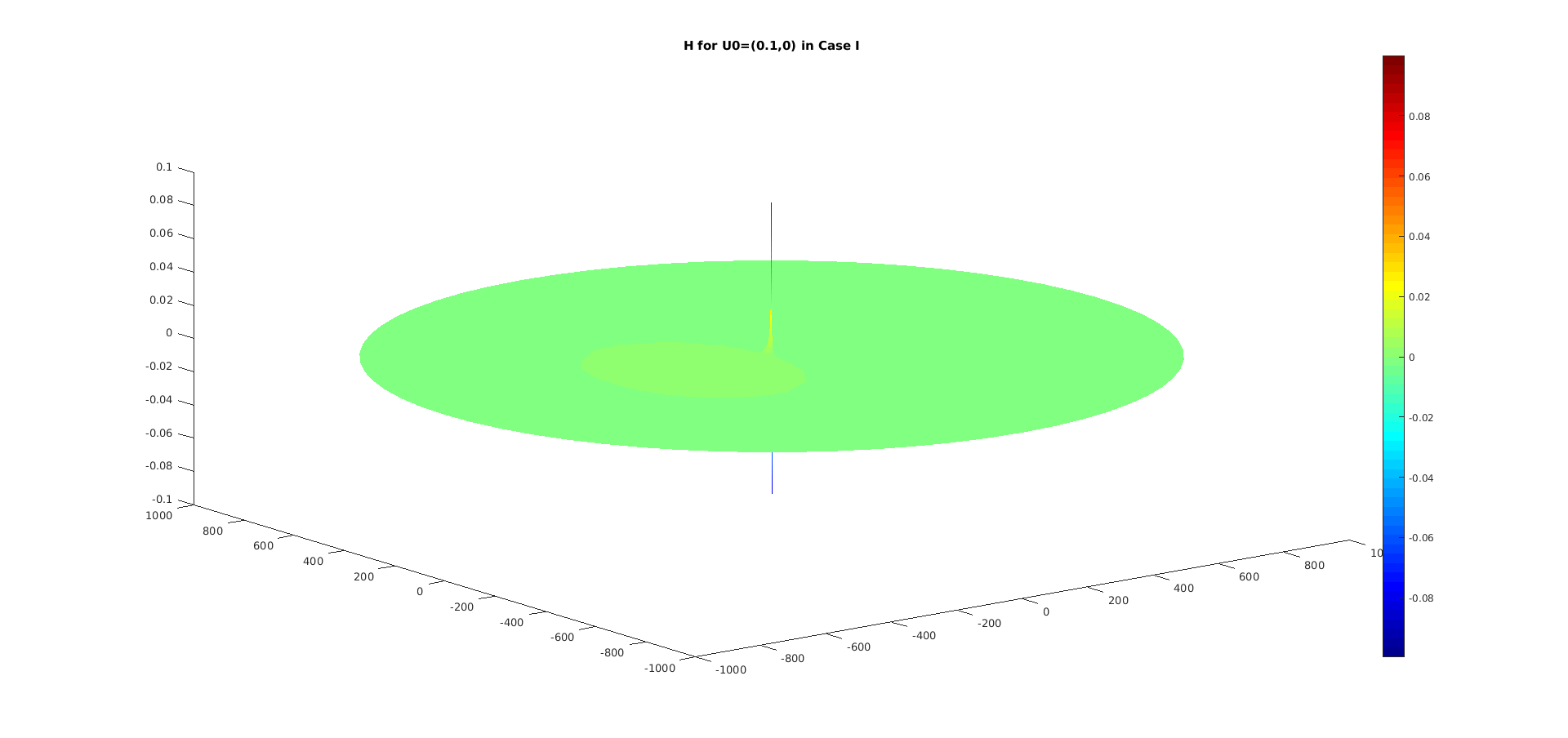}
            &\includegraphics[width=.33\textwidth, trim=20 10 20 10, clip=true]{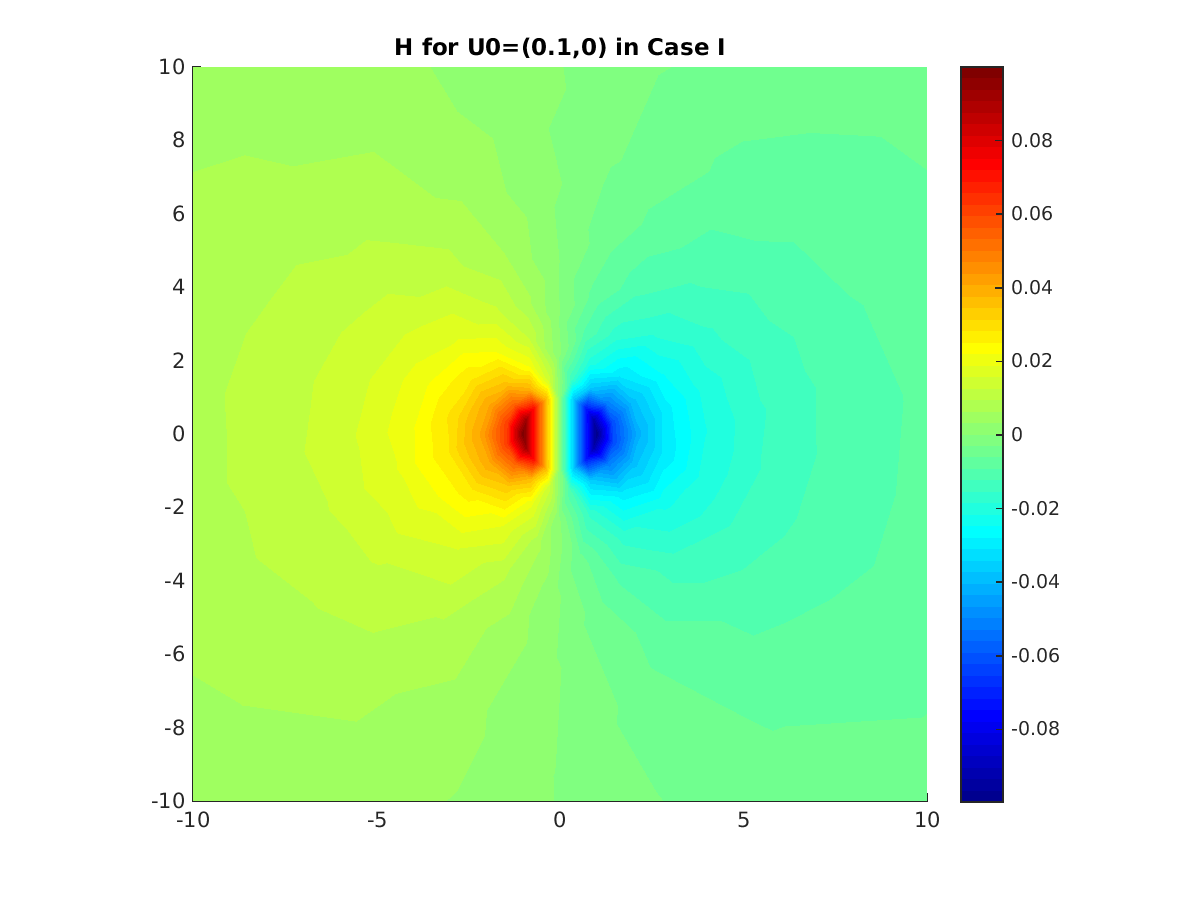} \\
            \begin{minipage}{0.66\textwidth}
            	\begin{tabular}{cc}
                 \includegraphics[width=.475\textwidth, trim=20 10 20 10, clip=true]{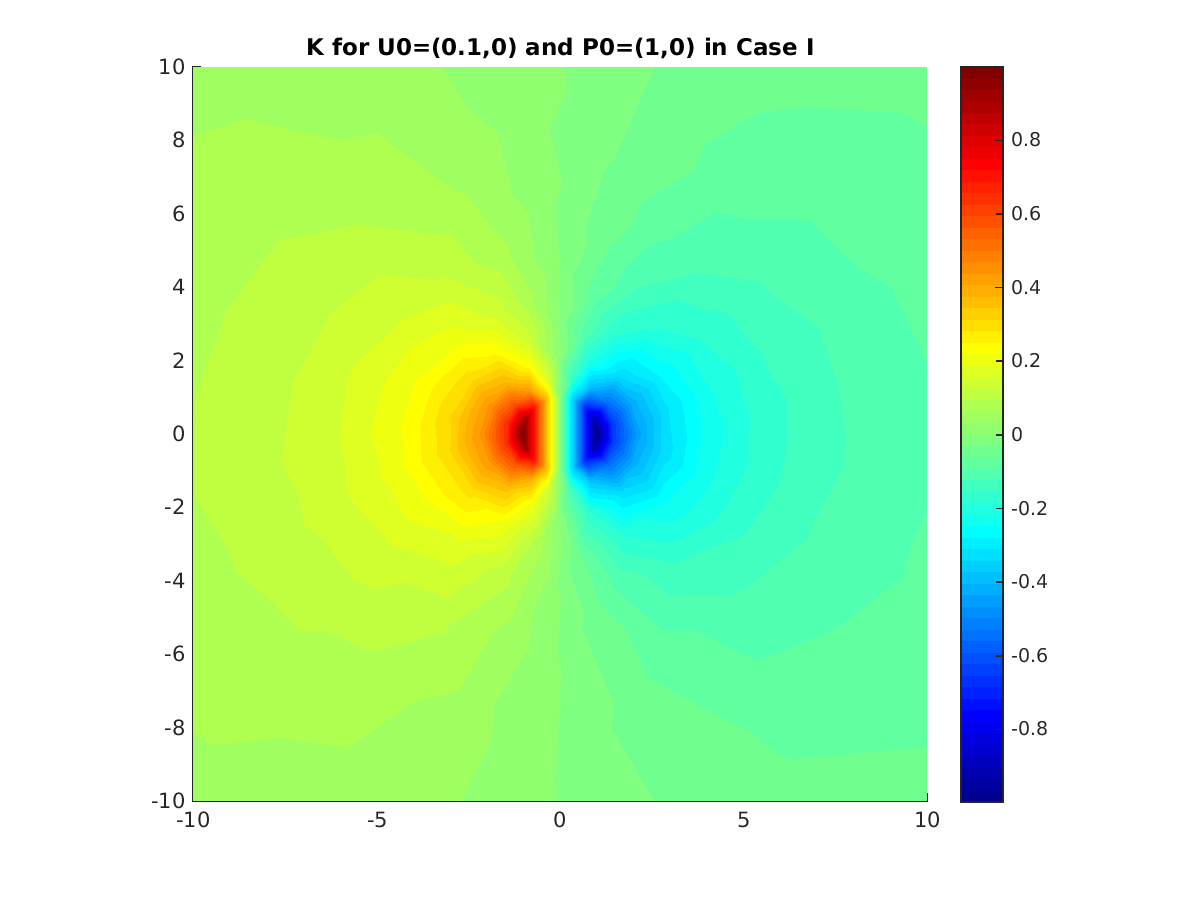} & \includegraphics[width=.475\textwidth, trim=20 10 20 10, clip=true]{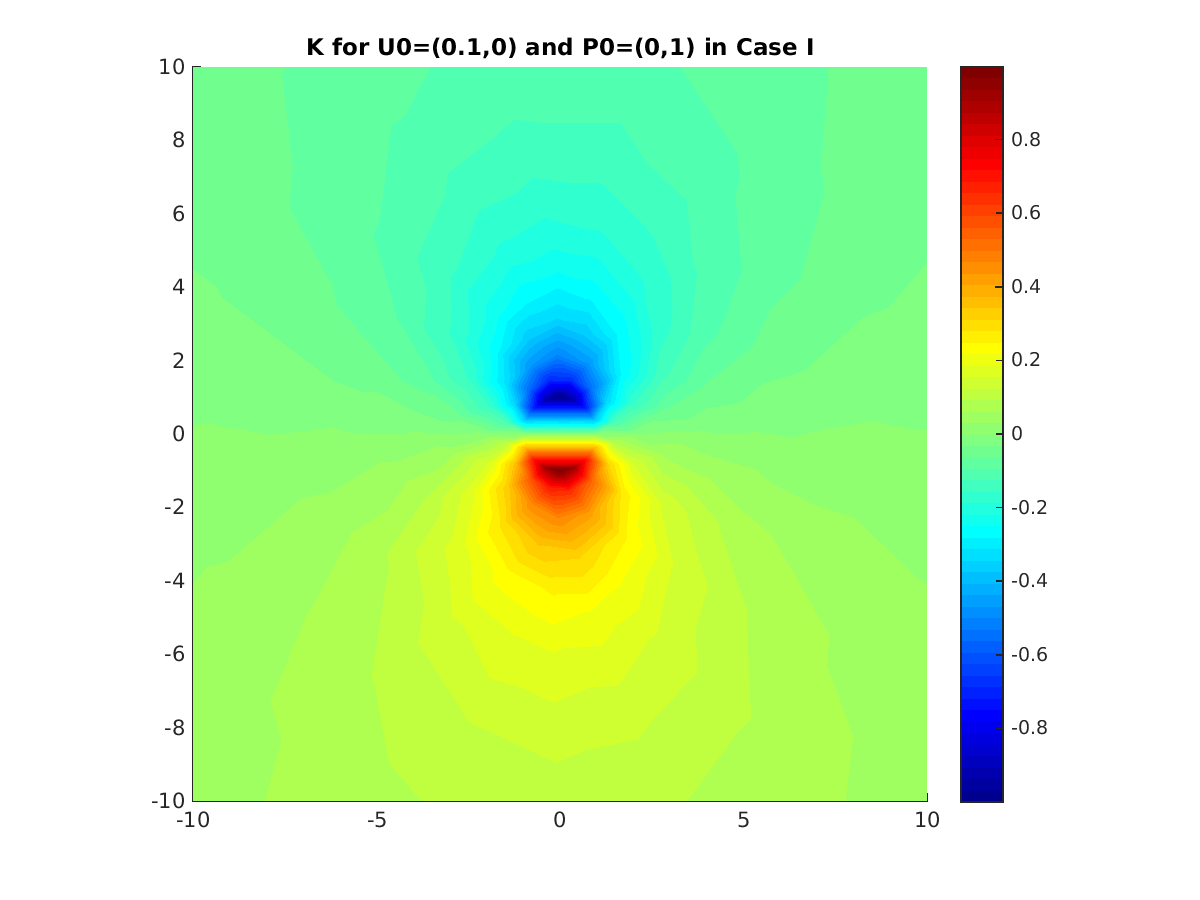}
                \end{tabular}
            \end{minipage}
            &
            \begin{minipage}{0.33\textwidth}
            \includegraphics[width=.95\textwidth, trim=20 10 20 10, clip=true]{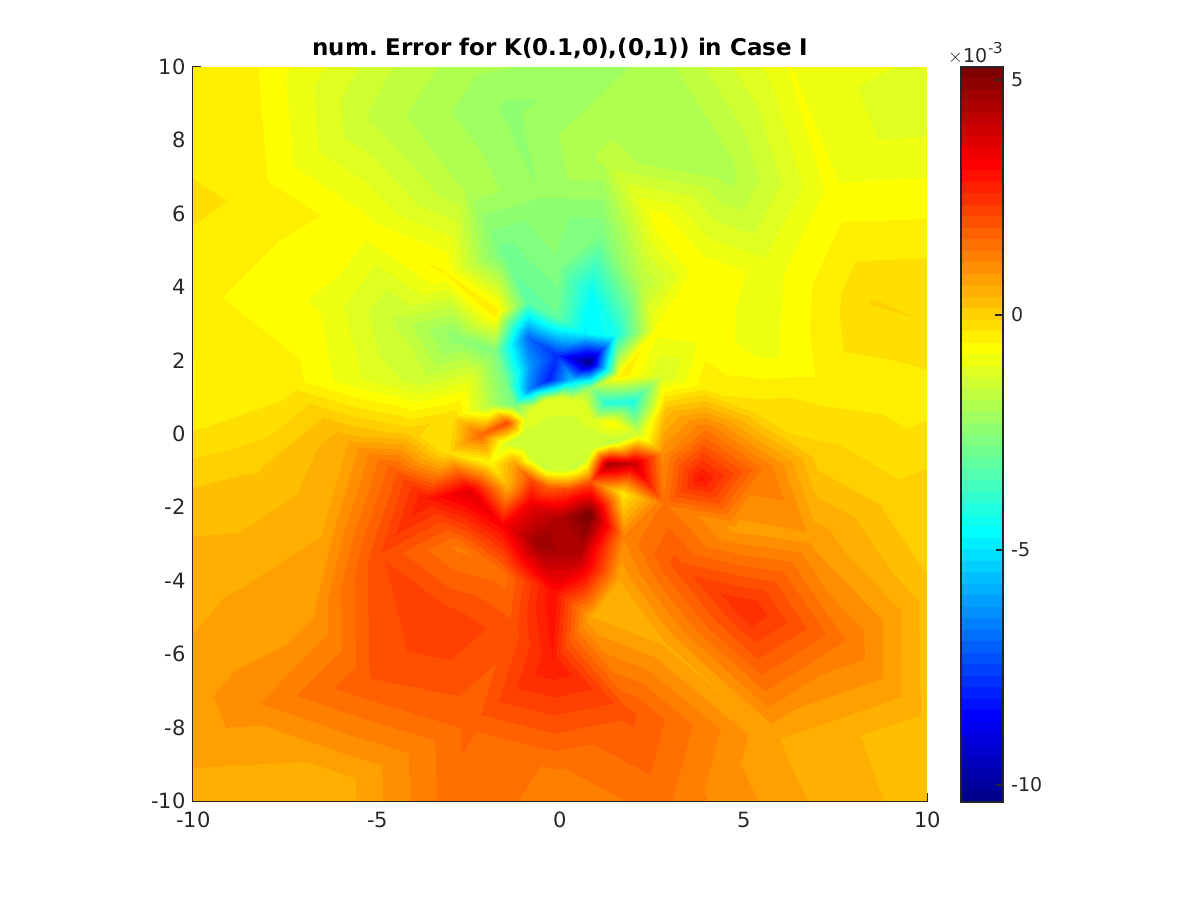}  
            \end{minipage}
        \end{tabular}
        \caption{$H_h$, $K_{h,10}$ and $K_{h,01}$ for $U_0 = (0.1, 0)^\top$ and difference between numerical approximation $K_{h,01}$ and exact $K(U_0, (0,1)^\top)$ in Case I.}
		\label{fig:HK_caseI}	
	\end{figure}
   
Next, we compute and compare the terms $J_1$ and $J_2$ appearing in the topological derivative \eqref{G1_final}.
	The quantities $J_1$ and $J_2$ depend on $U_0=t\, R_{\theta} e_1$ and $\Adjz =s \,R_{\varphi} e_1$ and thus have, in two space dimensions, four degrees of freedom. Both $J_1$ and $J_2$ are linear in the second argument $\Adjz$, thus we can neglect $s = |\Adjz|$, as a scaling of $\Adjz$ will result in the same scaling of $J_1$ and $J_2$.
	Furthermore, in terms of the angles $\theta, \varphi$, both $J_1$ and $J_2$ only depend on the difference $\varphi - \theta$. For $J_2$ this can be seen from \eqref{J2_precompute} and for $J_1$, this can be seen from \eqref{J1_polMat} and \eqref{TDMat_CaseI}. Thus, we can visualize $J_1$ and $J_2$ in dependence of two degrees of freedom, $|U_0|$ and $\varphi - \theta$. Figures \ref{fig:J1J2_caseI} and \ref{fig:J1J2_caseII} show $J_1$ and $J_2$ in Case I and Case II in dependence on these two degrees of freedom. Note that they are of a similar magnitude for certain values of $|U_0| = |\mathbf B(x_0)|$.
   
\begin{figure} 
		\begin{minipage}{7cm}
			\begin{tabular}{c}    
				\includegraphics[width=\textwidth, trim=0 0mm 0 0mm, clip=true]{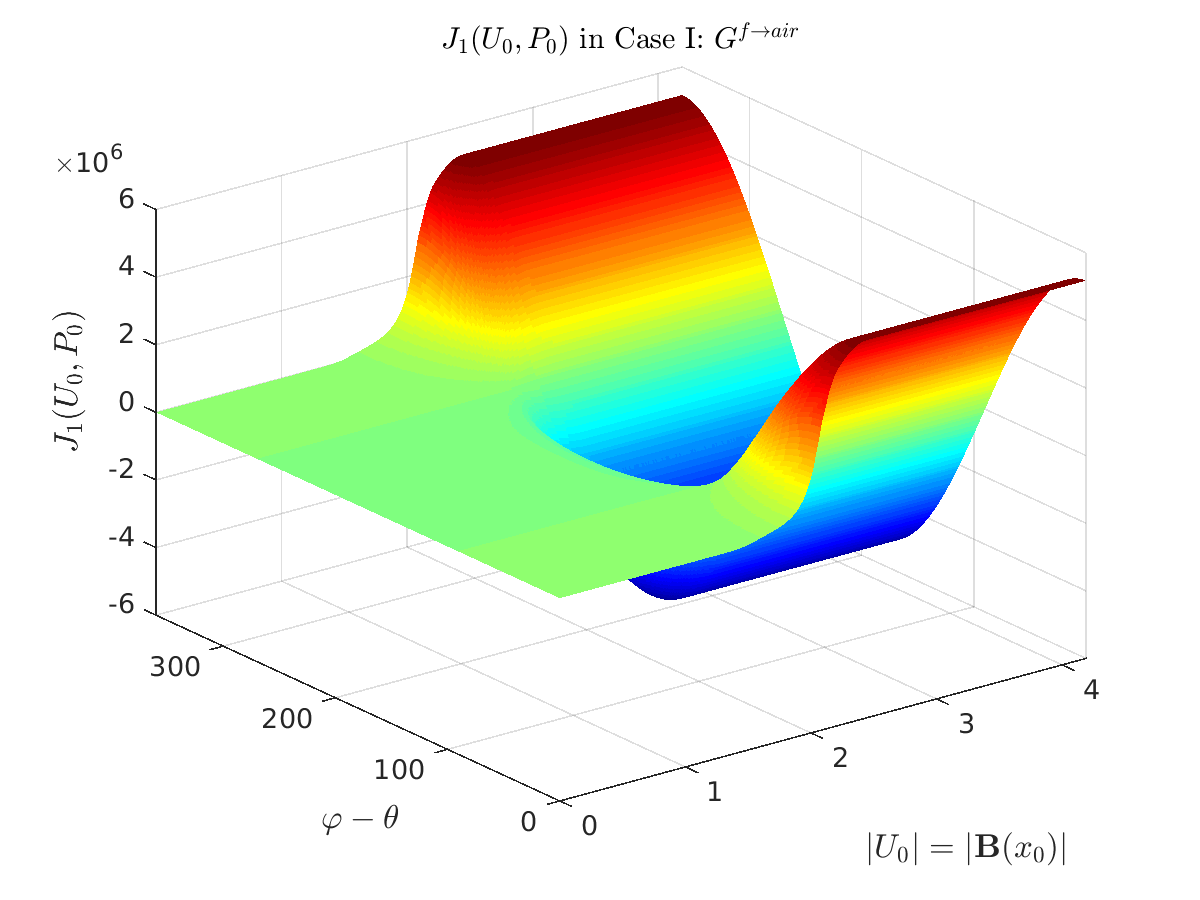}
			\end{tabular}
		\end{minipage}
		\begin{minipage}{7cm}
			\begin{tabular}{c}    
				\includegraphics[width=\textwidth, trim=0 0mm 0 0mm, clip=true]{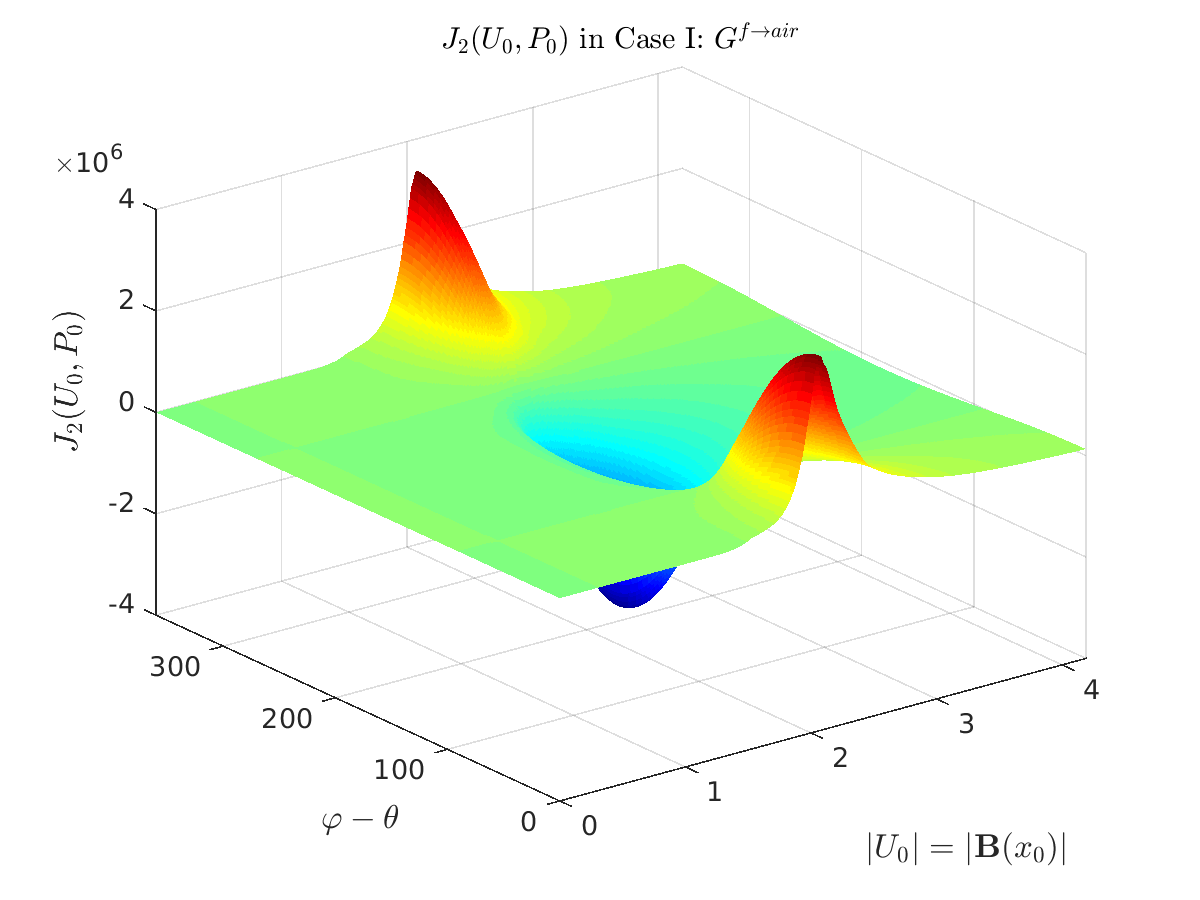}
			\end{tabular}
		\end{minipage}
 		\caption{$J_1(U_0,\Adjz)$ and $J_2(U_0,\Adjz)$ in Case I.}
		\label{fig:J1J2_caseI}	
	\end{figure}
	
	\begin{figure} 
		\begin{minipage}{7cm}
			\begin{tabular}{c}    
				\includegraphics[width=\textwidth, trim=0 0mm 0 0mm, clip=true]{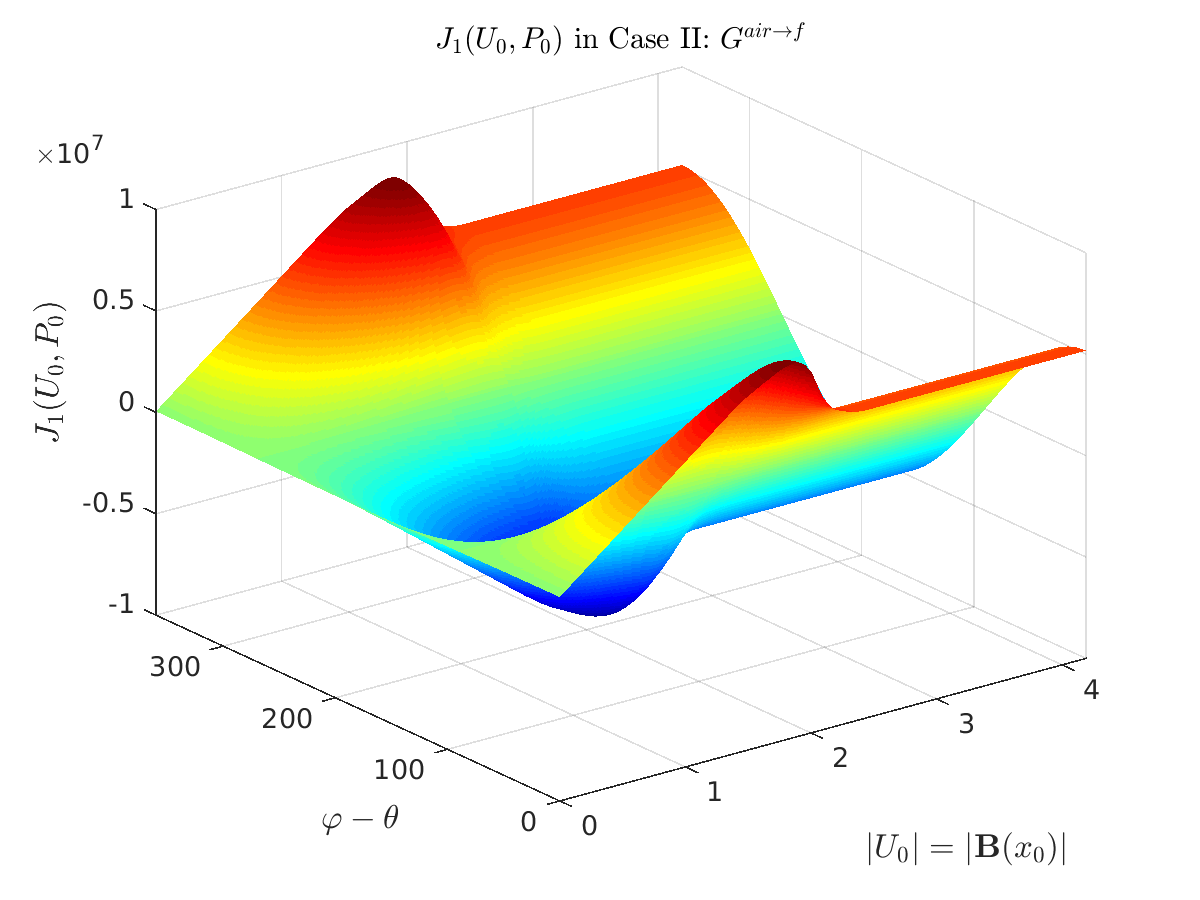}
			\end{tabular}
		\end{minipage}
		\begin{minipage}{7cm}
			\begin{tabular}{c}    
				\includegraphics[width=\textwidth, trim=0 0mm 0 0mm, clip=true]{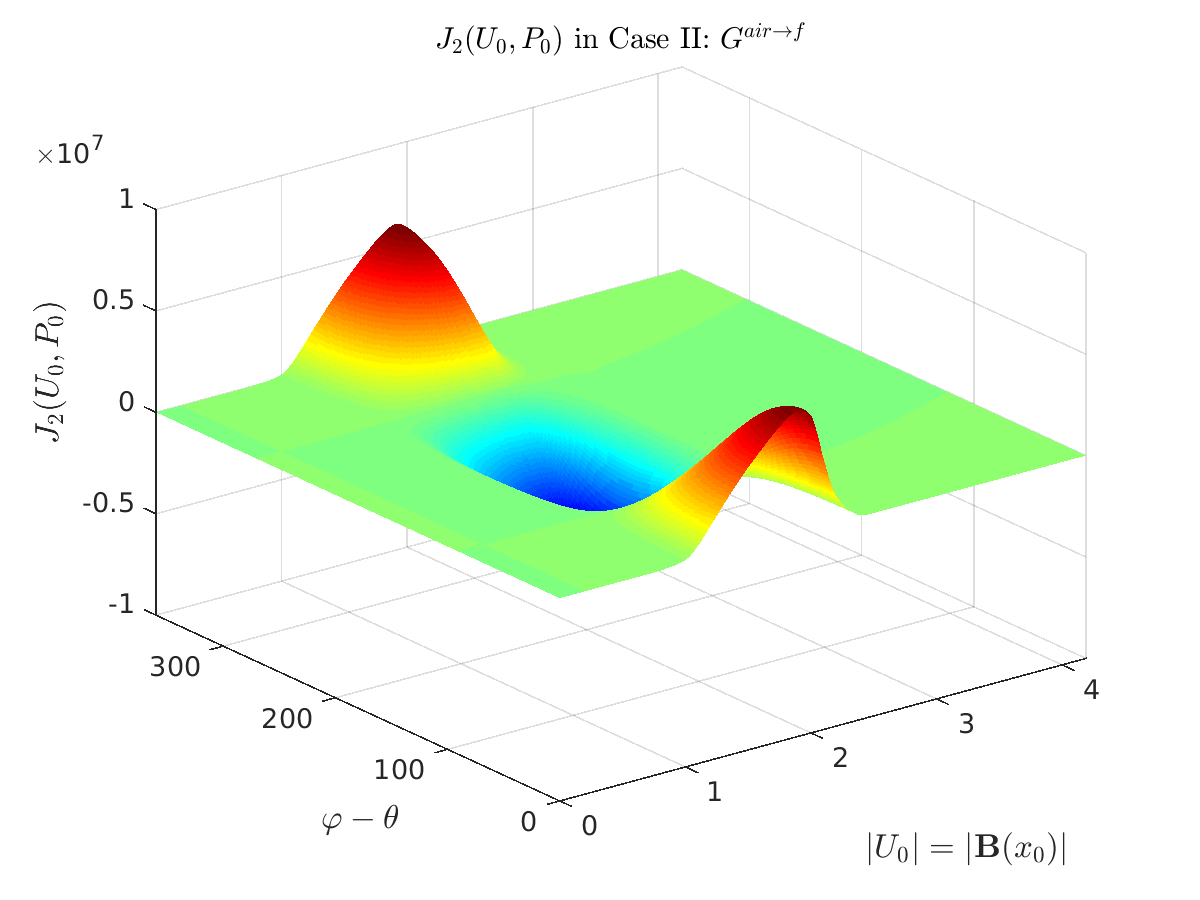}
			\end{tabular}
		\end{minipage}
        \caption{$J_1(U_0,\Adjz)$ and $J_2(U_0,\Adjz)$ in Case II.}
		\label{fig:J1J2_caseII}	
	\end{figure}

	
\section{Application to topology optimization of electric motor} \label{sec:applicTDNL}
Finally, we employ the topological derivative derived in \eqref{G1_final} and \eqref{G2_final} to the model design optimization problem introduced in \eqref{minJ}--\eqref{PDEconstraintUnperturbed} in Section \ref{sec:ProblemDescription}.

\subsection{Objective function}
The electric motor depicted in Figure \ref{fig:elMotor} consists of a fixed outer part (the stator) and a rotating inner part (the rotor) which are separated by a thin air gap. We introduce an objective function whose minimization corresponds to finding a design where the rotor rotates smoothly with little mechanical vibration and noise.
For that purpose, we consider the radial component of the magnetic flux density $\mathbf B$ on a circular curve $\Gamma_0 \subset \Omega_g$ in the air gap which is generated only by the permanent magnets, i.e., $J_z=0$. The goal is to find the optimal distribution of ferromagnetic material in the design domains such that this radial component is as close as possible to a given smooth curve in an $L^2$ sense, see Fig. \ref{fig_optiResults} (right) for the initial and desired curve as well as the curve for the final design. Thus, noting that $\mathbf B = \mathbf B(u) = \mbox{curl}((0,0,u)^\top)$, the objective function reads
\begin{align*}
	\mathcal J(u) = \int_{\Gamma_0} |\mathbf B(u)\cdot n - B_d|^2 \mbox ds = \int_{\Gamma_0} |\nabla u \cdot \tau - B_d|^2 \mbox ds,
\end{align*}
where $n$ and $\tau$ denote the outer unit normal vector and the tangential vector, respectively. Note that $\mathcal J$ is well-defined for $u$ the solution to \eqref{PDEconstraintUnperturbed}, which is smooth in the air gap $\Omega_g \supset \Gamma_0$.

\subsection{Algorithm}
We apply the level set algorithm introduced in \cite{AmstutzAndrae:2006a}, which is based on the topological derivative. This is in contrast to the level set method for shape optimization where the evolution of the interface is usually guided by shape sensitivity information and generally lacks a nucleation mechanism.
	We give a short overview of the algorithm and refer to the references \cite{AmstutzAndrae:2006a, Amstutz:2011a} for a more detailed description.
		
	Recall the notation of Section \ref{sec:ProblemDescription}. In particular, recall that the variable set $\Omega$ was defined as that subset of $\Omega^d$ which is currently occupied with ferromagnetic material. The current design is represented by means of a level set function $\psi : \Omega^d \rightarrow \mathbb R$ which is positive in the ferromagnetic subdomain and negative in the air subdomain. The zero level set of $\psi$ represents the interface between the two subdomains. Thus, we have
	\begin{align}
		\psi(x) > 0 \Leftrightarrow x \in \Omega, \qquad \psi(x) < 0 \Leftrightarrow x \in \Omega^d \setminus \overline{\Omega}, \qquad \psi(x) =0 \Leftrightarrow x \in \partial \Omega. \label{def_psi}
	\end{align}
	The evolution of this level set function is guided by the generalized topological derivative, which, for a given design represented by $\psi$, is defined in the following way:
	\begin{align} \label{def_Gtilde}
		\tilde G_{\psi}(x) := \begin{cases}
								G^{f\rightarrow air}(x), & x \in \Omega, \\
								-G^{air\rightarrow f}(x), & x \in \Omega^d \setminus \overline{\Omega}.
		                     \end{cases}
	\end{align}
	Note that the topological derivative is only defined in the interior of $\Omega$ and in the interior of $\Omega^d \setminus \overline{\Omega}$, but not on the interface.
	The algorithm is based on the following observation: If for all $x \in \Omega \cup (\Omega^d \setminus \overline{\Omega})$, it holds
		\begin{align}
			\psi(x) = c \, \tilde G_{\psi}(x) \label{def_optiCond}
		\end{align}
		for a constant $c>0$, then a small topological perturbation (introduction of an inclusion of air inside ferromagnetic material or vice versa) will always increase the objective function.

	This observation motivates the following algorithm:
	\begin{algorithm} \label{algoLevelSetTD}
		Initialization: Choose $\psi_0$ with $\| \psi_0 \| = 1$, compute $\mathcal J(\psi_0)$ and $\tilde G_{\psi_0}$ and set $k=0$.
		\begin{itemize}
			\item[(i)] Set $\theta_k = \mbox{arccos}\,\left(\psi_k, \frac{\tilde G_{\psi_k}}{\|\tilde G_{\psi_k}\|} \right)$ and
				\begin{align*}
						\textcolor{black}{\psi_{k+1}} = \frac{1}{\mbox{sin} \theta_k} \left( \mbox{sin}((1-\kappa_k)\theta_k) \, \textcolor{black}{\psi_k} + \mbox{sin}(\kappa_k\theta_k) \textcolor{black}{\frac{\tilde G_{\psi_k}}{\|\tilde G_{\psi_k}\|}} \right),
					\end{align*}
					where $\kappa_k = \mbox{max} \lbrace 1, 1/2, 1/4 \dots \rbrace$ such that $\mathcal J(\psi_{k+1})<\mathcal J(\psi_k)$
			\item[(ii)] Compute $\tilde G_{\psi_{k+1}}$ according to \eqref{def_Gtilde}
			\item[(iii)] If $\tilde G_{\psi_{k+1}} = \psi_{k+1}$ then stop, else set $k \leftarrow k+1$ and go to (ii)
		\end{itemize} 
	\end{algorithm}
	Here, we identified the domain $\Omega$ with the level set function $\psi$ representing $\Omega$ and wrote $\mathcal J(\psi)$ instead of $\mathcal J(\Omega)$.
	Note that each evaluation of the objective function $\mathcal J$ requires the solution of the state equation \eqref{PDEconstraintUnperturbed} and each evaluation of the generalized topological derivative $\tilde G_{\psi}$ additionally requires the adjoint state $\adj$, i.e., the solution to \eqref{adjoint_eps0}. 
    Here, the norms and the inner product are to be understood in the space $L^2(\Omega^d)$. More details on the algorithm and its implementation can be found in \cite{AmstutzAndrae:2006a, Amstutz:2011a}.

\subsection{Numerical results}
Figure \ref{fig_optiResults} shows the initial geometry of one out of eight design subdomains (left) and the the radial component of the magnetic flux density for the initial and final design (right). Figure \ref{fig_optiResults2} shows the final design obtained after 375 iterations of Algorithm \ref{algoLevelSetTD} together with the magnetic flux density caused by the permanent magnets. The objective value was reduced from $7.6011*10^{-4}$ to $2.0412*10^{-4}$. 
In order to preserve symmetry of the designs, we chose a slightly more conservative step size and started with $\kappa_k=1/10$ rather than $\kappa_k=1$ in Algorithm \ref{algoLevelSetTD}.

\begin{figure}[ht]
	\begin{tabular}{cc}
    	\includegraphics[trim = 10 200 10 60mm, clip, width = .5\textwidth]{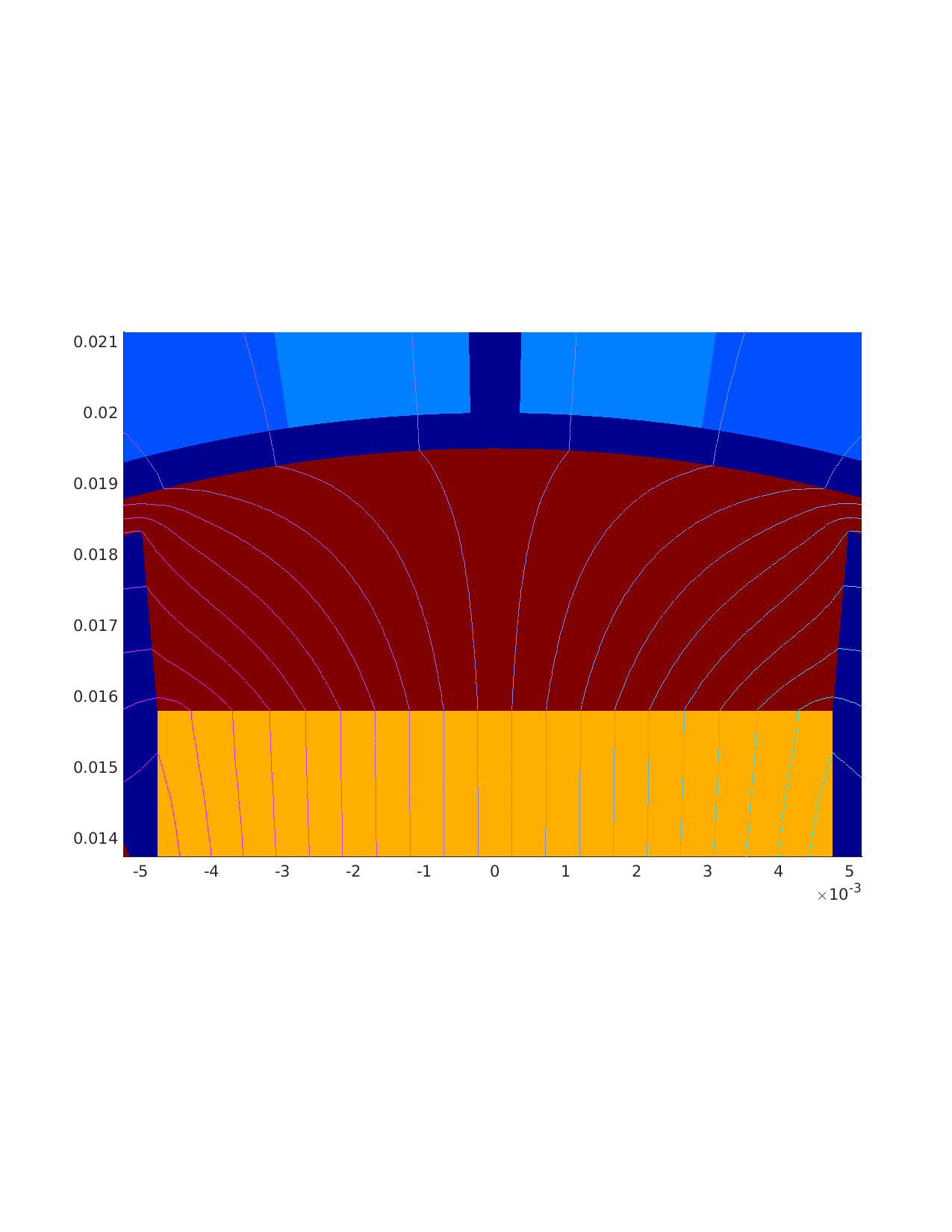}&
          \includegraphics[trim = 10 10 10 10mm, clip, width = .5\textwidth]{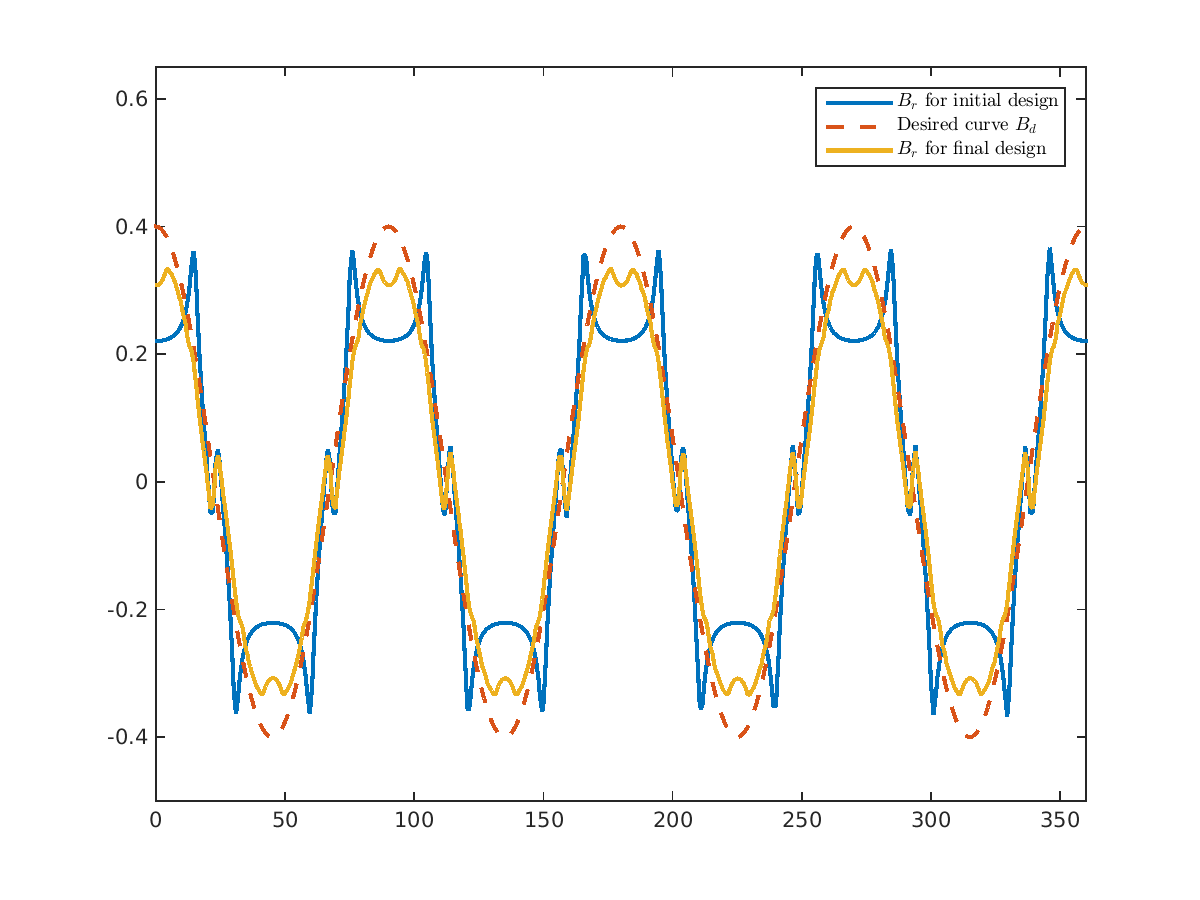}
    \end{tabular}
    \caption{Left: Initial design. Right: Radial component of magnetic flux density in air gap for initial and final design and desired curve.}
    \label{fig_optiResults}
\end{figure}

\begin{figure}[ht]
	\begin{tabular}{cc}
    	\includegraphics[trim = 10 200 10 60mm, clip, width = .5\textwidth]{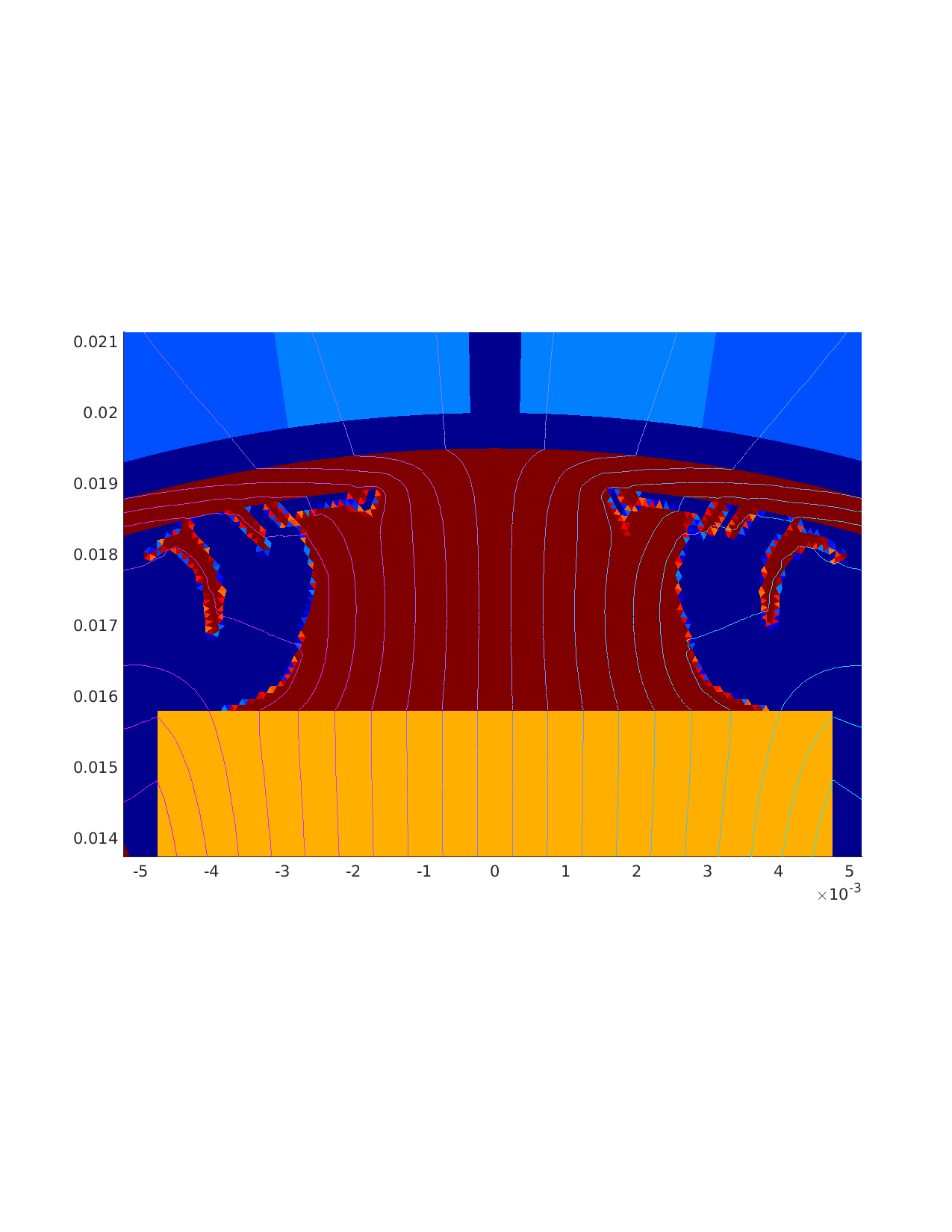}&
        \includegraphics[trim = 10 200 10 60mm, clip, width = .5\textwidth]{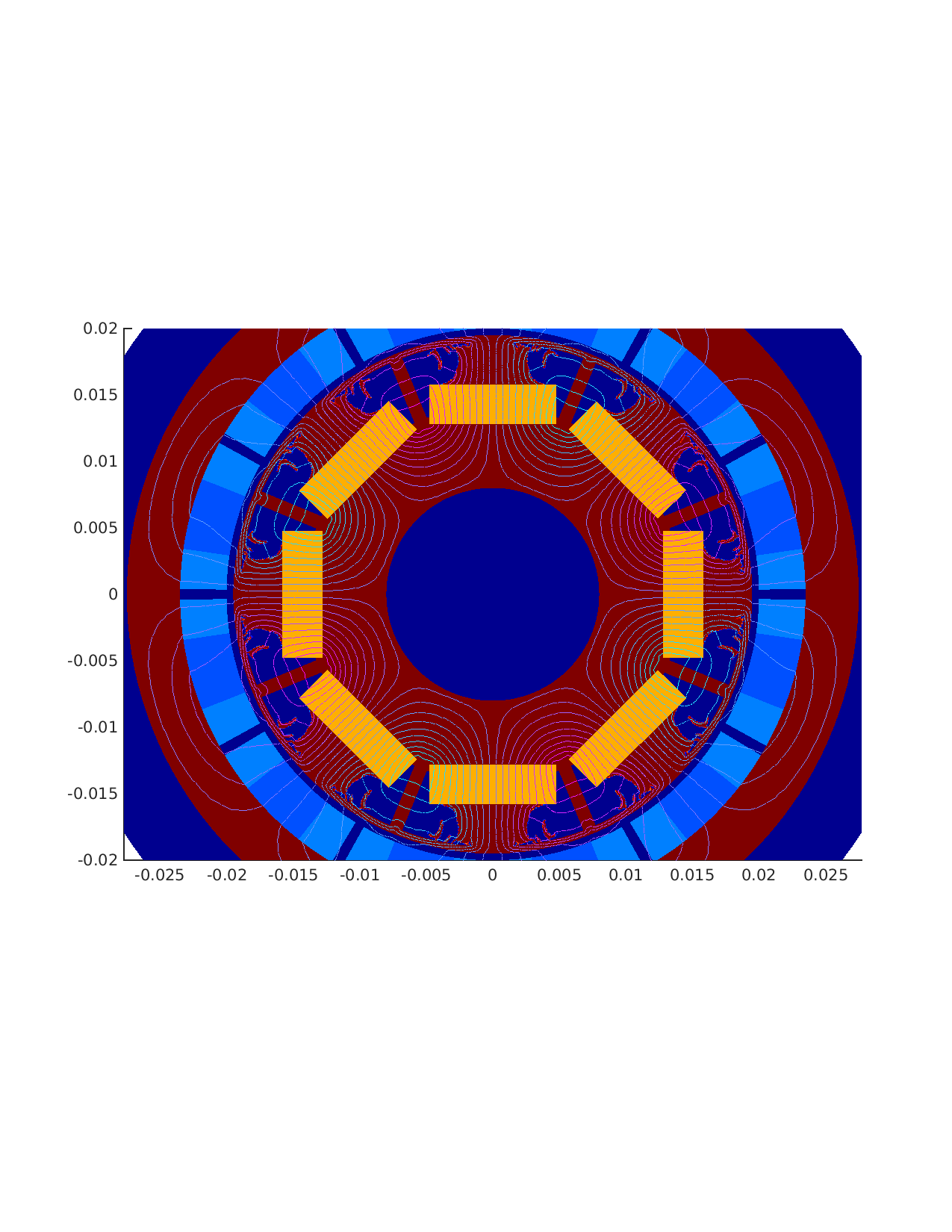}
      
    \end{tabular}
    \caption{Final design after 375 iterations when applying algorithm \cite{AmstutzAndrae:2006a}, which uses the topological derivatives \eqref{G1_final} and \eqref{G2_final}, to problem \eqref{minJ}--\eqref{PDEconstraintUnperturbed}.}
    \label{fig_optiResults2}
\end{figure}

We mention that, when the second term $J_2$ in the topological derivative is dropped in the optimization algorithm, the objective function cannot be reduced as much and the algorithm terminates prematurely.

\section*{Conclusion}
We derived the topological derivative for an optimization problem from electromagnetics which is constrained by the quasilinear partial differential equation of two-dimensional magnetostatics. We proved the formula for the topological derivative in the case where linear material (air) is introduced inside nonlinear (ferromagnetic) material and stated the corresponding formula for the reverse scenario. The key ingredient in the first case was to show a sufficiently fast decay of the variation of the direct state at scale 1 as $|x|\rightarrow \infty$. The topological derivative consists of two terms. The first term resembles the formula for the case of a linear state equation and includes a polarization matrix, which we computed explicitly. The second term is hard to evaluate in practice. We presented a way to efficiently compute the term approximately which can be used in a topology optimization algorithm. Finally we applied a level set algorithm which is based on the topological derivative to the optimization of an electric motor.

	\bibliographystyle{abbrv}

\begin{thebibliography}{10}

\bibitem{Ammari2008}
{\sc H.~Ammari}, {\em An introduction to mathematics of emerging biomedical
  imaging}, vol.~62 of Math\'{e}matiques \& Applications (Berlin) [Mathematics
  \& Applications], Springer, Berlin, 2008.

\bibitem{AmmariKang2007}
{\sc H.~Ammari and H.~Kang}, {\em {Polarization and Moment Tensors}}, Springer,
  New York, 2007.

\bibitem{AmroucheGiraultGiroire1994}
{\sc C.~Amrouche, V.~Girault, and J.~Giroire}, {\em {Weighted {S}obolev spaces
  for {L}aplace's equation in $\textbf{R}^n$}}, J. Math. Pures Appl., 73
  (1994), pp.~579--606.

\bibitem{Amstutz2005}
{\sc S.~Amstutz}, {\em {The topological asymptotic for the {N}avier-{S}tokes
  equations}}, ESAIM: COCV, 11 (2005), pp.~401--425.

\bibitem{Amstutz2006}
\leavevmode\vrule height 2pt depth -1.6pt width 23pt, {\em {Sensitivity
  analysis with respect to a local perturbation of the material property}},
  Asymptotic analysis, 49 (2006).

\bibitem{Amstutz2006b}
\leavevmode\vrule height 2pt depth -1.6pt width 23pt, {\em {Topological
  sensitivity analysis for some nonlinear {PDE} systems}}, Journal de
  Math{\'e}matiques Pures et Appliqu{\'e}es, 85 (2006), pp.~540--557.

\bibitem{Amstutz:2011a}
\leavevmode\vrule height 2pt depth -1.6pt width 23pt, {\em {Analysis of a level
  set method for topology optimization}}, Optimization Methods and Software -
  Advances in Shape an Topology Optimization: Theory, Numerics and New
  Application Areas, 26 (2011), pp.~555--573.

\bibitem{AmstutzAndrae:2006a}
{\sc S.~Amstutz and H.~Andr{\"a}}, {\em {A new algorithm for topology
  optimization using a level-set method}}, Journal of Computational Physics,
  216 (2006), pp.~573--588.

\bibitem{AmstutzBonnafe2015}
{\sc S.~Amstutz and A.~Bonnaf{\'e}}, {\em {Topological derivatives for a class
  of quasilinear elliptic equations}}, Journal de Math{\'e}matiques Pures et
  Appliqu{\'e}es, 107 (2017), pp.~367--408.

\bibitem{AmstutzNovotny2010}
{\sc S.~Amstutz and A.~A. Novotny}, {\em {Topological optimization of
  structures subject to Von {M}ises stress constraints}}, Structural and
  Multidisciplinary Optimization, 41 (2010), pp.~407--420.

\bibitem{AmstutzNovotny2011}
\leavevmode\vrule height 2pt depth -1.6pt width 23pt, {\em {Topological
  asymptotic analysis of the {K}irchhoff plate bending problem}}, ESAIM: COCV,
  17 (2011), pp.~705--721.

\bibitem{AmstutzNovotnyVangoethem2014}
{\sc S.~Amstutz, A.~A. Novotny, and N.~V. Goethem}, {\em {Topological
  sensitivity analysis for elliptic differential operators of order 2m}},
  Journal of Differential Equations, 256 (2014), pp.~1735--1770.

\bibitem{Bonnafe2013}
{\sc A.~Bonnaf{\'e}}, {\em {D{\'e}veloppements asymptotiques topologiques pour
  une classe d'{\'e}quations elliptiques quasi-lin{\'e}aires. Estimations et
  d{\'e}veloppements asymptotiques de p-capacit{\'e}s de condensateur. Le cas
  anisotrope du segment}}, PhD thesis, Universit{\'e} de Toulouse, France,
  2013.

\bibitem{CanelasNovotnyRoche2014}
{\sc A.~Canelas, A.~A. Novotny, and J.~R. Roche}, {\em {Topology design of
  inductors in electromagnetic casting using level-sets and second order
  topological derivatives}}, Structural and Multidisciplinary Optimization, 50
  (2014), pp.~1151--1163.

\bibitem{CeaGarreauGuillaumeMasmoudi2000}
{\sc J.~C{\'e}a, S.~Garreau, P.~Guillaume, and M.~Masmoudi}, {\em {The shape
  and topological optimizations connection}}, Computer Methods in Applied
  Mechanics and Engineering, 188 (2000), pp.~713--726.

\bibitem{ChaabaneMasmoudiMeftahi2013}
{\sc S.~Chaabane, M.~Masmoudi, and H.~Meftahi}, {\em {Topological and shape
  gradient strategy for solving geometrical inverse problems}}, Journal of
  Mathematical Analysis and Applications, 400 (2013), pp.~724--742.

\bibitem{EschenauerKobelevSchumacher1994}
{\sc H.~A. Eschenauer, V.~V. Kobelev, and A.~Schumacher}, {\em {Bubble method
  for topology and shape optimization of structures}}, Structural optimization,
  8 (1994), pp.~42--51.

\bibitem{FeijooNovotnyPadraTaroco2002}
{\sc R.~A. Feij{\'o}o, A.~A. Novotny, C.~Padra, and E.~O. Taroco}, {\em {The
  topological-shape sensitivity analysis and its applications in optimal
  design}}, in {Mec{\'a}nica Computacional Vol XXI}, 2002.

\bibitem{Gangl2017}
{\sc P.~Gangl}, {\em {Sensitivity-Based Topology and Shape Optimization with
  Application to Electrical Machines}}, PhD thesis, Johannes Kepler University
  Linz, 2017.

\bibitem{GanglLanger2014}
{\sc P.~Gangl and U.~Langer}, {\em {Topology optimization of electric machines
  based on topological sensitivity analysis}}, Computing and Visualization in
  Science,  (2012).

\bibitem{GanglLangerLaurainMeftahiSturm2015}
{\sc P.~Gangl, U.~Langer, A.~Laurain, H.~Meftahi, and K.~Sturm}, {\em {Shape
  Optimization of an Electric Motor Subject to Nonlinear Magnetostatics}}, SIAM
  Journal on Scientific Computing, 37 (2015), pp.~B1002--B1025.

\bibitem{GarreauGuillaumeMasmoudi2001}
{\sc S.~Garreau, P.~Guillaume, and M.~Masmoudi}, {\em {The Topological
  Asymptotic for PDE Systems: The Elasticity Case}}, SIAM Journal on Control
  and Optimization, 39 (2001), pp.~1756--1778.

\bibitem{Hackl2006}
{\sc B.~Hackl}, {\em {Geometry Variations, Level Sets and Phase-field Methods
  for Perimeter Regularized Geometric Inverse Problems}}, PhD thesis, Johannes
  Kepler University Linz, 2006.

\bibitem{Heise1994}
{\sc B.~Heise}, {\em {Analysis of a fully discrete finite element method for a
  nonlinear magnetic field problem}}, SIAM J. Numer. Anal., 31 (1994),
  pp.~745--759.

\bibitem{Hintermueller2007}
{\sc M.~Hinterm{\"u}ller}, {\em {Real-Time PDE-Constrained Optimization}},
  SIAM, 2007, ch.~13. A Combined Shape-Newton and Topology Optimization
  Technique in Real-Time Image Segmentation, pp.~253--275.

\bibitem{HintermuellerLaurainNovotny2012}
{\sc M.~Hinterm{\"u}ller, A.~Laurain, and A.~A. Novotny}, {\em {Second-order
  topological expansion for electrical impedance tomography}}, Advances in
  Computational Mathematics, 36 (2012), pp.~235--265.

\bibitem{MR2541192}
{\sc M.~Iguernane, S.~A. Nazarov, J.-R. Roche, J.~Soko{\l}owski, and K.~Szulc},
  {\em {Topological derivatives for semilinear elliptic equations}}, Int. J.
  Appl. Math. Comput. Sci., 19 (2009), pp.~191--205.

\bibitem{LarrabideFeijooNovotnyTaroco2008}
{\sc I.~Larrabide, R.~Feij{\'o}o, A.~Novotny, and E.~Taroco}, {\em {Topological
  derivative: A tool for image processing}}, Computers \& Structures, 86
  (2008), pp.~1386 -- 1403.
\newblock Structural Optimization.

\bibitem{Masmoudi1998}
{\sc M.~Masmoudi}, {\em {Topological Optimization with Dirichlet Condition}},
  in {Picof}, M.~J. e.~J. Jaffr{\'e}, ed., 1998, pp.~121--127.

\bibitem{MasmoudiPommierSamet2005}
{\sc M.~Masmoudi, J.~Pommier, and B.~Samet}, {\em {The topological asymptotic
  expansion for the {M}axwell equations and some applications}}, Inverse
  Problems, 31 (2005), pp.~545--564.

\bibitem{NovotnyFeijooTarocoMasmoudiPadra2005}
{\sc A.~A. Novotny, R.~A. Feij{\'o}o, E.~Taroco, M.~Masmoudi, and C.~Padra},
  {\em {Topological Sensitivity Analysis for a Nonlinear Case: the p-{P}oisson
  problem}}, in {6th World Congress on Structural and Multidisciplinary
  Optimization}, 2005.

\bibitem{NovotnySokolowski2013}
{\sc A.~A. Novotny and J.~Soko{\l}owski}, {\em {Topological derivatives in
  shape optimization}}, {Interaction of Mechanics and Mathematics}, Springer,
  Heidelberg, 2013.

\bibitem{Pechstein2004}
{\sc C.~Pechstein}, {\em {Multigrid-{N}ewton-Methods For Nonlinear
  Magnetostatic Problems}}, Master's thesis, Johannes Kepler University Linz,
  2004.

\bibitem{PechsteinJuettler2006}
{\sc C.~Pechstein and B.~J{\"u}ttler}, {\em {Monotonicity-preserving
  interproximation of {B-H}-curves}}, J. Comp. App. Math., 196 (2006),
  pp.~45--57.

\bibitem{Schumacher1995}
{\sc A.~Schumacher}, {\em {Topologieoptimierung von Bauteilstrukturen unter
  Verwendung von Lochpositionierungskriterien}}, PhD thesis, Univ. Siegen,
  1995.

\bibitem{SokolowskiZochowski1999}
{\sc J.~Soko{\l}owski and A.~Zochowski}, {\em {On the Topological Derivative in
  Shape Optimization}}, SIAM Journal on Control and Optimization, 37 (1999),
  pp.~1251--1272.

\bibitem{Zeidler1990}
{\sc E.~Zeidler}, {\em {Nonlinear Functional Analysis and its Applications
  {II}/{B}: Nonlinear Monotone Operators}}, Springer, New York, 1990.

\end{thebibliography}

\end{document}